\documentclass[a4paper,12pt]{article}

\usepackage{float}
\usepackage{amsmath}
\usepackage{amsthm}
\usepackage{amsfonts}
\usepackage{amssymb}

\usepackage[T1]{fontenc}
\usepackage{geometry}
\usepackage{graphicx}
\usepackage{hyperref}
\usepackage{icomma}
\usepackage[latin1]{inputenc}
\usepackage{latexsym}
\usepackage[numbers]{natbib}
\usepackage{textcomp}
\usepackage[all]{xy}
\usepackage{cancel}
\usepackage{color}
\usepackage{dsfont}
\usepackage{subfig}
\usepackage{caption}

\title{Quasi-stationary distributions and population processes}
\author{Sylvie M\'el\'eard and Denis Villemonais}

\DeclareMathSymbol{\minus}{\mathord}{operators}{"2D}

\newtheorem{theorem}{Theorem}
\newtheorem{lemme}[theorem]{Lemma}
\newtheorem{proposition}[theorem]{Proposition}
\newtheorem{remark}{Remark}
\newtheorem{example}{Example}
\newtheorem{definition}{Definition}

\newtheorem{corollary}[theorem]{Corollary}

\def\be{\begin{eqnarray}}
\def\ee{\end{eqnarray}}
\def\ben{\begin{eqnarray*}}
\def\een{\end{eqnarray*}}
\def\bei{\begin{itemize}}
\def\eei{\end{itemize}}
\def\me{\medskip \noindent}
\def\bi{\bigskip \noindent}
\def\E{\mathbb{E}}
\def\P{\mathbb{P}}
\def\R{\mathbb{R}}
\def\1{\mathbf{1}}
\def\N{\mathbb{N}}
\def\W{\mathbb{W}}

\begin{document}
\maketitle
\begin{abstract}
This survey concerns the study of quasi-stationary distributions with
a specific focus on models derived from ecology and population
dynamics. We are concerned with the long time behavior of different
stochastic population size processes when $0$ is an absorbing point
almost surely attained by the process. The hitting time of this point,
namely the extinction time, can be large compared to the physical time
and the population size can fluctuate for large amount of time before
extinction actually occurs.  This phenomenon can be understood by the
study of quasi-limiting distributions. In this paper, general results
on quasi-stationarity are given and examples developed in detail. One
shows in particular how this notion is related to the spectral
properties of the semi-group of the process killed at $0$. Then we
study different stochastic population models including nonlinear terms
modeling the regulation of the population. These models will take
values in countable sets (as birth and death processes) or in
continuous spaces (as logistic Feller diffusion processes or
stochastic Lotka-Volterra processes). In all these situations we study
in detail the quasi-stationarity properties. We also develop an
algorithm based on Fleming-Viot particle systems and show a lot of
numerical pictures. \end{abstract}


\section{Introduction}
\label{sec:intro}
\label{section:introduction}
We are interested in the long time behavior of isolated biological
populations with a regulated (density-dependent) reproduction.
Competition for limited resources impedes these natural populations
without immigration to grow indefinitely and leads them to become
extinct. When the population's size attains zero, nothing happens
anymore and this population's size process stays at zero. This point
$0$ is thus an absorbing point for the process. Nevertheless, the time
of extinction can be large compared to the individual time scale and
it is common that population sizes fluctuate for large amount of time
before extinction actually occurs. For example, it has been observed
 in populations of endangered species, as  the Arizona
ridge-nose rattlesnakes studied in Renault-Ferri\`ere-Porter
\cite{Renault}, that the statistics of some biological traits seem to
stabilize.  Another stabilization phenomenon is given by the
{mortality plateau}. While demographers thought for a long time
that the rate of mortality of individuals grows as an exponential function
of the age, it has been observed more recently that the rate of mortality slows at
advanced ages, or even stabilizes.  To capture these phenomena, we
will study the long time behavior of the process conditioned on non
extinction and the related notion of {quasi-stationarity}. In
particular, we will see that a Markov process with extinction which
possesses a quasi-stationary distribution has a mortality plateau.

\me In all the following,  the population's size
process $(Z_t, t\geq 0)$ will be   a Markov process going almost surely 
to extinction. We are interested in looking for characteristics of this
process giving more  information on its long time behavior. One way to approach this problem is to study
the "quasi-limiting distribution" (QLD) of the process (if it exists),
that is the limit, as $t\to +\infty$, of the distribution of $Z_t$
conditioned on non-absorption up to time $t$.  This distribution,
which is also called  Yaglom limit, provides particularly useful
information if the time scale of absorption is substantially larger
than the one of the quasi-limiting distribution. In that case, the
process relaxes to the quasi-limiting regime after a relatively short
time, and then, after a much longer period, absorption will eventually
occur.  Thus the quasi-limiting distribution bridges the gap between
the known behavior (extinction) and the unknown time-dependent
behavior of the process.

\noindent  There is another point of view concerning quasi-stationarity. A
quasi-stationary distribution for the process $(Z_t, t\geq 0)$ denotes
any proper initial distribution on the non-absorbing states
such that the distribution of $Z_t$ conditioned on non-extinction up
to time $t$ is independent of $t, t\geq 0$.

\me There is a large literature on quasi-stationary distributions and
Yaglom limits (see for example the large bibliography updated by
Pollett~\cite{Pollett}) and a lot of references will be given during the
exposition. The present paper is by no means exhaustive, but is a
survey presenting a collection of tools for the study of QSD
concerning specific population's size models. More than the
originality of the proofs, we emphasize some general patterns for
qualitative and quantitative results on QSD. We also provide a lot of
numerical illustrations of the different notions.

\me In Section~\ref{section:definitions} of this survey, we will
introduce the different notions of QSD and review theoretical
properties on QSD and QLD. We will also highlight the relations
between QSDs and mortality plateaus.  In
Section~\ref{section:finite-case}, we will study the simple case of
QSD for processes in continuous time with finite state space. We
develop a simple example to make things more concrete.  Thus we will
concentrate on QSD for several stochastic population models
corresponding to different space and time scales. We will underline
the importance of spectral theory as mathematical tool for the
research of QSD, in these different contexts.  In
Section~\ref{section:QSD-for-BD-process}, we will consider birth and
death processes. We will state results giving explicit conditions on
the coefficients ensuring the almost sure extinction of the process,
and the existence and uniqueness (or not) of a QSD. We will especially
focus on the density-dependence case, when the death rate of each
individual is proportional to the population's size (called logistic
birth and death process). We will show that in that case, the process
goes almost surely to extinction, and that there is a unique QSD,
coinciding with the unique QLD.
In Section~\ref{section:the-logistic-feller-diffusion-process}, the
birth and death process is rescaled by a growing initial population
and by small individual weights (small biomass). This process is
proved to converge, as the initial population's size tends to
infinity, to the unique solution of the deterministic logistic
equation, whose unique stable equilibrium is given by the carrying
capacity.  If the individual birth and death rates are proportional to
the population's size while preserving the ecological balance, more
numerous are the individuals, smaller are their weights and faster are
their birth and death events, reflecting allometric demographies. In
that case, the rescaled birth and death process converges, as the
initial population size increases, to the solution of a stochastic
differential equation with a $1/2$-H\"older diffusion coefficient and
a quadratic drift, called logistic Feller equation.  The existence of
the QSD is proved in this case and uniqueness is characterized by a
condition meaning the return of the process from infinity in finite
time.  The proof relies QSD investigation to spectral theory tools
developed in a functional $\mathbb{L}^2-$space. The logistic Feller
equation describes the size of a mono-type population where
individuals are indistinguable. Motivated by ecological and
biodiversity problems, we generalize the model to multi-type
populations with intra-specific and inter-specific competition. That
leads us to consider stochastic Lotka-Volterra processes. We give
conditions ensuring mono-type transient states or coexistence,
preserving one or several dominant types in a longer scale.
A large place in the survey is given to illustrations obtained by
simulations. Some of them derive from an approximation method based on
Fleming-Viot interacting particle systems, which is carefully
described in Section~\ref{section:simulation}.

%
\bi\textbf{A brief bibliography on quasi-stationary distributions}

\me The study of quasi-stationary distributions began with the work of
Yaglom on sub-critical Galton-Watson processes \cite{Yaglom1947}. Since
then, the existence, uniqueness and other properties of
quasi-stationary distributions for various processes have been
studied.

\me In particular, the case of Markov processes on finite state spaces has
been studied by Darroch and Seneta, who proved under some
irreducibility conditions the existence and uniqueness of the QSD, for
both discrete \cite{Darroch1965} and continuous time settings \cite{Darroch1967}
(detailed proofs and results are reproduced in Section~\ref{section:finite-case} of the
present paper). We also refer the reader to the works of van Doorn and
Pollett~\cite{vanDoorn2009} for a relaxation of the irreducibility
condition.

\me The case of discrete time birth and death processes has been treated
by Seneta and Vere-Jones~\cite{Seneta1966} and Ferrari, Mart\'inez and
Picco~\cite{Ferrari1991}. For continuous time birth and death
processes, we refer to van Doorn~\cite{vanDoorn1991}. This last case
is quite enlightening, since it leads to examples of processes with no
QSD, of processes with a Yaglom limit and an infinite number of QSD
and to processes with a Yaglom limit which is the unique QSD (detailed
proofs and results are also developed in
Section~\ref{section:QSD-for-BD-process} of this survey). For further
developments, we may refer to Pakes and Pollett~\cite{Pakes1989}
(where results on continuous-time birth and death processes with
catastrophic events are obtained), to Bansaye~\cite{Bansaye2009}
(where a discrete time branching process in random environment is
studied), to Coolen-Schrijner~\cite{Coolen2006} (where general discrete time
processes are studied) and references therein.

\me Diffusion processes have also been extensively studied in the past
decades, beginning with the seminal work of Mandl~\cite{Mandl1961} for
the one-dimensional case and of Pinsky~\cite{Pinsky1985} and Gong,
Qian and Zhao~\cite{Gong1988} in the multi-dimensional
situation. Mart\'inez, Picco and San Mart\'in~\cite{Martinez1998} and
Lladser and San Mart\'in~\cite{Lladser2000} highlighted cases of
diffusions with infinitely many quasi-limiting distributions, with a
non-trivial dependence on the initial distribution of the process. For
recent development of the theory of QSDs for diffusion processes, we
refer to Steinsaltz and Evans~\cite{Steinsaltz2007} and Kolb and
Steinsaltz~\cite{Kolb2010} where the case of one dimensional
diffusions with different boundary conditions is studied. We also
emphasize that in the case of Wright-Fisher diffusions and some of its
relatives, Huillet~\cite{Huillet2007} derived explicit values of
QSDs. Other diffusion processes related to demographic models have
been studied in Cattiaux, Collet, Lambert, Mart\'inez, M\'el\'eard and
San Mart\'in~\cite{Cattiaux2009}, where the case of the Feller
logistic diffusion is developed (proofs and results are also written
in detail in
Section~\ref{section:the-logistic-feller-diffusion-process} of this
paper), and in Cattiaux and M\'el\'eard~\cite{Cattiaux2008}, where the
case of a two dimensional stochastic Lotka-Volterra system is
developed ($k$ types stochastic Lotka-Volterra systems are also
studied in Section~\ref{subsection:multi-type-population} of this
survey).

\me Let us mention the original renewal approach of Ferrari, Kesten,
Mart\'inez and Picco~\cite{Ferrari1995}, also studied recently by Barbour and
Pollett~\cite{Barbour2010} in order to provide an approximation method for
the QSD of discrete state space Markov processes in continuous
time. Other approximation methods have been proposed by Pollett and
Stewart~\cite{Pollett1994} and by Hart and Pollett~\cite{Hart2000}. In this
survey, we describe the approximation method based on Fleming-Viot
type interacting particle systems, introduced by Burdzy, Holyst,
Ingerman et March~\cite{Burdzy1996} in 1996 and studied later by Burdzy,
Holyst and March~\cite{Burdzy2000}, Grigorescu and Kang~\cite{Grigorescu2004},
Ferrari and Mari\`c~\cite{Ferrari2007}, Villemonais~\cite{Villemonais2010}
\cite{Villemonais2011} and Asselah, Ferrari and Groisman~\cite{Asselah2011}.

\me For studies on the so-called Q-process, which is the process
distributed as the original process conditioned to never extinct, we
refer the reader to the above cited articles~\cite{Mandl1961},
\cite{Pinsky1985}, \cite{Gong1988}, \cite{Darroch1965}, \cite{Darroch1967} and, for
further developments, to the works of Collet, Mart\'inez and San
Mart\'in~\cite{Collet1995} and of Lambert~\cite{Lambert2007} and references
therein.

%



\bi\textbf{The framework}

\me Let us now introduce our framework  in more details. The  population's size $(Z_t, t\geq 0)$ is a
Markov process taking values in a subset $E$ of $\mathbb{N}$ or
$\mathbb{R}_+$, in a discrete or continuous time setting.  If the
population is isolated, namely without immigration, then the state
$0$, which describes the extinction of the population, is a
trap. Indeed, if there are no more individuals, no reproduction can
occur and the population disappears. Thus if the system reaches $0$,
it stays there forever, that is, if $Z_t=0$ for some $t$, then $Z_s=0$
for any $s\geq t$.

\me We denote by $T_0$ the  extinction time, i.e. the stopping time
 \begin{equation}
 \label{extinction-time}
 T_0= \inf\{t>0, Z_t = 0\}.
 \end{equation}
 We will consider cases  for which the process goes almost surely to
 zero, whatever the initial state is, namely, for all $z\in
 E$,
 \begin{equation}
 \label{extinction}
 \mathbb{P}_z(T_0<\infty)=1.
 \end{equation}
 Before extinction, the process
 takes its values in the space
 $ E^*= E \backslash \{0\}.$
 Any long time distribution of the process conditioned on
 non-extinction will be supported by $E^*$.

\bi \textbf{Notations} For any probability measure $\mu$ on $E^*$, we
denote by $\P_{\mu}$ (resp. $\E_{\mu}$) the probability (resp. the
expectation) associated with the process $Z$ initially distributed
with respect to $\mu$. For any $x\in E^*$, we set
$\P_x=\P_{\delta_x}$ and $\E_x=\E_{\delta_x}$.
We denote by $(P_t)_{t\geq 0}$ the semi-group of the process $Z$ killed
at $0$.  More precisely, for any $z>0$ and $f$ measurable and
bounded on $E^*$, one defines
\begin{equation}
 \label{sgkill}
 P_tf(z) = \mathbb{E}_z(f(Z_t) {\bf 1}_{t<T_0}).
\end{equation}
For any finite measure $\mu$ and any bounded measurable function $f$, we set
$$\mu(f)=\int_{E^*} f(x) \mu(dx),$$ and we also define the finite measure
$\mu P_t$ by 
\begin{equation*}
\mu P_t (f)=\mu(P_t f)=\mathbb{E}_{\mu}(f(Z_t) {\bf 1}_{t<T_0}).
\end{equation*}

\section{Definitions, general properties and first examples}
\label{section:definitions}

\me
There are several natural questions associated with this situation.

\me \textit{Question 1} What is the distribution of the size of a
non-extinct population at a large time $t$ ?  The mathematical
quantity of interest is thus the conditional distribution of $Z_t$
defined, for any Borel subset $A\subset E^*$, by
\begin{equation}
\label{conditioning}
\mathbb{P}_{\nu}(Z_t\in A | T_0>t)=\frac{\mathbb{P}_{\nu}(Z_t\in A
; T_0>t)}{\mathbb{P}_{\nu}( T_0>t)}=\frac{\nu P_t(\mathbf{1}_A)}{\nu P_t (\mathbf{1}_{E^*})},
\end{equation}
where $\nu$ is the initial distribution of the population's size
$Z_0$.  We will study the asymptotic behavior of this conditional
probability when $t$ tends to infinity. The first definition that we
introduce concerns the existence of a limiting conditional distribution.
\begin{definition}
Let $\alpha$ be a probability measure on $E^*$. We say that $\alpha$ is a \textbf{quasi-limiting distribution} (QLD) for $Z$, if there exists a
probability measure $\nu$ on $E^*$ such that, for any measurable set
$A\subset E^*$,
\begin{equation*}
\lim_{t\rightarrow\infty} \P_{\nu}\left(Z_t\in A | T_0 > t\right)=\alpha(A).
\end{equation*}
\end{definition}
\noindent 
In some cases 
the long time behavior of the conditioned distribution can be proved to be
initial state independent. This
leads to the following definition.
\begin{definition}
We say that $Z$ has a \textbf{Yaglom limit} if there exists a
probability measure $\alpha$ on $E^*$ such that, for any $x\in E^*$ and any
measurable set $A\subset E^*$,
\begin{equation}
\label{yag}
\lim_{t\rightarrow\infty} \P_{x}\left(Z_t\in A | T_0 > t\right)=\alpha(A).
\end{equation}
\end{definition}
\noindent When it exists, the Yaglom limit is a QLD. The reverse isn't true in
general and \eqref{yag} will actually not imply  the same property for any initial distribution.

\me \textit{Question 2} As in the ergodic case, we can ask if this
Yaglom limit has the conditional stationarity property given by the
following definition.
\begin{definition}
  Let $\alpha$ be a probability measure on $E^*$. We say that $\alpha$
  is a \textbf{quasi-stationary distribution} (QSD) if, for all $t\geq 0$ and any measurable set $A\subset E^*$,
\begin{equation*}
\alpha(A)=\P_{\alpha}\left(Z_t\in A| T_0 > t\right).
\end{equation*}
\end{definition}
\noindent The main questions are: Does a QSD exists? Is there a unique
QSD for the process? We will study examples where  QSDs do not
exist, or with an infinity of QSDs, or with a
unique QSD. The relation between the existence of QSD, QLD and Yaglom
limit is clarified in Proposition \ref{prop:qldimpliesqsd}
below. Namely, we will prove that
\begin{equation*}
  \text{Yaglom limit }\Rightarrow \text{QSD }\Leftrightarrow \text{ QLD}.
\end{equation*}

\me \textit{Question 3} Since the processes we are interested in
become extinct in finite time almost surely, the event $t<T_0$ becomes
a rare event when $t$ becomes large. An important question is then to
know whether the convergence to the Yaglom limit happens before the
typical time of extinction, or if it happens only after very large
time periods, in which case the populations whose size are distributed
with respect to the Yaglom limit are very rare. Both situations can
appear, as illustrated by the simple example of
Section~\ref{subsectionSimpleExample}.

\me \textit{Question 4} While most of theoretical results on QLDs, QSDs
and Yaglom limits are concerned with existence and uniqueness
problems, it would be useful in practice to have qualitative
information on the Yaglom limit. We present here particle
approximation results and numerical computations of the Yaglom limit
for some population's size models, providing some enlightenment on
Question 3 above.

\me \textit{Question 5} Another mathematical quantity related to this
conditioning is based on a pathwise point of view. In the finite state
space case of Section~\ref{section:finite-case} and the logistic
Feller diffusion case of
Section~\ref{section:the-logistic-feller-diffusion-process}, we will
describe the distribution of the trajectories who never attain the
trap. This will allow us to define a process, commonly referred to as
the $Q$ process for $Z$. We will prove that the new process defined by
this distribution is ergodic, and that its stationary distribution is
absolutely continuous with respect to the QSD (but not equal).

\bi The present section is organized as follows. In
Subsection~\ref{subsectionGeneralProperties}, we state general
properties of QLDs, QSDs and Yaglom limits. In
Subsection~\ref{subsectionGaltonWatson}, we develop the case of the
Galton-Watson process. This discrete time process is of historical
importance, since the notion of Yaglom limit has originally been
developed for this process by Yaglom itself (see
\cite{Yaglom1947}). In Subsection~\ref{subsectionSimpleExample}, we
develop a very simple example of a process evolving in a finite subset
of $\mathbb{N}$. For this process, one can easily prove the existence
of the Yaglom limit, the uniqueness of the QSD, and compare the speed of
extinction to the speed of convergence to the Yaglom limit. We also
provide  numerical computation of the relevant quantities.

\subsection{General properties}
\label{subsectionGeneralProperties}
\label{subsection:general-properties}

Most of the following results are already known by the QSD community. In this section, we emphasize their generality.

\subsubsection{QSD, QLD and Yaglom limit}
\label{subsubsection:QSD-QLD-Yaglom-limit}

\me It is clear that any Yaglom limit and any QSD is also a QLD. The
reverse implication has been proved by Vere-Jones~\cite{Vere1969} for
continuous time Markov chains evolving in a countable state
space. The following proposition extends this result to the general
setting.

\begin{proposition}
\label{prop:qldimpliesqsd}
\label{proposition:qld-implies-qld}
Let $\alpha$ be a probability measure on $E^*$. The distribution
$\alpha$ is a QLD for $Z$ if and only if it is a QSD for $Z$.
\end{proposition}

\begin{remark} \upshape
When it exists, the Yaglom limit is uniquely
defined, while there are processes with an infinity of QSDs (see the
birth and death process case of Section~\ref{sectionQSDbdProcess}). We
immediately deduce that there exist QSDs which aren't a Yaglom
limit.
\end{remark}

\begin{proof} (1) If $\alpha$ is a QSD then it is a QLD for $Z$ starting with distribution $\alpha$.

\me (2) Assume now that $\alpha$ is a QLD for $Z$ and for an initial probability
measure $\mu$ on $E^*$.
 Thus, for any measurable and bounded  function $f$ on
$E^*$, 
\begin{eqnarray*}
  \alpha(f)&=&\lim_{t\to \infty} \mathbb{E}_\mu(f(Z_t)| T_0>t)
  = \lim_{t\to \infty} \frac{\mathbb{E}_\mu(f(Z_t); T_0>t)}{\mathbb{P}_\mu(T_0>t)}.
\end{eqnarray*}
 Applying the latter with
$f(z)=\mathbb{P}_z(T_0>s)$, we get by the Markov property
\begin{equation*}
\mathbb{P}_\alpha( T_0>s)=\lim_{t\to
\infty} \frac{\mathbb{P}_\mu( T_0>t+s)}{\mathbb{P}_\mu(T_0>t)}.
\end{equation*}

\me  Let us now consider
$f(z)=\mathbb{P}_z(Z_s\in A, T_0>s)$, with $A\subset E^*$. Applying the
Markov property again, it yields
\begin{eqnarray*}
  \mathbb{P}_{\alpha}(Z_s\in A; T_0>s) &=& \lim_{t\to \infty}
  {\mathbb{P}_\mu(Z_{t+s}\in A; T_0>t+s)\over \mathbb{P}_\mu(
  T_0>t)}\\
  &=& \lim_{t\to \infty} {\mathbb{P}_\mu(Z_{t+s}\in A;
  T_0>t+s)\over \mathbb{P}_\mu( T_0>t+s)} {\mathbb{P}_\mu(
  T_0>t+s)\over \mathbb{P}_\mu( T_0>t)}.
\end{eqnarray*}
By  definition of the QLD
$\alpha$, ${\mathbb{P}_\mu(Z_{t+s}\in A; T_0>t+s)\over \mathbb{P}_\mu(
T_0>t+s)}$ converges to $\alpha(A)$ and ${\mathbb{P}_\mu( T_0>t+s)\over \mathbb{P}_\mu( T_0>t)}$
converges to $\mathbb{P}_\alpha( T_0>s)$, when $t$ goes to infinity.
We deduce that, for any Borel set $A$ of $E^*$ and any
$s>0$,
\begin{equation*}
\alpha(A)=\mathbb{P}_{\alpha}(Z_s \in A | T_0>s).
\end{equation*}
The probability measure $\alpha$ is then a QSD.
\end{proof}

\subsubsection{Exponential extinction rate}
\label{subsubsection:exponential-extinction-rate}

\begin{proposition}
\label{propositionExtinctionRate}
\label{proposition:extinction-rate-from-qsd}
 Let us consider a Markov process $Z$
with absorbing point $0$ satisfying \eqref{extinction}. Assume
that $\alpha$ is a QSD for the process. Then there exists a positive real number $\theta(\alpha)$
 depending on the QSD such that
 \begin{equation}
 \label{extinction-rate}
 \mathbb{P}_{\alpha} (T_0>t) = e^{-\theta(\alpha) t}.
 \end{equation}
\end{proposition}

\me
 This theorem shows us that starting from a QSD, the extinction
 time has an exponential distribution with parameter $\theta(\alpha)$ independent of $t>0$, given by
 $$ \theta(\alpha) = -  {\ln \mathbb{P}_{\alpha}
 (T_0>t)\over t}.
 $$

\me
\begin{proof}
By the Markov
property,
\begin{eqnarray*}
  \mathbb{P}_{\alpha} \left(T_0>t+s\right) 
    &=& \mathbb{E}_{\alpha}\left(\mathbb{P}_{Z_t} ( T_0>s)\mathbf{1}_{T_0>t}\right)\\
    &=& \mathbb{P}_{\alpha}(T_0 >t) \mathbb{E}_{\alpha}\left(\mathbb{P}_{Z_t} ( T_0>s)| T_0>t\right),
\end{eqnarray*}
since $T_0\leq t$ implies $Z_t=0$, and $\mathbb{P}_{0} ( T_0>s)=0$.
By definition of a QSD, we get
\begin{eqnarray*}
  \mathbb{E}_{\alpha}\left(\mathbb{P}_{Z_t} ( T_0>s)| T_0>t \right) = \mathbb{P}_{\alpha}(T_0>s).
\end{eqnarray*}
Hence we obtain that for all $s,t>0$,
$  \mathbb{P}_{\alpha} ( T_0>t+s) = \mathbb{P}_{\alpha} ( T_0>s)
  \mathbb{P}_{\alpha} ( T_0>t).
$
Let us denote $g(t) =\mathbb{P}_{\alpha} ( T_0>t)$. We have $g(0)=1$
and, because of \eqref{extinction}, $g(t)$ tends to $0$ as $t$ tends
to infinity. An elementary proof allows us to
conclude that there exists a real number $\theta(\alpha)>0 $ such
that
$$\mathbb{P}_\alpha(T_0>t) = e^{-\theta(\alpha) t}.$$
\end{proof}

\subsubsection{QSD and exponential moments}
\label{subsubsection:QSD-and-exponential-moments}

 \me The following statement gives a necessary condition for the
 existence of QSDs in terms of existence of exponential moments of
 the hitting time $T_{0}$.

\begin{proposition}
\label{prop:moment}
\label{proposition:moment-for-extinction-time}
Assume  that $\alpha$ is a QSD. Then, for
any $\, 0<\gamma<\theta(\alpha)$,
\begin{equation}
\label{prop-moment} \mathbb{E}_{\alpha}(e^{\gamma T_0})<+\infty.
\end{equation}
In particular, there exists a positive number $z$ such that
$\mathbb{E}_z(e^{\gamma T_0})<+\infty$.
\end{proposition}

\noindent Proposition \ref{prop:moment} suggests that if the population can
escape extinction for too long times with positive probability, then
the process has no QSD. This is the case for the critical
Galton-Watson process: its extinction time is finite almost surely,
but its expectation isn't finite.

\begin{proof}  
We compute the exponential moment in continuous and discrete time settings.
In both cases, it is finite if and only if
$\theta(\alpha)>\gamma$.

\me In the continuous time setting, \eqref{extinction-rate} says that,
under $\mathbb{P}_{{\alpha}}$, $T_{0}$ has an exponential distribution
with parameter $\theta(\alpha)$. We deduce that, for any $\theta(\alpha)>\gamma$,
\begin{equation*}
  \E_{\alpha}\left(e^{\gamma
    T_0}\right)=\frac{\theta(\alpha)}{\theta(\alpha)-\gamma}.
\end{equation*}
In the discrete time setting, \eqref{extinction-rate} says that under
$\P_{\alpha}$, $T_0$ has a geometric distribution with parameter
$e^{-\theta(\alpha)}$. We deduce that
\begin{equation*}
  \E_{\alpha}\left(e^{\gamma
    T_0}\right)=\frac{1-e^{-\theta(\alpha)}}{e^{-\gamma}-e^{-\theta(\alpha)}}.
\end{equation*}
 Since $\E_{\alpha}(e^{\gamma T_0})$ is equal to
$\int_{E^*} \mathbb{E}_z(e^{\gamma T_0}) \alpha(dz)$, the finiteness of the
integral implies the last assertion.

\end{proof}

\begin{remark} \rm In the particular case of an irreducible continuous time Markov chain
with state space $\N$ such that
 $ \lim_{z\rightarrow+\infty} \P_z(T_0\leq t)=0,\ \forall t\geq 0,$
Ferrari, Kesten, Mart\'inez and Picco~\cite{Ferrari1995} proved that
the existence of the moment \eqref{prop-moment} for some
$z\in\mathbb{N}$ and some $\gamma>0$ is equivalent to the the existence of a
quasi-stationary distribution.

\me It is actually not true in any case, as shown by the following
counter-example.  Let $Z$ be a continuous time random walk on $\N$
reflected on $1$ and killed at rate $1$. Thus, for any
$\lambda\in[0,1[$ and any probability measure $\mu$ on $\N$,
    $\mathbb{E}_{\mu}(e^{\lambda T_0})$ is finite. Nevertheless the
    conditional distribution $\mathbb{P}_ z(Z_t\in\cdot|t<T_0)$ is the
    distribution of a standard continuous time random walk reflected
    on $1$, which converges to $0$ as $t$ tends to infinity. In
    particular $Z$ has no QLD and thus no QSD.
\end{remark}

\subsubsection{A spectral point of view}
\label{subsubsection:a-spectral-point-of-view}
In this section, the results are stated in the continuous time
setting.  The operator $L$ with domain ${\cal D}(L)$ denotes the infinitesimal generator of
the sub-Markovian semi-group $(P_{t})$ associated with the killed
process $Z$.  The next proposition links the existence of QSDs for
$Z$ and the spectral properties of the dual of the operator $L$.
It is one of the main tools used in a large literature studying QSDs.

\begin{proposition}
\label{PrQSDeigenfunction}
\label{proposition:qsd-spectral-point-of-view}
Let $\alpha$ be a probability measure on $E^*$. We assume that there
exists a set $D \subset {\cal D}(L)$ such that, for any
measurable subset $A\subset E^*$, there exists a uniformly bounded
sequence $(f_n)$ in $D$ converging point-wisely to
$\mathbf{1}_{A}$.

 Then $\alpha$ is a quasi-stationary distribution if
and only if there exists $\theta(\alpha)>0$ such that
\begin{equation*}
\alpha (L f) = -\theta(\alpha) \alpha(f),\ \forall f\in{ D}.
\end{equation*}
\end{proposition}

We emphasize that the existence of $ D$ is always true if the
state space $E^*$ is discrete. It is also fulfilled if $E^*$ is an open subset
of $\mathbb{R}^d$ and if $Z$ is a diffusion with locally bounded
coefficients.

\begin{proof}
(1) Let $\alpha$ be a QSD for $Z$. By definition of a QSD, 
we have, for every Borel set $A\subseteq E^*$,
\begin{eqnarray*}
  \alpha(A)=\frac{\alpha P_t (\mathbf{1}_A)}{\alpha P_t (\mathbf{1}_{E^*})}.
\end{eqnarray*}

\me  By Theorem \ref{extinction-rate}, there
exists $\theta(\alpha)>0$ such that for each $t>0$,
\begin{equation*}
 \alpha P_t (\mathbf{1}_{E^*}) =  \mathbb{P}_\alpha(T_0>t) = e^{-\theta(\alpha) t}.
\end{equation*}
We deduce that, for any measurable set $A\subseteq E^*$, 
$  \alpha P_t (\mathbf{1}_A) = e^{-\theta(\alpha) t} \alpha(A),$
which is equivalent to 
$\ \alpha P_t= e^{-t \theta(\alpha)} \alpha$.
By Kolmogorov's forward equation and by assumption on $ D$, we
have
\begin{equation*}
  \left| \frac{\partial P_t f}{\partial t}(x) \right| = |P_t Lf(x)| \leq
  \|Lf\|_{\infty} <+\infty,\ \forall f\in{ D}.
\end{equation*}
In particular, one can differentiate $\ \alpha P_t
f=\int_{E^*} P_t f(x)\alpha(dx)$ under the integral sign, which implies that
\begin{equation*}
\alpha(Lf)=-\theta(\alpha) \alpha(f),\ \forall f\in{ D}.
\end{equation*}

\bi (2) Assume know that $\alpha(Lf)=-\theta(\alpha) \alpha(f)$ for
all $f\in{ D}$. By Kolmogorov's backward equation and the
same ``derivation under the integral sign'' argument, we have
\begin{equation*}
  \frac{\partial \alpha (P_t f)}{\partial t}=\alpha(L P_t
  f)=-\theta(\alpha)\alpha P_t (f),\ \forall f\in{ D}.
\end{equation*}
We deduce that
\begin{equation*}
  \alpha P_t (f)=e^{-t\theta(\alpha)} \alpha(f),\ \forall f\in{ D}.
\end{equation*}
By assumption, there exists, for any measurable subset $A\subset E^*$, a
uniformly bounded sequence $(f_n)$ in $ D$ which converges point-wisely to $\mathbf{1}_A$.
Finally, we deduce by  dominated convergence that
\begin{equation*}
 \alpha P_t( \mathbf{1}_A)= e^{-t\theta(\alpha)} \alpha(A). 
\end{equation*}
This implies immediately that $\alpha$ is a quasi-stationary
distribution for $Z$.
\end{proof}

\subsubsection{Long time limit of the extinction rate}
\label{subsubsection:limit-of-the-extinction-rate}

Another quantity of interest in the demography and population's
dynamics is given by the long time behavior of the killing or
extinction rate. In the demography setting, the process $Z$ models the
vitality of some individual and $t$ its physical age. Thus $T_{0}$ is
the death time of this individual. The long time behavior of the
extinction rate has been studied in detail by Steinsaltz-Evans
\cite{Steinsaltz2004} for specific cases.

\me
 The definition of the
extinction rate depends on the time setting:
\begin{itemize}
  \item In the discrete time setting, the extinction rate of $Z$
    starting from $\mu$ at time $t\geq 0$ is defined by
  \begin{equation*}
    r_{\mu}(t)=\P_{\mu}(T_0=t+1|T_0> t).
  \end{equation*}
  \item In the continuous time setting,  the extinction rate of $Z$
    starting from $\mu$ at time $t\geq 0$ is defined by
    \begin{equation*}
      r_{\mu}(t)=-\frac{\frac{\partial}{\partial t}\P_{\mu}(T_0> t)}{\P_{\mu}(T_0>t)},
    \end{equation*}
    when the derivative exists and is integrable with respect to $\mu$.
\end{itemize}

\me Historically (\textit{cf.} \cite{Gompertz1825}), demographers
applied the Gompertz law meaning that this extinction rate was
exponentially increasing with time. However in 1932, Greenwood and
Irwin \cite{Greenwood1939} observed that in some cases, this behavior
was not true. In particular there exist cases where the extinction
rate converges to a constant when time increases, leading to the
notion of mortality plateau.  This behavior of the extinction rate has
been observed in experimental situations (see for instance
\cite{Carey1992}). 

 \me The QSDs play a main role in this framework. Indeed, by
 Proposition~\ref{propositionExtinctionRate}, if $\alpha$ is a
 QSD, then the extinction rate $r_{\alpha}(t)$ is constant and given
 by
  \begin{equation*}
    r_{\alpha}(t)=\left\lbrace
    \begin{array}{l}
       1-e^{-\theta(\alpha)}\ \text{in the discrete time setting}\\
       \theta(\alpha)\ \text{in the continuous time setting}
    \end{array}
    \right. ,\ \forall t\geq 0.
  \end{equation*}
  We refer to the introduction of Steinsaltz-Evans
  \cite{Steinsaltz2004} for a nice discussion of the notion of QSD in
  relationship with mortality plateaus.

\me In the next proposition, we prove that  the existence
of a QLD for $Z$ started from $\mu$ implies the existence of a long term mortality plateau.

\begin{proposition}
  \label{PrExtinctionRate}
  \label{proposition:convergence-of-extinction-rate}
  Let $\alpha$ be a QLD for $Z$, initially distributed with respect to a
probability measure $\mu$ on $E^*$.  
  In the continuous time setting, we assume moreover that there exists $h>0$ such
  that $L (P_h \mathbf{1}_{E^*})$ is well defined and bounded.  In
  both time settings, the rate of extinction converges in the long
  term:
  \begin{equation}
    \label{EqPrYa1}
    \lim_{t\rightarrow\infty} r_{\mu}(t)=r_{\alpha}(0).
  \end{equation}
\end{proposition}

\begin{proof}

\noindent In the discrete time setting, by the semi-group property and
the definition of a QLD, we have
  \begin{equation*}
    \begin{split}
    r_{\mu}(t)=1-\frac{\mu P_{t}(P_1\mathbf{1}_{E^*})}{\mu P_{t}(\mathbf{1}_{E^*})}
            \   \xrightarrow[t\rightarrow+\infty]{} \ 1 -  \alpha(P_1 \mathbf{1}_{E^*}) = r_\alpha(0).
    \end{split}
  \end{equation*}
  The limit is by definition the extinction rate at time $0$ of
  $Z$ starting from $\alpha$.

\noindent In the continuous time setting, by the Kolmogorov's backward
equation, we have
 \begin{equation*}
   \frac{\partial}{\partial t}P_{t+h} \mathbf{1}_{E^*}(x)= P_t L(P_h \mathbf{1}_{E^*})(x),\ \forall x\in E^*.
 \end{equation*}
 Since $L( P_h\mathbf{1}_{E^*})$ is assumed to be bounded, we deduce that
 \begin{equation*}
   \frac{\partial}{\partial t}\mu P_{t+h}(\mathbf{1}_{E^*})= \mu P_t L(P_h \mathbf{1}_{E^*}).
 \end{equation*}
 Then
  \begin{equation*}
    \begin{split}
      \frac{\frac{\partial}{\partial t}\mu P_{t+h}
        (\mathbf{1}_{E^*})}{\mu P_t (\mathbf{1}_{E^*})}
      =\frac{\mu P_t
        L(P_h \mathbf{1}_{E^*})}{\mu P_t
        (\mathbf{1}_{E^*})}
      \ \xrightarrow[t\rightarrow\infty]{}\ 
        \alpha(L P_h \mathbf{1}_{E^*}) =-\theta(\alpha) \alpha(P_h
        \mathbf{1}_{E^*}),
    \end{split}
  \end{equation*}
  by the definition of a QLD and by Proposition
  \ref{PrQSDeigenfunction}. We also have
  \begin{equation*}
    \frac{\mu (P_{t+h} \mathbf{1}_{E^*})}{\mu (P_t \mathbf{1}_{E^*})}
    \xrightarrow[t\rightarrow\infty]{}
    \alpha(P_h \mathbf{1}_{E^*}).
  \end{equation*}
  Finally, we get
  \begin{equation*}
    r_{\mu}(t+h)=-\frac{\frac{\partial}{\partial t}\mu (P_{t+h}
        \mathbf{1}_{E^*})}{\mu (P_{t+h} \mathbf{1}_{E^*})}
     \xrightarrow[t\rightarrow\infty]{} \theta(\alpha),
  \end{equation*}
  which allows us to conclude the proof of Proposition
  \ref{PrExtinctionRate}.
\end{proof}

\subsection{A historical example in discrete time: the Galton-Watson process}
\label{subsectionGaltonWatson}
\label{subsection:Galton-Watson-process}
The Galton-Watson process is a population's dynamics
model in discrete time, whose size $(Z_n)_{n\geq0}$ evolves according to the recurrence
formula $Z_0$ and
\begin{equation*}
  Z_{n+1}=\sum_{i=1}^{Z_n}\xi^{(n)}_i,
\end{equation*}
where $(\xi^{(n)}_i)_{i,n}$ is a family of independent random
variables, identically distributed following the   probability
measure $\mu$ on $\mathbb{N}$ with generating function $g$. As
defined, $Z_n$ is the size of the $n^{th}$ generation of a population
where each individual has a random number of children, chosen following $\mu$ and independently of the rest of the population.
This process has been introduced by Galton and Watson (see
\cite{Galton1874}) in order to study the extinction of aristocratic
surnames.

\noindent  We will assume in the whole section that
 $ 0<\mu(\{0\})+\mu(\{1\})<1$.
We denote by $m=E(\xi^{(0)}_1)$ the average number of children by
individual in our Galton-Watson process. The independence of  descendants implies that starting from $Z_{0}$, the process $Z$ is equal to the sum of $Z_{0}$ independent Galton-Watson processes issued from a single individual. By this branching property, the
probability of extinction for the population starting from one
individual is obtained as follows:
\begin{equation*}
  \mathbb{P}_1(\exists n\in\mathbb{N},\ Z_n=0)=\lim_{n\rightarrow+\infty} \mathbb{E}_1(0^{Z_n})
   =\lim_{n\rightarrow\infty} g\circ \cdots \circ g(0)\ (n\text{ times}).
\end{equation*}
 There are three different situations
(see for instance Athreya-Ney \cite{Athreya1972}):
\begin{itemize}
\item[-] The sub-critical case $m<1$: the
  process becomes extinct in finite time almost surely and the average extinction time
  $\E(T_0)$ is finite.
\item[-] The critical case $m=1$:  the
  process becomes extinct in finite time almost surely, but $\E(T_0)=+\infty$.
\item[-] The super-critical case $m>1$: the process is never
  extinct with a positive probability, and it yields immediately that $\E(T_0)=+\infty$.
\end{itemize}

\begin{theorem}[Yaglom \cite{Yaglom1947}, 1947]
  \label{ThYaglom}
  \label{theorem:yaglom}
  Let $(Z_{n})_{n\geq 0}$ be a Galton-Watson process with the
  reproduction generating function $g$. There is no quasi-stationary
  distribution in the critical and the super-critical case. In the
  sub-critical case, the Yaglom limit exists and is the unique QSD of
  $Z$. Moreover, its generating function $\hat g$ fulfills
  \begin{equation}
    \label{EqGenYaglom}
    \hat g(g(s))=m\hat g(s)+1-m,\ \forall s\in[0,1].
  \end{equation}
\end{theorem}

\begin{proof}
The proof is adapted from Athreya-Ney \cite{Athreya1972} p. 13-14. In
the critical or the super-critical case, we have $\E_1(T_0)=+\infty$,
which implies that $\E_{\alpha}(T_0)=+\infty$ for all probability
measure $\alpha$ on $\mathbb{N}^*$. We deduce from Proposition
\ref{prop:moment} that there is no QSD.

\me Assume now that $m<1$. Let us fix an arbitrary probability
measure $\nu$ on $\mathbb{N}^*$ and prove that there exists a QLD
$\alpha$ for $Z$ starting with distribution $\nu$. 
For each $n\geq 0$, we denote by $g_n$ the generating function of
$Z_n$, 
 $ g_n(s)=\E_{\nu}\left(s^{Z_n}\right)\ \forall s\in [0,1]$. Recall that $g_{n+1}=g_1\circ g_n $. 
 Let us also denote by $\hat g_n$ the generating function of $Z_n$ conditioned
to $\{Z_n>0\}=\{T_{0}>n\}$:
\begin{eqnarray*}
 \hat g_{n}(s) &=& \E_{\nu}(s^{Z_{n}}|\, Z_{n}>0) = \frac{\E_{\nu}(s^{Z_{n}}{\bf 1}_{ Z_{n}>0}) }{\P_{\nu}(Z_{n}>0)}\\
&=& \frac{\E_{\nu}(s^{Z_{n}}) - \P_{\nu}(Z_{n}=0) }{1-\P_{\nu}(Z_{n}=0)}\\
&=& \frac{g_{n}(s)-g_{n}(0)}{1-g_{n}(0)} = 1 - \frac{1-g_{n}(s)}{1-g_{n}(0)} \in [0,1].
\end{eqnarray*}
Note that $\hat g_{n}(0) =0$, which is quite natural since the
conditional law doesn't charge $0$. 
For a fixed $s\in[0,1)$, we set $\Gamma(s)=\frac{1-g_1(s)}{1-s}$. Then we
  have, for all $n\geq 0$,
\begin{equation*}
1-\hat g_{n+1}(s)=\frac{\Gamma(g_n(s))}{\Gamma(g_n(0))}\left(1-\hat g_{n}(s)\right).
\end{equation*}
 Since $g_1$ is convex, $\Gamma$
is non-decreasing. 
Moreover $m<1$ implies that $g_n(x)\geq x$, so that $g_n(s)$ and $1-\hat g_n(s)$ are non-decreasing in $n$.
 In particular, $\lim_{n\rightarrow\infty} \hat
g_n(s)$ exists. Let us denote by $\hat g(s)$ its limit and by $\alpha$
the corresponding finite measure (whose mass is smaller than one). In
order to prove that $\alpha$ is a probability measure on
$\mathbb{N}^*$, it is sufficient to prove that $\hat g(s)\rightarrow
1$ when $s$ goes to $1$. We have
\begin{equation*}
  \Gamma(g_n(0))\left(1-\hat g_{n+1}(s)\right)=\left(1-\hat g_{n}(g_1(s))\right).
\end{equation*}
Taking the limit on each size,
where $\lim_{n\rightarrow\infty} \Gamma(g_n(0))=\Gamma(1)=m$, we deduce
that
\begin{equation*}
 m( 1-\hat g(s) )= 1-\hat g(g_1(s)),
\end{equation*}
which implies Equation \eqref{EqGenYaglom}. Since $\lim_{s\rightarrow
  1} g_1(s)=1$ and $m<1$, then $\hat g(1)=1$. Finally, $\alpha$ is a
QLD for $Z$ starting with distribution $\nu$.

\me One could think \textit{a priori} that the function $\hat g$
depends on the starting distribution $\nu$. We prove now that it isn't the
case, so that there is a unique QLD, and then a unique QSD, which is
also the Yaglom limit of the process (indeed, one could choose
$\nu=\delta_x$, $x\in\mathbb{N}^*$).

\noindent Assume that there exist two generating functions $\hat g$ and $\hat
h$ which fulfill Equation \eqref{EqGenYaglom}. By induction, we have,
for all $n\geq 1$ and all $s\in[0,1]$,
\begin{eqnarray*}
  \hat g(g_n(s))=m^n \hat g(s)+\left(m^{n-1}+\cdots+m+1\right)(m-1),\\
  \hat h(g_n(s))=m^n \hat h(s)+\left(m^{n-1}+\cdots+m+1\right)(m-1).
\end{eqnarray*}

\noindent We deduce that for $s\in [0,1[$
\begin{equation*}
{\hat g}'(g_n(s))\ g'_{n}(s)= m^n\, {\hat g}'(s)\ ;\ {\hat h}'(g_n(s))\ g'_{n}(s)= m^n\, {\hat h}'(s).
\end{equation*}
Since for the sub-critical case $g_{n}(0) \uparrow 1$ when
$n\rightarrow\infty$, for any $s\in [0,1[$ there will be a $k$ such
    that
$$g_{k}(0) \leq s\leq g_{k+1}(0).$$
Hence,
\begin{equation*}
  \frac{{\hat g}'(s)}{{\hat h}'(s)}=\frac{{\hat g}'(g_n(s))}{{\hat h}'(g_n(s))}\leq \frac{{\hat g}'(g_{n+k+1}(0))}{{\hat h}'(g_{n+k}(0))}=  \frac{{\hat g}'(0)}{{\hat h}'(0)} \,\frac{m\,g'_{n+k}(0)}{g'_{n+k+1}(0)} = \frac{{\hat g}'(0)}{{\hat h}'(0)} \,\frac{m}{g'(g_{n+k}(0))}.
\end{equation*}
When  $n$ goes to infinity, we obtain $ \frac{{\hat g}'(s)}{{\hat h}'(s)}\leq  \frac{{\hat g}'(0)}{{\hat h}'(0)}$. The converse inequality is established similarly.
Since $\hat g$ and $\hat h$ are generating functions of probability
measures on $\mathbb{N}^*$, we have $\hat g(0)=\hat h(0)=0$ and $\hat
g(1)=\hat h(1)=1$. Finally, the two functions $\hat g$ and $\hat h$ are
equal, which concludes the proof of Theorem \ref{ThYaglom} .
\end{proof}

\subsection{The simple example of an ergodic process with uniform killing in a finite state space}
\label{subsection:a-simple-example}
\label{subsectionSimpleExample}
We present a very simple Markov process with extinction whose
quasi-stationary distribution, Yaglom limit, speed of extinction
and speed of convergence to the Yaglom limit are very easy to obtain.

\noindent Let $(X_t)_{t\geq 0}$ be an exponentially ergodic Markov process which evolves in the state
space $E^*=\{1,\cdots,N\}$, $N\geq 1$. By exponentially
ergodic, we mean that there exist a probability measure $\alpha$ on
$E^*$ and two positive constants $C,\lambda>0$ such that, for all
$z\in\{1,\cdots,N\}$ and all $t\geq 0$,
\begin{equation*}
  \sup_{i\in E^*}\left| \P_z(X_t=i)-\alpha(\{i\})\right|\leq C e^{-\lambda t}.
\end{equation*}
There is no possible extinction for $(X_t)$. Let $d>0$ be a positive
constant and let $\tau_d$ be an exponential random time of parameter
$d$ independent of the process $(X_t)$. We define the  process $(Z_t)$  by setting
\begin{equation*}
  Z_t=
\left\lbrace
\begin{array}{l}
  X_t,\ \text{if}\ t<\tau_d\\
  0,\ \text{if}\ t\geq \tau_d.
\end{array}
\right.
\end{equation*}
This model is a model for the size of a population which cannot be extinct, except at a catastrophic event which happens with rate $d$. Thus we have
\begin{equation*}
  \P_z(t<T_0)=e^{-dt},\ \forall t\geq 0.
\end{equation*}
The conditional distribution of $Z_t$ is simply given by
the distribution of $X_t$:
\begin{equation*}
  \P_z(Z_t=i|Z_t\neq 0)=\P_z(X_t=i),\ \forall z\in E^*.
\end{equation*}
We deduce that the unique QSD is the Yaglom limit $\alpha$ and that  for all $z\in E^*$ and all $t\geq 0$,
\begin{equation*}
  \sup_{i\in E^*}\left| \P_z(Z_t=i|T_0>t)-\alpha(\{i\})\right|\leq C e^{-\lambda t}.
\end{equation*}
Thus in this case, the conditional distribution of $Z$ converges
exponentially fast to the Yaglom limit $\alpha$, with rate $\lambda>0$
and the process becomes extinct exponentially fast, with rate $d>0$.
  
Hence the comparison between the speed of convergence to the Yaglom
limit and the speed of extinction will impact the observables of the
process before extinction:
\begin{itemize}
\item[(a)] If $\lambda \gg d$, then the convergence to the Yaglom
  limit happens before the typical time of extinction of the
  population and the quasi-stationary regime will be observable.
\item[(b)] If $\lambda \ll d$, then the extinction of the population
  occurs before the quasi-stationary regime is reached. As a
  consequence, we are very unlikely to observe the Yaglom limit.
  \item[(c)] If $\lambda \sim d$, the answer is not so immediate and
    depends on other parameters, as in particular the initial
    distribution.
\end{itemize}

\bi

\begin{example}\upshape
The population size $Z$ is described by a random walk in continuous
time evolving in $E=\{0,1,2,\cdots,N\}$ with transition rates given by
\begin{eqnarray*}
  &&i\rightarrow i+1\text{ with rate }1,\text{ for all }i\in\{1,2,\cdots,N-1\},\\
  &&i\rightarrow i-1\text{ with rate }1,\text{ for all }i\in\{2,3,\cdots,N\},\\
  &&i\rightarrow 0\text{ with rate }d>0,\text{ for all }i\in\{1,2,\cdots,N\}.
\end{eqnarray*}
The boundedness of the population size models a constraint of fixed
resources which acts on the growth of the population. We will see more
realistic fixed resources models including logistic death rate in the
next sections.
One can check that the quasi-stationary probability measure of $Z$ is given
by $\alpha_i=1/N$ for all $i\in E^*$.

\me {\it Numerical simulations}.  We fix $N=100$. In that finite case,
one can obtain by numerical computation the whole set of eigenvalues
and eigenvectors of the infinitesimal generator $L$ (we use here the
software SCILAB).  Numerical computation using the fact that $\lambda$
is the spectral gap of the generator of $X$ gives $\lambda=0.098$. For
different values $d=0.001$, $d=0.500$ and $d=0.098$, we compute
numerically the mathematical quantities of interest: the extinction
probability $\P_{z}(T_{0}>t)=e^{-dt}$ as a function of $t$ (cf. Figure
\ref{figExample1} left picture) and the distance $ \sup_{i\in
  E^*}\left| \P_z(Z_t=i|T_0>t)-\alpha(\{i\})\right|$ between the
conditional distribution of $Z_t$ and $\alpha$ as a function of $-\log
\P_{z}(T_{0}>t)$, which gives the extinction's time scale.
(cf. Figure \ref{figExample1} right picture).
\begin{center}
\begin{figure}
\includegraphics[width=13cm]{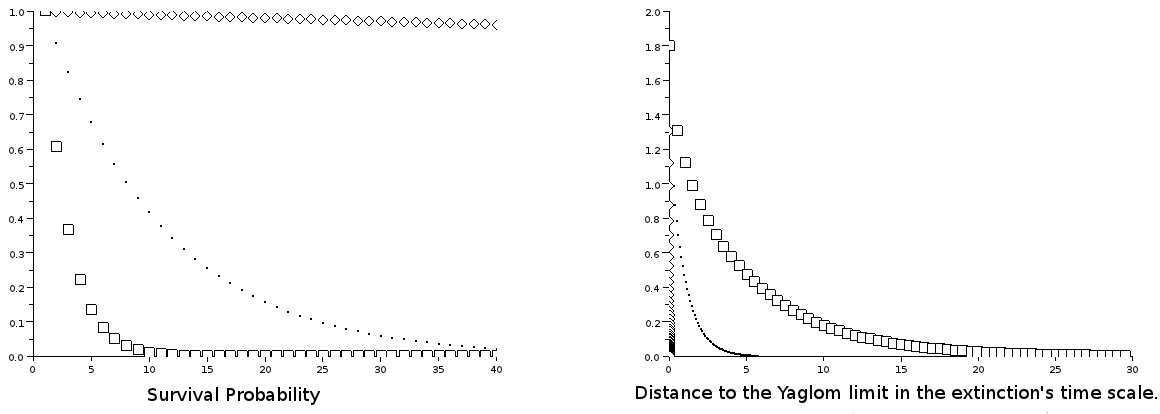}
\caption{\label{figExample1} Example 1. A numerical computation leads
  to $\lambda=0.098$. Three different situations are observed, which
  lead to three very different patterns for the speed of convergence
  to the Yaglom limit in the extinction's time scale: ($\diamond$) $\lambda\gg d=0.001$; ($\Box$)
  $\lambda\ll d=0.500$; ($\cdot$) $\lambda=d=0.098$.  }
\end{figure}
\end{center}
 We observe that the convergence to the Yaglom limit happens rapidly
in the case ($\diamond$)~$\lambda=0.098\gg d=0.001$. Indeed the
distance to the Yaglom limit is equal to $0.05$, while the survival
probability can't be graphically distinguished from $1$. On the
contrary, we observe that the convergence happens very slowly in the
case ($\Box$) $\lambda=0.098\ll d=0.500$. Indeed, the distance to the
Yaglom limit is equal to $0.05$ when the survival probability appears
to be smaller than $e^{-15}\simeq 3\times 10^{-7}$. The case
($\cdot$)~$\lambda=0.98=d$ is an intermediate case, where the distance
to the Yaglom limit is equal to $0.05$ when the survival probability
appears to be equal to $e^{-3}\simeq 0.05$.


\end{example}

\section{The finite case, with general  killing's rate}
\label{section:finite-case}

\subsection{The quasi-stationary distributions}

\bi The Markov process $(Z_t)_{t\geq 0}$ evolves in continuous time in $E=\{0,1,...,N\}$, $N\geq 1$ and we still assume that $0$ is
its unique absorbing state.  The
semi-group $(P_t)_{t\geq 0}$  is the sub-Markovian semi-group of the killed process and we still
denote by $L$ the associated infinitesimal generator. In this finite state space
case, the operators $L$ and $P_t$ are matrices, and a probability
measure on the finite space $E^*$ is  a vector of non-negative
entries whose sum is equal to $1$. The results of this section
have been originally proved by Darroch and Seneta (\cite{Darroch1965} and \cite{Darroch1967}).

 \begin{theorem}
   \label{theorem-qsd-finitecase}
   \label{theorem:qsd-finitecase}
   Assume that $Z$ is an irreducible and aperiodic process before
   extinction, which means that there exists $t_0>0$ such that the
   matrix $P_{t_0}$ has only positive entries (in particular, it
   implies that $P_t$ has positive entries for  $t>t_{0}$). Then the
   Yaglom limit $\alpha$ exists and is the unique QSD of the process
   $Z_t$.

   Moreover, denoting by $\theta(\alpha)$ the extinction rate
   associated to $\alpha$ (see
   Proposition~\ref{proposition:extinction-rate-from-qsd}), there
   exists a probability measure $\pi$ on $E^*$ such that, for any
   $i,j\in E^*$,
   \begin{equation*}
     \lim_{t\to \infty} e^{\theta(\alpha) t}\,
     \mathbb{P}_i(Z_t=j) = \pi_i \,\alpha_j
   \end{equation*}
   and
   \begin{equation*}
     \lim_{t\to \infty} {\mathbb{P}_i(T_0>t+s)\over \mathbb{P}_j(T_0>t)} = {\pi_i\over
       \pi_j} e^{-\theta(\alpha) s}.
   \end{equation*}
 \end{theorem}

 The main tool of the proof of Theorem \ref{theorem:qsd-finitecase} is the
 Perron-Frobenius Theorem, which gives us a complete description of
 the spectral properties of $P_{t}$ and $L$. The main point is that
 the matrix $P_{1}$ has positive entries. For the proof of the
 Perron-Frobenius Theorem, we refer to Gantmacher
 \cite{Gantmacher1959} or Serre \cite{Serre2002}.
 \begin{theorem}[Perron-Frobenius Theorem]
   \label{thm:perron-frobenius}
   \label{theorem:perron-frobenius}
   Let $(P_{t})$ be a submarkovian semi-group on $\{1,\cdots,N\}$ such
   that the entries of $P_{t_{0}}$ are positive for $t_{0}>0$.  Thus,
   there exists a unique positive eigenvalue $\rho$, which is the
   maximum of the modulus of the eigenvalues, and there exists a
   unique left-eigenvector $\alpha$ such that $\alpha_i>0\ $ and
   $\ \sum_{i=1}^N \alpha_i = 1$, and there exists a unique
   right-eigenvector $\pi$ such that $\pi_i>0\ $ and $\ \sum_{i=1}^N
   \alpha_i \pi_i = 1$, satisfying
   \begin{equation}
     \label{equationThmPerron1}
     \alpha P_ {t_{0}}= \rho\, \alpha \ ;\ P_{t_{0}} \pi = \rho\, \pi.
   \end{equation}
   
   \noindent In addition, since $(P_t)$ is a sub-Markovian
   semi-group, $\rho <1$ and there exists $\theta>0$ such that
   $\rho = e^{-\theta}$.  Therefore
   \begin{equation}
     \label{rateP} P_t=e^{-\theta t} A + \vartheta(e^{-\chi t}),
   \end{equation}
   where $A$ is the
   matrix defined by $A_{ij}= \pi_i \alpha_j$, and $\chi >  \theta$
   and $\vartheta(e^{-\chi t})$ denotes a matrix such that none of the
   entries exceeds $C e^{-\chi t}$, for some constant $C>0$.
 \end{theorem}

\begin{proof}[Proof of Theorem \ref{theorem:qsd-finitecase}]
Applying Perron-Frobenius Theorem to the submarkovian semi-group
$(P_t)_{t\geq 0}$, it is immediate from \eqref{rateP} that there
exists $\theta>0$ and a probability measure $\alpha$ on $E^*$ such
that, for any $i,j\in E^*$,
  \begin{equation}
    \label{EqFi1}
    e^{\theta t}\mathbb{P}_i(Z_t=j) = e^{\theta t} [P_t]_{ij}= \pi_i \alpha_j + \vartheta(e^{-(\chi-\theta) t}).
  \end{equation}
  Summing over $j\in E^*$, we deduce that
  \begin{equation}
     \label{EqFi2}
     e^{\theta t}\mathbb{P}_i(T_0>t) = \pi_i + \vartheta(e^{-(\chi-\theta)t}).
  \end{equation}
  It follows that, for any $i,j\in E^*$,
  \begin{equation*}
    \mathbb{P}_i(Z_t=j|T_0>t) =
           {\mathbb{P}_i(Z_t=j)\over \mathbb{P}_i(T_0>t)} \xrightarrow[t\to
             \infty]{} \alpha_j.
  \end{equation*}
  Thus the Yaglom limit exists and is equal to $\alpha$.
  Since $E$ is finite, we have for any initial distribution
  $\nu$ on $E^*$
 \begin{equation*}
   \lim_{t\rightarrow\infty} \mathbb{P}_{\nu}(Z_t=j|T_0>t)=\sum_{i\in
     E^*} \nu_i \lim_{t\rightarrow\infty} \mathbb{P}_{i}(Z_t=j|T_0>t)
   =\sum_{i\in E^*} \nu_i \alpha_j=\alpha_j.
 \end{equation*}
 We deduce that the Yaglom limit $\alpha$ is the unique QLD of $Z$,
 and thus it is its unique QSD.
 By Proposition \ref{propositionExtinctionRate}, we have $\alpha
 P_1(\mathbf{1}_{E^*})=e^{-\theta(\alpha)}$. By
 \eqref{equationThmPerron1}, this  quantity is also equal to
 $e^{-\theta}$, so that $\theta=\theta(\alpha)$. The end of
 Theorem~\ref{theorem:qsd-finitecase} is thus a straightforward
 consequence of \eqref{EqFi1} and \eqref{EqFi2}
\end{proof}

\begin{remark}\upshape
  One can deduce from \eqref{EqFi1} and \eqref{EqFi2} that there
  exists a positive constant $C_{L}$ such that
  \begin{equation*}
    \sup_{j\in E^*,  i\in E^* }\left| P_i(Z_t=j | Z_t>0 )-\alpha_{j}
    \right|\leq C_L e^{-(\chi-\theta(\alpha))},
  \end{equation*} 
  where the quantity $\chi-\theta(\alpha)$ is the spectral gap of $L$,
  \textit{i.e.} the distance between the first and the second
  eigenvalue of $L$.  Thus if the time-scale $\chi-\theta(\alpha)$ of the
  convergence to the quasi-limiting distribution is substantially
  bigger than the time scale of absorption ($\chi-\theta(\alpha)
  \gg \theta(\alpha)$), the process will relax to the QSD after a
  relatively short time, and after a much longer period, extinction
  will occur. On the contrary, if $\chi-\theta(\alpha) \ll
  \theta(\alpha)$, then the extinction happens before the process had
  time to relax to the quasi-limiting distribution.

\noindent   In intermediate cases, where $\lambda-\theta(\alpha)\approx
  \theta(\alpha)$, the constant $C_L$, which depends on the whole set
  of eigenvalues and eigenfunctions of $L$, plays a main role which
  needs further investigations.
\end{remark}

\me 
We generalize Example 1 to a more realistic case where the killing's rate can depend
on the size of the population. For instance, it can be higher  for a small population
than for a big one.

\begin{example}\upshape

Let $Z$ be a Markov process which models a population whose individuals 
reproduce and die independently, with individual birth rate
$\lambda>0$ and individual death rate $\mu=1$. In order to take into
account the finiteness of the resources, the process is reflected when
it attains a given value $N$, that we choose here arbitrarily equal to $100$.
Thus the process $Z$ evolves in the finite state space $\{0,1,\cdots,100\}$ and
its transition rates are given by
\begin{eqnarray*}
  &&i\rightarrow i+1\text{ with rate }\lambda i,\text{ for all }i\in\{1,2,\cdots,99\},\\
  &&i\rightarrow i-1\text{ with rate }  \mu i,\text{ for all }i\in\{1,2,3,\cdots,100\}.
\end{eqnarray*}
The infinitesimal generator of $Z$ is given by
\begin{eqnarray*}
&&L_{1,1}=-1 -\lambda \text{ and }  L_{1,2}=\lambda,\\
&&L_{i,i-1}= i,\quad L_{i,i}=- (1+\lambda)i \text{ and }
 L_{i,i+1}=\lambda\, i,\quad \forall i\in\{2,\cdots,99\},\\
&&L_{100,99}=100\text{ and } L_{100,100}=-100,\\
&&L_{i,j}=0,\ \forall i,j\in\{1,\cdots,100\}\text{ such that }|j-i|>1.
\end{eqnarray*}
The process $Z$ clearly fulfills the conditions of Theorem
\ref{theorem:qsd-finitecase}.  As a consequence, it has a Yaglom limit
$\alpha$, which is its unique QSD.  Moreover, the probability measure
$\alpha$ is the unique normalized and positive left eigenvector of
$L$. Since $L$ is a finite matrix of size $100\times 100$, one can
numerically compute the whole set of eigenvectors and eigenvalues of
the matrix $(L_{ij})$.  This will allow to obtain numerically the
Yaglom limit $\alpha$, its associated extinction rate
$\theta(\alpha)$, and the speed of convergence $\chi-\theta(\alpha)$.
Moreover, for any $t\geq 0$, one can compute the value of $e^{tL}$,
which is equal to $P_t$ (the semi-group of $Z$ at time $t$). Hence, we
may obtain the numerical value of the conditioned distribution
$\P_{Z_0}(Z_t\in .|t<T_0)$, for any initial size $Z_0$.  Finally, we
are also able to compute numerically the distance between $\alpha$ and
the conditioned distribution $\P_{Z_0}(Z_t\in .|t<T_0)$, for any value
of $\lambda>0$ and $Z_0\in\{1,\cdots,100\}$.

\me In Figure \ref{example2_yagloms}, we represent the Yaglom limit $\alpha$
 for different values of $\lambda$, namely
$\lambda=0.9$, $\lambda=1.0$ and $\lambda=1.1$.
\begin{figure}[position]
   \includegraphics[width=13cm]{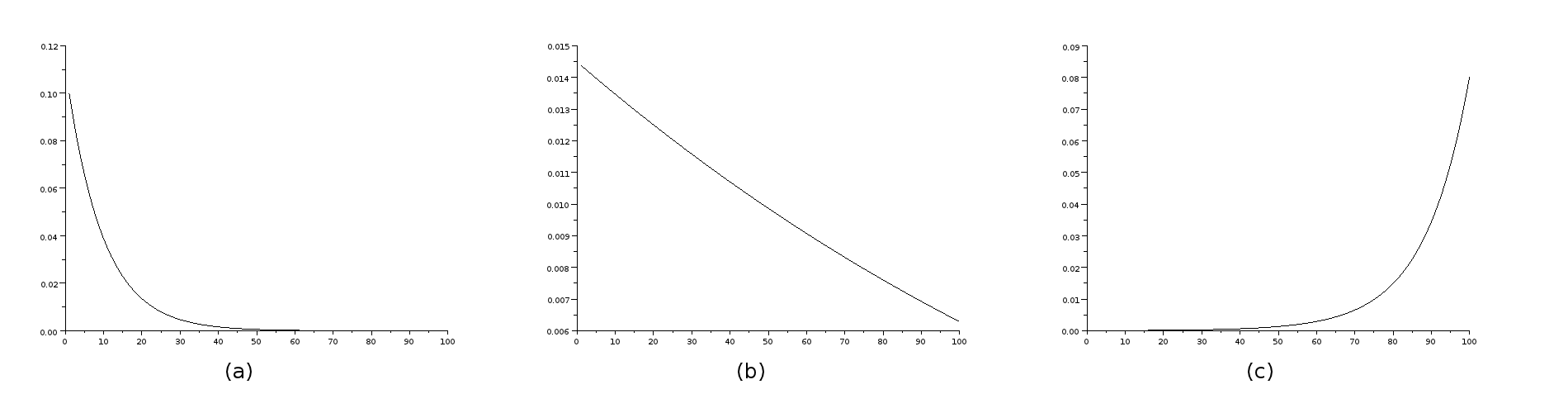}
   \caption{\label{example2_yagloms} Example 2. Yaglom limits for different
     values of $\lambda$. The following values of $\theta(\alpha)$ are obtained
     by numerical computation.  (a) $\lambda=0.9$, $\theta(\alpha)=0.100$;
     (b) $\lambda=1.0$, $\theta(\alpha)=0.014$; (c)
     $\lambda=1.1$, $\theta(\alpha)= 5.84 \times
     10^{-5}$.}
\end{figure}
Let us comment the numerical results.
\begin{itemize}
  \item[(a)] In the first case ($\lambda=0.9$), an individual is more
    likely to die than to reproduce and we observe that the Yaglom
    limit is concentrated near the absorbing point $0$. The rate of
    extinction $\theta(\alpha)$ is the highest in this case, equal to
    $0.100$.  In fact, the process reaches the upper bound $100$ very
    rarely, so that the behavior of the process is very similar to the
    one of a linear birth and death process with birth and death rates
    equal to $\lambda$ and $\mu$ respectively.  In
    Section~\ref{section:QSD-for-BD-process}, we study such linear
    birth and death processes. We show that the Yaglom limit (which
    exists if and only if $\lambda<\mu$) is given by a geometric law
    and $\theta(\alpha)=\mu-\lambda$.

  \item[(b)] In the second case ($\lambda=\mu=1$), we observe that
    $\alpha$ decreases almost linearly from $\alpha_1$ to
    $\alpha_{100}$ and the upper bound $N=100$ plays a crucial
    role. In fact, letting $N$ tend to $+\infty$, one would observe
    that for any $i\geq 1$, $\alpha_i$ decreases to $0$. The
    extinction rate $\theta({\alpha})$ which is equal to $0.014$ for
    $N=100$ would also go to $0$. The counterpart of this phenomenon
    for the linear birth and death process studied in
    Section~\ref{section:QSD-for-BD-process} is that the Yaglom limit
    will not exist when $\mu=\lambda$.

    \item[(c)] In the third case ($\lambda=1.1$), the Yaglom limit
      $\alpha$ is concentrated near the upper bound $100$, while the
      extinction rate is $\theta(\alpha)=5.84\times 10^{-5}$. The
      comparison with the linear birth and death process is no more
      relevant, since the important factor in this case is the effect
      of the upper bound $N=100$, which models the finiteness of the
      resources in the environment. 
\end{itemize}

\me In Figure \ref{example2_graphs}, we study the effect of the
initial position and of the value of the parameter $\lambda$ on the
speed of convergence to the Yaglom limit and on the speed of
extinction. We choose the positions $Z_0=1$, $Z_0=10$ and $Z_0=100$,
and we look at the two different cases $\lambda=0.9$ and
$\lambda=1.1$, which correspond to the subcritical case (a) and to
the supercritical case (c) respectively. We represent, for each set of values of
$(\lambda,Z_0)$, the distance to the Yaglom limit,
$\sup_{i\in\{1,\cdots,100\}}|\P_{Z_0}(Z_t=i|t<T_0)-\alpha_i|$ as a
function of the time, and the same distance as a function of the
logarithm of the survival probability $-\log\P_{Z_0}(t<T_0)$ (i.e. the
extinction time scale). By numerical computation, we also obtain that
\begin{itemize}
  \item[(a)] $\lambda=0.9$: $\theta(\alpha)=0.100$ and $\theta(\alpha)-\chi=0.102$.
  \item[(c)] $\lambda=1.1$: $\theta(\alpha)=5.84\times 10^{-5}$ and $\theta(\alpha)-\chi=0.103$.
\end{itemize}
\begin{figure}
\centering
\subfloat{\includegraphics[width=13cm]{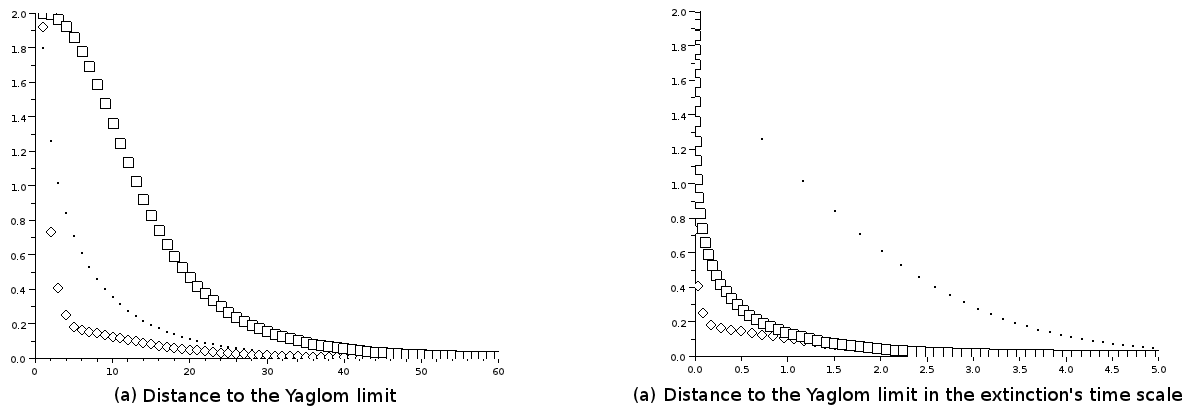}}\\
\subfloat{\includegraphics[width=13cm]{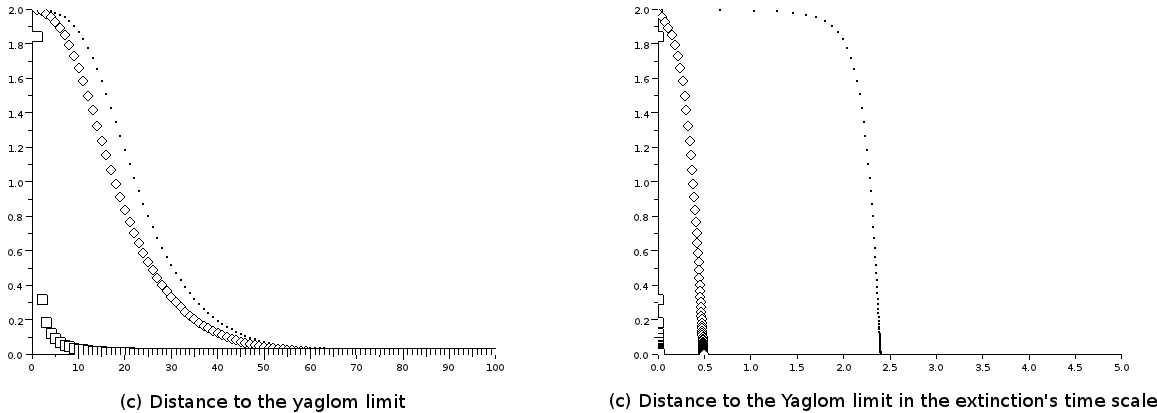}}
\caption{Example 2.  Pictures (a) and (c) correspond to different
  values of $\lambda$ (the following values of $\theta(\alpha)-\chi$
  have been obtained by numerical computation): (a) $\lambda=0.9$,
  $\theta(\alpha)=0.100$, $\theta(\alpha)-\chi=0.102$; (c)
  $\lambda=1.1$, $\theta(\alpha)=5.84\times 10^{-5}$,
  $\theta(\alpha)-\chi=0.103$; each curve corresponds to a given
  initial size of the population: $(\cdot)\ Z_0=1$;
  $(\diamond)\ Z_0=10$; $(\Box)\ Z_0=100$.}
\label{example2_graphs}
\end{figure}

\noindent In the case (a), we have $\theta(\alpha)=0.100\simeq
\chi-\theta(\alpha)=0.102$ and we observe that the speed of
convergence depends on the initial position in a non-trivial way:
while the survival probability is smaller for the process starting
from $10$ than for the process starting from $100$, the convergence to
the Yaglom limit in the extinction's time scale happens faster in the
case $Z_0=10$. In the case (c), we have $\theta(\alpha)=5.84\times
10^{-5} \ll \chi-\theta(\alpha)=0.103$. The speed of convergence to
the Yaglom limit in the extinction's time scale depends on the initial
position: if ($\Box$)~$Z_0=100$, then it is almost immediate; if
($\diamond$)~$Z_0=10$, the distance between the conditional
distribution and the Yaglom limit is equal to $0.05$ when the survival
probability is around $e^{-0.5}\simeq 0.61$; if ($\cdot$)~$Z_0=1$,
then this distance is equal to $0.05$ when the survival probability is
around $e^{-2.4}\simeq 0.091$.
\end{example}

\subsection{The Q-process}
Let us now study the marginals of the process conditioned to never be
extinct.
\begin{theorem}
  \label{thm:finiteqproc}
  \label{theorem:q-process-finite-case}
  Assume that we are in the conditions of
  Theorem~\ref{theorem:qsd-finitecase}.  For any $\ i_0, i_1, \cdots,
  i_k\in E^*$, any $\ 0<s_1< \cdots, s_k< t$, the limiting value
  $\lim_{t\to \infty} \mathbb{P}_{i_0}(Z_{s_1}={i_1}, \cdots,
  Z_{s_k}={i_k}|T_0>t)\ $ exists.

  \me
  Let $(Y_t, t\geq 0)$ be the process starting from $i_0\in E^*$ and
  defined by its finite dimensional distributions
  \begin{equation}
  \mathbb{P}_{i_0}(Y_{s_1}={i_1}, \cdots, Y_{s_k}={i_k}) =
  \lim_{t\to \infty}\mathbb{P}_{i_0}(Z_{s_1}={i_1}, \cdots,
  Z_{s_k}={i_k}|T_0>t).
  \end{equation}
  Then $Y$ is a Markov process with values
  in $E^*$ and transition probabilities given by
  \begin{equation*}
    \mathbb{P}_{i}(Y_{t}={j})= e^{\theta(\alpha) t}\ {\pi_i\over \pi_j}\
  P_{ij}(t).
  \end{equation*}
  It is conservative, and has the unique stationary
  probability measure $(\alpha_j \pi_j)_j$.
\end{theorem}

\me Remark that the stationary probability  is absolutely continuous with respect to
the QSD, but, contrary to intuition, it is not equal to the QSD.

\me
\begin{proof}
Let us denote $\theta(\alpha)$ by $\theta$ for simplicity. 
Let $i_0, i_1, \cdots, i_k\in E^*$ and $0<s_1< \cdots < s_k< t$.
We introduce the filtration ${\cal F}_s= \sigma(Z_u, u\leq
s)$. Then 
\begin{eqnarray*}
 \mathbb{P}_{i_0}(Z_{s_1}={i_1}, \cdots,
     Z_{s_k}={i_k}\ ;\ T_0>t)
 &=& \E_{i_0}\left(\mathbf{1}_{Z_{s_1}=i_1, \cdots, Z_{s_k}=i_k}\,
       \E_{i_0}\left(\mathbf{1}_{T_0>t}|{\cal F}_{s_k}\right)\right)\\
 &=& \E_{i_0}(\mathbf{1}_{Z_{s_1}={i_1}, \cdots, Z_{s_k}={i_k}}\,
       \mathbb{E}_{i_k}(\mathbf{1}_{T_0>t-s_k}))\\&& (\hbox{ by Markov property}) \\
 &=& \P_{i_0}\left(Z_{s_1}={i_1}, \cdots,Z_{s_k}={i_k}\right)\P_{i_k}(T_0>t-s_k). 
\end{eqnarray*}
By Theorem \ref{theorem:qsd-finitecase},
\begin{equation*}
  \lim_{t\rightarrow\infty} \frac{ \P_{i_k}(T_0>t-s_k)}{\mathbb{P}_{i_0}(T_0>t)}
      =\frac{\pi_{i_k}}{\pi_{i_0}} e^{\theta s_k}.
\end{equation*}
Thus 
\begin{eqnarray}
 \lim_{t\to \infty}\mathbb{P}_{i_0}(Z_{s_1}={i_1}, \cdots, Z_{s_k}={i_k}|T_0>t)
&=&\mathbb{P}_{i_0}(Z_{s_1}={i_1}, \cdots, Z_{s_k}={i_k})\frac{\pi_{i_k}}{\pi_{i_0}} e^{\theta s_k}. \label{equation:q-proc-finite-case-1}
\end{eqnarray}
Let us now show that $Y$ is a Markov
process. We have 
\begin{eqnarray*}
  \mathbb{P}_{i_0}(Y_{s_1}={i_1}, \cdots,
  Y_{s_k}={i_k}, Y_t=j)&=& e^{\theta t}\ {\pi_j\over \pi_{i_0}}\
  \mathbb{P}_{i_0}(Z_{s_1}={i_1}, \cdots, Z_{s_k}={i_k},
  Z_t=j)\\
  &=& e^{\theta(t-s_k)}\ e^{\theta s_k}\ {\pi_j\over \pi_{i_k}}
  {\pi_{i_k}\over \pi_{i_0}}\ \mathbb{P}_{i_0}(Z_{s_1}={i_1},
  \cdots, Z_{s_k}={i_k}) \\
  &&{ \quad \times \mathbb{P}_{i_k}( Z_{t-s_k}=j)
  \quad (\text{by
    Markov property of }Z)}\\
  &=&\mathbb{P}_{i_0}(Y_{s_1}={i_1}, \cdots, Y_{s_k}={i_k})\
  \mathbb{P}_{i_k}( Y_{t-s_k}=j),
\end{eqnarray*}
and thus $\mathbb{P}(Y_t = j|
  Y_{s_1}={i_1}, \cdots, Y_{s_k}={i_k}) = \mathbb{P}_{i_k}(
  Y_{t-s_k}=j).$

\me By \eqref{equation:q-proc-finite-case-1} and Theorem \ref{theorem:qsd-finitecase}, we have
\begin{equation*}
\mathbb{P}_{i}( Y_{t}=j)= {\pi_j\over \pi_i} \ \mathbb{P}_{i}(
Z_{t}=j)\ e^{\theta t} \xrightarrow[t\to +\infty]{} {\pi_j\over \pi_i}\ \alpha_j\ \pi_i =
\alpha_j \pi_j.
\end{equation*}

\me Moreover let us compute the infinitesimal generator $\hat{L}$
of $Y$ from the infinitesimal generator $L$ of $Z$.  We have for $j\neq i$,\ben \hat{L}_{ij} = \lim_{s\to 0}
\hat{P}_{ij}(s) = {\pi_j\over \pi_i}\ L_{ij}.\een For $j=i$, \ben
\hat{L}_{ii}& =& - \lim_{s\to 0} {1-\hat{P}_{ii}(s)\over s} = -
\lim_{s\to 0} {1-e^{\theta s}{P}_{ii}(s)\over s}\\
&=&- \lim_{s\to 0} {1- e^{\theta s} + e^{\theta s} (1-
p_{ii}(s))\over s} = \theta + L_{ii}.\een 
We thus check that
\ben \sum_{j\in
E^*} \hat{L}_{ij}= \sum_{j\in E^*} {\pi_j\over \pi_i} L_{ij} +
\theta.\een Since $L\pi = -\theta \pi$, then $\sum_{j\in E^*}
\pi_j L_{ij} = - \theta \pi_i$ and thus $\sum_{j\in E^*}
\hat{L}_{ij}= 0$.
\end{proof}

\section{QSD for birth and death processes}
\label{sectionQSDbdProcess}
\label{section:QSD-for-BD-process}
We are describing  here the dynamics of isolated asexual populations, as for
example populations of bacteria with cell binary division, in
continuous time. Individuals may reproduce or die, and there is only
one child per birth.  The population size dynamics will be modeled by
a birth and death process in continuous time. The individuals may
interact, competing (for example) for resources and therefore the
individual death's rate will depend on the total size of the
population. 
In a first part, we recall and partially prove some results on the
non-explosion of continuous time birth and death processes. We will
also recall conditions on the birth and death rates which ensure that
the process goes to extinction in finite time almost surely.  In a second part,
we concentrate on the cases where the process goes almost surely to
zero and we study the existence and uniqueness of quasi-stationary
distributions.

\subsection{Birth and death processes}
\label{subsection:BD-processes-general-study}

We consider birth and death processes with rates $(\lambda_i)_i$ and $(\mu_i)_i$, that is
 $\mathbb{N}$-valued
pure jump Markov
 processes, whose jumps are $+1$ or $-1$, with
transitions
\begin{eqnarray*}
  \label{birth-death}
  &i & \to \quad i+1 \quad \hbox{ with rate } \quad \lambda_i\ ,\\
  &i & \to \quad i-1 \quad \hbox{ with rate } \quad \mu_i,
\end{eqnarray*}
where $\lambda_i$ and $\mu_i$, $i\in \mathbb{N}$, are non-negative
real numbers.

\me
Knowing that the process is at state $i$ at a certain time,  the
process will wait for an exponential time of parameter $\lambda_i$
before jumping to $i+1$ or independently, will wait for an
exponential time of parameter $\mu_i$ before jumping to $i-1$. The
total jump rate from state $i$ is thus $\lambda_i+\mu_i$. We
will assume in what follows that $\lambda_0=\mu_0=0$. This
condition  ensures  that $0$ is an absorbing point, modeling the
extinction of the population. Since these processes have a main importance in the modeling of biological processes, we study in detail their existence and extinction properties, and then their QSDs.

 \me The most standard examples
are the following ones.
\begin{enumerate}
\item The Yule process. For each  $i\in \mathbb{N}$, $\
\lambda_i=\lambda i\ $ for a positive real number $\lambda$, and
$\ \mu_i =0$. There are no deaths. It's a fission model.

\item The linear birth and death process, or binary branching
process. There exist positive numbers $\lambda$ and $\mu$ such that
$\lambda_i = \lambda i$ and $ \mu_i=\mu i$. This model holds if
individuals reproduce and die independently, with birth rate equal to
$\lambda$ and death rate equal to $\mu$.

\item The logistic birth and death process. We assume that
every individual in the population has a constant birth rate
$\lambda>0$ and a natural death rate $\mu>0$. Moreover the individuals compete to share fixed
resources, and each individual $j\neq i$ creates a competition pressure on
individual $i$ with rate $c>0$. Thus, given that the population's size
is $i$, the individual death rate due to competition is given by $c(i-1)$ and the total
death rate is $\mu_i= \mu i + c i(i-1).$
\end{enumerate}

\bi
In the following, we
will assume that $\lambda_i>0$ and $\mu_i>0$ for any
$i\in\mathbb{N}^*$.

\me We denote by $(\tau_n)_n$ the sequence of the jump times of the
process, either births or deaths. Let us first see under which
conditions on the birth and death rates the process is well defined
for all time $t\geq 0$, i.e. $\tau=\lim_n \tau_n = + \infty$ almost surely. Indeed,
if $\tau=\lim_n \tau_n < \infty$ with a positive probability, the
process would only be defined for $t< \tau$ on this event. There would
be an accumulation of jumps near $\tau$ and the process could
increase until infinity in finite time.

\me Let us give a necessary and sufficient condition ensuring that a
birth and death process does not explode in finite time. The result is already
stated in Anderson~\cite{Anderson1991},
but the following proof is actually far much shorter and easier to follow.
\begin{theorem}
  \label{thm:BDexistence}
  \label{theorem:BDexistence}
  The birth and death process does not explode in finite time, almost
  surely, if and only if $\sum_{n}r_{n}=+\infty,$ where
  \begin{equation*}
    r_n = {1\over \lambda_n} + \sum_{k=1}^{n-1} {\mu_{k+1} \cdots
      \mu_n\over \lambda_k \lambda_{k+1} \cdots \lambda_n} + {\mu_{1}
      \cdots \mu_n\over \lambda_1 \cdots \lambda_n}.
  \end{equation*}
\end{theorem}

\begin{proof}
1) Let us more generally consider a pure jump Markov process $(X_t,
 t\geq 0)$ with values in $\mathbb{N}$, and generator $(L_{ij}, i, j
 \in \mathbb{N})$. We set $q_i= - L_{ii}$.  Let $(\tau_n)_n$ be
 the sequence of jump times of the process and $(U_n)_n$ the sequence
 of inter-times defined by
 \begin{equation*}
   U_n=\tau_n-\tau_{n-1}, \quad  \forall n\geq 1; \quad \tau_0=0, \quad U_0=0.
 \end{equation*}
 We also set $\tau_{\infty}=\lim_{n\rightarrow\infty}\tau_n
 \in[0,+\infty]$. The process does not explode in finite time almost
 surely (and is well defined for all time $t\in\mathbb{R}_+$), if and
 only if for each $ i \in \mathbb{N}$
 \begin{equation*}
   \mathbb{P}_i(\tau_{\infty} <\infty) = 0.
 \end{equation*}

 \me Let us show that this property is equivalent to the fact that the
 unique non-negative and bounded solution $x = (x_i)_{i\in \mathbb{N}}$
 of $\ L\, x = x\ $ is the null solution.

 \me For any $i$, we set $h_i^{(0)}=1$ and, for $n\in \mathbb{N}^*$,
 $h_i^{(n)} = \mathbb{E}_i(\exp(-\sum_{k=1}^n U_k))$.  For any $n\in
 \mathbb{N}$, we have
 \begin{equation*}
   h_i^{(n+1)} = \sum_{j\neq i} {L_{ij}\over q_i} \
   h_j^{(n)}\ \mathbb{E}_i(\exp(-U_1)).
 \end{equation*}
 Indeed, the property is true for $n=0$ since $\sum_{i\neq j}
 {L_{ij}\over q_i} = 1$.  Moreover, by conditioning with respect to $\
 U_1$ and using the strong Markov property, we get
 \begin{eqnarray}
   \label{EqConditionalExp1}
   \mathbb{E}_i\left(\left.\exp(-\sum_{k=1}^{n+1} U_k)\right| U_1\right)&=&
   \exp(- U_1)\
   \mathbb{E}_{X_{U_1}}\left(\exp(-\sum_{k=1}^n U_k)\right),
 \end{eqnarray} 
 since the jump times of the $U_1$-translated process are the $\tau_n
 - U_1, n\in \mathbb{N}^* $. We
 have
 \begin{eqnarray*}
   {\E}_i\left(\mathbb{E}_{X_{U_1}}\left(\exp(-\sum_{k=1}^n U_k)\right)\right) 
     &=& \sum_{j\neq i} \mathbb{P}_i(X_{U_1} = j)\ \mathbb{E}_j\left(\exp(-\sum_{k=1}^n U_k)\right)\\
     &=& \sum_{j\neq i} \frac{L_{ij}}{q_i}\ \mathbb{E}_j(\exp(-\sum_{k=1}^n U_k)),
 \end{eqnarray*}
 since $\mathbb{P}_i(X_{U_1} = j)= {L_{ij}\over q_i}$. By
 independence of $U_1$ and $X_{U_1}$, we deduce from
 \eqref{EqConditionalExp1} that
 \begin{equation*}
   \mathbb{E}_i\left(\exp(-\sum_{k=1}^{n+1} U_k)\right) =
   \sum_{j\neq i} {L_{ij}\over q_i}\ \mathbb{E}_j\left(\exp(-\sum_{k=1}^n
   U_k)\right) \ \mathbb{E}_i\left(\exp(- U_1)\right).
 \end{equation*}

 \bi As
 \begin{equation*}
   \mathbb{E}_i(\exp(-U_1))=\int_0^\infty q_i e^{-q_i s} e^{-s} ds =
	  {q_i\over 1+ q_i},
 \end{equation*}
 it turns out that
 \begin{equation}
   \label{rec}
   h_i^{(n+1)}= \sum_{j\neq i} {L_{ij}\over 1
     + q_i}\ h_j^{(n)}.
 \end{equation}

 \me Let $(x_i)_i$ be a nonnegative solution of $Lx=x$ bounded by $1$, then
   $x_i=\sum_{j} L_{ij}\ x_j= L_{ii} x_i + \sum_{j\neq i}
   L_{ij} x_j = -q_i x_i + \sum_{j\neq i} L_{ij} x_j$,
 so that
 \begin{equation}
   \label{rec-bis}
   x_i=\sum_{j\neq i} {L_{ij}\over 1+q_i} x_j.
 \end{equation}
 Since $h_i^{(0)}=1 \geq x_i\geq 0$ and $\frac{L_{ij}}{1+q_i} \geq 0$
 for all $i,j\in E$, we deduce by iteration from \eqref{rec} and \eqref{rec-bis} that $h_i^{(n)} \geq
 x_i\geq 0$, for any $n\in \mathbb{N}$.

 \me Let us in the other hand define for any $j$ the quantity $z_j=
 \mathbb{E}_j (e^{-\tau_\infty})$. Using $\tau_\infty =\lim_n \tau_n$, and
 $\tau_n=\sum_{k=1}^n U_k$, we deduce by monotone convergence that $z_j =
 \lim_n h_j^{(n)}$.
 
 \me If the process does not explode a.s., then $\tau_\infty =\infty$
 a.s., and $\lim_n h_i^{(n)}=z_i=0$. Since $h_i^{(n)} \geq x_i\geq 0$,
 we deduce that $x_i=0$.  It turns out that the unique nonnegative bounded
 solution of $Lx=x$ is zero.

 \me If the process explodes with positive probability, then there
 exists $i$ such that $\mathbb{P}_i(\tau_\infty<\infty)>0$. Making $n$
 tend to infinity in \eqref{rec}, we get 
 $  z_i= \sum_{j\neq i}
   {L_{ij}\over 1 + q_i} \ z_j$.
 Since $z_i>0$, $z$ is a positive
 and bounded solution of $L z = z$.

 \bi 2) Let us now apply this result to the birth and death process
 with $\lambda_0=\mu_0=0$. Then for $i\geq 1$,
 $L_{i,i+1}=\lambda_i,\ L_{i,i-1}=\mu_i,\ L_{i,i}=-(\lambda_i+\mu_i)$.
 The equation $Lx=x$ is given by $\ x_0 = 0$ and for all $n\geq 1$ by
 \begin{equation*}
   \lambda_n x_{n+1} -(\lambda_n + \mu_n) x_n + \mu_n x_{n-1} = x_n.
 \end{equation*}
 Thus, if we set $\Delta_n= x_n - x_{n-1}$, we have $\Delta_1= x_1$
 and for $n\geq 1$, $\Delta_{n+1} = \Delta_n \ {\mu_n\over \lambda_n}
 + {1\over \lambda_n} x_n$. Let us remark that, for any $n$, $\Delta_n
 \geq 0$ and the sequence $(x_n)_n$ is nondecreasing.  If
 $x_1=0$, the solution is zero. If not, we get by induction
 \begin{equation*}
   \Delta_{n+1} = {1\over \lambda_n} x_n + \sum_{k=1}^{n-1} {1\over
     \lambda_k} \ {\mu_{k+1}\over \lambda_{k+1}} \cdots {\mu_n\over
     \lambda_n} \ x_k + {\mu_1 \over \lambda_1}\cdots {\mu_n\over
     \lambda_n} \ x_1.
 \end{equation*}
Letting 
 \begin{equation*}
   r_n = {1\over
     \lambda_n} +  \sum_{k=1}^{n-1} {\mu_{k+1} \cdots  \mu_n\over \lambda_k
     \lambda_{k+1} \cdots \lambda_n} + {\mu_{1} \cdots \mu_n\over \lambda_1
     \cdots \lambda_n},
 \end{equation*} 
we deduce that
  $ r_n \ x_1 \leq \Delta_{n+1} \leq r_n \ x_n$.
 Then 
 \begin{equation*}
   x_1(1+ r_1+ \cdots r_n ) \leq x_{n+1}\leq x_1 \
   \prod_{k=1}^n(1+r_k) .
 \end{equation*}
 The boundedness of the sequence $(x_n)_n$ is thus equivalent to the
 convergence of the series $\sum_k r_k$.
\end{proof}

 \me
\begin{corollary}
  Let us consider a BD-process with birth rates $(\lambda_i)_i$. If
  there exists a constant $\lambda>0$ such that
  \begin{equation*}
    \lambda_i \leq \lambda i,\ \forall i\geq 1,
  \end{equation*}
  then the process is well defined on $\mathbb{R}_+$.
\end{corollary}

\me The proof is immediate. It turns out that the linear BD-processes and
the logistic processes are well defined on $\mathbb{R}_+$.

\bi Let us now recall under which assumption a BD-process goes to
extinction almost surely.

\me
\begin{proposition}
  \label{prop-BDextinct}
  The BD-process goes almost-surely to
  extinction if and only if 
  \begin{equation}
    \label{cond-ext}
    \sum_{k=1}^\infty {\mu_1\cdots \mu_k\over \lambda_1\cdots
      \lambda_k} = +\infty.
  \end{equation}
\end{proposition}


\begin{proof}
  Let us introduce 
    \begin{equation*}
    u_i:=\mathbb{P}(Extinction|Z_0=i) = \mathbb{P}_i(T_0<\infty),
  \end{equation*}
  which is the probability to attain $0$ in finite time, starting from
  $i$.  As before $T_0$ denotes the extinction time and $T_{I}$  the hitting time of any $I$. The Markov property yields the induction formula
  \begin{equation*}
  \lambda_i\ u_{i+1}-(\lambda_i+\mu_i)\ u_i +\mu_i \ u_{i-1}=0,\ \forall i\geq 1.
  \end{equation*}
  To resolve this equation, we firstly assume that the rates
  $\lambda_i, \mu_i$ are nonzero until some fixed level $I$ such that
  $\lambda_I=\mu_I=0$.  We set for each $i$,
 $
    u_i^{(I)}:=\mathbb{P}_i(T_0<T_I).
$
  Thus
 $    u_i = \lim_{I\to \infty} u_i^{(I)}.
  $  Defining 
  $  U_I:=\sum_{k=1}^{I-1} {\mu_1\cdots \mu_k\over
      \lambda_1\cdots\lambda_k},$
  an easy computation shows that for
  $i\in \{1,\cdots,I-1\}$, 
  \begin{equation*}
    u_i^{(I)} = (1+U_I)^{-1}
    \sum_{k=i}^{I-1} {\mu_1\cdots \mu_k\over
      \lambda_1\cdots\lambda_k}.
  \end{equation*}

  \me In particular, $u_1^{(I)} = {U_I\over 1+U_I}$.  Hence, either
  $(U_I)_I$ tends to infinity when $I\to \infty$ and any extinction
  probability $u_i$ is equal to $1$ or $(U_I)_I$ converges to a
  finite limit $U_\infty$ and for $i\geq 1$, 
  \begin{equation*}
    u_i= (1+U_\infty)^{-1} \sum_{k=i}^\infty {\mu_1\cdots \mu_k\over
    \lambda_1\cdots\lambda_k} < 1.
  \end{equation*}
 
\end{proof}

\begin{corollary}
  \begin{enumerate}
  \item The linear BD-process with rates $\lambda i$ and $\mu i$ goes
    almost surely to extinction if and only if $\lambda \leq \mu.$
  \item The logistic BD-process goes almost surely to extinction.
  \end{enumerate}
\end{corollary}

\begin{proof}
  1) If $\lambda\leq \mu$, \textit{i.e.} when the process is
  sub-critical or critical, we obtain $U_I\geq I-1$ for any $I\geq 1$.
  Then $(U_I)_I$ goes to infinity when $I\to \infty$ and the process
  goes to extinction with probability $1$. Conversely, if $\lambda>
  \mu$, the sequence $(U_I)_I$ converges to ${\mu\over \lambda-\mu}$,
  and an easy computation gives $u_i=(\lambda/\mu)^i$.

  \me 2) Here we have 
  \begin{equation}
    \label{logistique} 
    \lambda_i=\lambda i\ ;\
    \mu_i=\mu i+ ci(i-1).
  \end{equation}
  It is easy to check that \eqref{cond-ext} is satisfied.
\end{proof}

\subsection{Quasi-stationary distributions for birth and death processes}
\label{subsection:QSD-for-BD-processes-detailed-study}

We consider a BD-process $(Z_t)$ with almost sure extinction. A probability measure $\alpha$ on
$\mathbb{N}^*$ is given by a sequence $(\alpha_j)_{j\geq 1}$ of
non-negative numbers such that $\sum_{j\geq 1} \alpha_j =1$.

\noindent Our first result is a necessary and sufficient condition for such a
sequence $(\alpha_j)_{j\geq 1}$ to be a QSD for $Z$. Thereafter we
will study the set of sequences which fulfill this condition (we refer
the reader to van Doorn \cite{vanDoorn1991} for more details).
\begin{theorem}
  \label{thm:qsd-BD}
  \label{theorem:qsd-BD-process1}
  The sequence $(\alpha_j)_{j\geq 1}$ is a QSD if and
  only if

\begin{enumerate}
  \item $\alpha_j \geq 0, \  \forall j\geq 1\,$ and $\,\sum_{j\geq 1}
  \alpha_j=1$.
  \item $\forall \ j\geq 1$, 
    \begin{eqnarray}
      \label{qsd-BD}
      \lambda_{j-1}\alpha_{j-1}
      -(\lambda_j+\mu_j) \alpha_{j} +\mu_{j+1}\alpha_{j+1} &=& - \mu_1 \alpha_1 \alpha_j\,; \nonumber\\
      -(\lambda_1+\mu_1) \alpha_{1} +\mu_{2}\alpha_{2} &=& - \mu_1 \alpha_1^2.
    \end{eqnarray}
\end{enumerate}
\end{theorem}

\me
The next result follows immediately.
\begin{corollary}
\label{vando}
\me Let us define inductively the sequence of polynomials
$(H_n(x))_n$ as follows: $H_1(x)=1$ for all $x\in
\mathbb{R}$ and for $n\geq 2$,
\be \label{polynomial}  \lambda_n \ H_{n+1}(x)& =& (\lambda_n + \mu_n - x)\ H_n(x) -
\mu_{n-1}\
H_{n-1}(x)\,;\nonumber\\
\lambda_1 \ H_2(x)& =& \lambda_1+\mu_1 - x. \ee

Then, any quasi-stationary distribution $(\alpha_j)_j$
satisfies for all $j\geq 1$, 
\begin{equation*}
  \alpha_j = \alpha_1 \ \pi_j \ H_j(\mu_1
  \alpha_1),
\end{equation*}
where
\begin{equation}
  \label{def-pi} 
  \pi_1 =1 \ ;\ \pi_n = {\lambda_1\cdots
    \lambda_{n-1}\over \mu_2\cdots \mu_n}.
\end{equation}
\end{corollary}


\bigskip
\begin{proof} [Proof of Theorem
  \ref{thm:qsd-BD}.]
  
By Proposition \ref{PrQSDeigenfunction} and for a QSD $\alpha$, there
exists $\theta>0$ such that
\begin{equation*}
  \alpha L=-\theta\, \alpha,
\end{equation*}
where  $L$ is the infinitesimal generator of $Z$
restricted to $\mathbb{N}^*$. Taking the $j^{th}$ component of this
equation, we get
\begin{eqnarray*}
&&\lambda_{j-1}\alpha_{j-1}
      -(\lambda_j+\mu_j) \alpha_{j} +\mu_{j+1}\alpha_{j+1} = - \theta \,\alpha_j,\ \forall j\geq 2\\
&&  -(\lambda_1+\mu_1) \alpha_{1}
  +\mu_{2}\alpha_{2} =- \theta\, \alpha_1.
\end{eqnarray*}
Summing over $j\geq 1$, we get after re-indexing
\begin{equation*}
 0=\sum_{j\geq 1} \lambda_j \alpha_j - (\lambda_j+\mu_j)\alpha_j+\mu_{j}\alpha_j=-\theta\sum_{j\geq 1}\alpha_j+\mu_1 \alpha_1.
\end{equation*}
We deduce that $\theta=\mu_1\alpha_1$, which concludes the proof of
Theorem~\ref{theorem:qsd-BD-process1}.

\end{proof}

\bigskip \noindent
\me The study of the polynomials $(H_{n})$ has been detailed in  Van Doorn \cite{vanDoorn1991}. In particular it is shown that there  exists a non-negative number $\xi_{1}$ such that  
\begin{equation*}
  x\leq \xi_1 \
  \Longleftrightarrow\ H_n(x) >0, \ \forall n\geq 1.
\end{equation*}

\bi By Corollary \ref{vando},  $ \alpha_j = \alpha_1 \ \pi_j \ H_j(\mu_1
  \alpha_1)$.   Since for any $j$, $\alpha_j> 0$, we have $ H_j(\mu_1 \alpha_1) > 0$
  for all $j\geq 1 $ and then  
\begin{equation*}
  0 < \mu_1 \alpha_1 \leq \xi_1.
\end{equation*}

\me
We can immediately deduce from this property that if $\xi_1=0$, then
there is no quasi-stationary distribution.

\me To go further, one has to study more carefully the spectral
properties of the semi-group $(P_t)$ and the polynomials $(H_n)_n$, as
it has been done in \cite{Karlin1957}, \cite{Good1968} and
\cite{vanDoorn1991}.  From these papers, the polynomials $(H_n)_n$ are
shown to be orthogonal with respect to the spectral measure of
$(P_t)$.  In addition, it yields a tractable necessary and sufficient
condition for the existence of QSD based on the birth and death
rates. The series $(S)$ with general term $$S_n={1\over \lambda_n
  \pi_n} \sum_{i=n+1}^\infty \pi_i$$ plays a crucial role.  Remark
that $(S)$ converges if and only if $$ \sum_{n=1}^{\infty} \pi_n
\left({1\over \mu_1} + \sum_{i=1}^{n-1} {1\over \lambda_i
  \pi_i}\right)<+\infty.  $$

\me
\begin{theorem}
  \label{thm:BD-qsd}
  \label{theorem:qsd-BD-process2}
  (\cite{vanDoorn1991}, Theorems 3.2 and 4.1).
  We have the convergence
  \begin{equation*}
    \lim_{t\to \infty} \mathbb{P}_i(Z_t=j|
    T_0>t) = {1\over \mu_1}\ \pi_j\ \xi_1\ H_j(\xi_1).
  \end{equation*}
  In particular, we obtain 
  \begin{equation}
    \label{equation:xi-1}
    \xi_1=\lim_{t\rightarrow\infty} \mu_1
    \mathbb{P}_1(Z_t=1| T_0>t)
  \end{equation}

  \begin{enumerate}
       \item If $\xi_1=0$, there is no
          QSD.    
      \item  If $(S)$ converges, then $\xi_1>0$ and the Yaglom limit is the unique QSD.
      \item If $(S)$ diverges and $\xi_1\neq 0$, then there is a
        continuum of QSD, given by the one parameter family $(\hat \alpha_j(x))_{0<x\leq  \xi_1}$:
	\begin{equation*}
	  \hat \alpha_j(x)= {1\over \mu_1}\  \pi_j \ x\
	  H_j(x).
	\end{equation*}
   \end{enumerate}
\end{theorem}

\begin{remark}\upshape
\begin{enumerate}
\item Formula~\eqref{equation:xi-1} and the approximation method described
  in Section~\ref{section:simulation} allow us to deduce a simulation
  algorithm to get $\xi_1$.
\item Cases with more general birth and death processes have also
  been studied recently. Let us mention the infinite dimensional state
  space setting of Collet, Mart\'inez, M\'el\'eard and San
  Mart\'in~\cite{Collet2011}, where each individual has a type in a
  continuous state space which influences its birth and death rates,
  and mutation on the type can occur. The authors give sufficient and
  quite general conditions for the existence and uniqueness of a QSD. We
  also refer the reader to the recent work of
  van~Doorn~\cite{vanDoorn2011a}, where a transition to the state $0$
  may occur from any state. The author provides sufficient
  conditions for the existence of QSDs.
\end{enumerate}
\end{remark}
\me
\bi Let us now develop some examples.

\me 
\textbf{The linear case. } We assume $\lambda_i=\lambda\, i\ ;
\ \mu_i=\mu\, i$ and $\lambda\leq \mu$. In that case, the BD-process  is a branching process,
where each individual reproduces with rate $\lambda$ and dies with
rate $\mu$. A straightforward computation shows that the series $(S)$
diverges. 

\me Setting $f_s:k\mapsto s^k$, we get by the Kolmogorov forward equation,
\begin{equation*}
  \frac{\partial P_t f_s (1)}{\partial t}=\mu P_t f_s(0) - (\lambda+\mu)P_t f_s(1)+\lambda P_t f_s(2).
\end{equation*}
But the branching property of the process implies $P_t f_s(2)=\left(P_t f_s(1)\right)^2$, while $f_s(0)=1$ so that
\begin{equation*}
  \frac{\partial P_t f_s (1)}{\partial t}=\mu - (\lambda+\mu)P_t f_s(1)+\lambda \left(P_t f_s(1)\right)^2.
\end{equation*}
Setting $m=2\frac{\lambda}{\lambda+\mu}$, we deduce that for $s<1$, 
\begin{equation*}
  P_t f_s(1)=1-\frac{2(1-s)(2 m-1)}{(ms+m-2)e^{-(\lambda+\mu)(m-1)t}+(1-s)m}.
\end{equation*}
In particular, we deduce that the generating function $F_t:s\mapsto
\E(s^{Z_t}|Z_t>0)$ of $Z_t$ conditioned to $Z_t>0$ converges when $t$
  goes to infinity:
\begin{equation*}
 F_t(s)=\frac{P_t f_s(1)-P_t f_0(1)}{1-P_t
   f_0(1)}\xrightarrow[t\rightarrow\infty]{} \frac{
     (\lambda-\mu) s}{\lambda s-\mu}.
\end{equation*}
We deduce that the Yaglom limit of $Z$ does not exist if $\lambda=\mu$
and is given by the geometric distribution with parameter
$\frac{\lambda}{\mu}$ if $\lambda<\mu$:
\begin{equation*}
  \alpha_k=\left(\frac{\lambda}{\mu}\right)^{k-1}\left(1-\frac{\lambda}{\mu}\right).
\end{equation*}
\me An easy computation yields $\xi_1 = \mu - \lambda$, since by
\eqref{equation:xi-1}, $\alpha_1=\frac{\xi_1}{\mu}$. But the series
$(S)$ diverges so that for $\lambda< \mu$, $\xi_1>0$ and there is an
infinite number of QSD. If $\lambda= \mu$, $\xi_1=0$ and there is no
QSD.

\me

\bi \textbf{The logistic case.} We assume $\lambda_i=\lambda i\ ;
\ \mu_i=\mu i + c i(i-1)$. Because of the quadratic term, the
branching property is lost and we can not compute the Yaglom limit as
above. Therefore, we have no other choice than to study the
convergence of the series $(S)$.

\me We have 
\begin{eqnarray*}
  \sum_{i=n+1}^\infty \pi_i
  &\leq & \sum_{i=n+1}^\infty \left({\lambda\over c}\right)^{i-1}
      {1\over i!} = \sum_{p=0}^\infty \left({\lambda\over
	c}\right)^{n+p} {1\over (n+p+1)!}\\
  &\leq& \left({\lambda\over c}\right)^{n}{1\over
	(n+1)!}\sum_{p=0}^\infty \left({\lambda\over c}\right)^{p} {1\over
	p!} = \left({\lambda\over c}\right)^{n}{1\over (n+1)!}\
      e^{\lambda\over c},
\end{eqnarray*}
since ${(n+1)!\over (n+p+1)!} \leq {1\over p!}$. Thus as ${1\over
\pi_n}\leq C \left({c\over \lambda}\right)^{n-1}\ n!\ $, we get
\begin{equation*}
  {1\over \lambda_n \pi_n} \sum_{i=n+1}^\infty \pi_i \leq
  {C\over c} {1\over n(n+1)}\ e^{\lambda\over c}.
\end{equation*}
Hence the series converges. Thus the Yaglom limit exists and is the
unique quasi-stationary distribution.

\me One can obtain substantial qualitative
information by looking closer to the jump rates of the process.
 For instance, $(\lambda -\mu)/c$ is a key value for the
process. Indeed, given a population size $i$, the expectation of the
next step is equal to $\frac{i(\lambda-\mu-c(i-1))}{\lambda i+\mu i +
  c i(i-1)}.$ Then the sign of this expectation depends on the
position of $i-1$ with respect to $(\lambda-\mu)/c$:
\begin{itemize}
\item If $i\leq (\lambda-\mu)/c+1$, then the expectation of the next step will be positive.
\item If $i=(\lambda-\mu)/c+1$, then it will be $0$.
\item If $i>(\lambda-\mu)/c+1$, then it will be negative.
\end{itemize}
We deduce that the region around $(\lambda-\mu)/c$ is stable: it plays
the role of a typical size for the population and we expect that the
mass of the Yaglom limit is concentrated around it.  The value
$(\lambda-\mu)/c$ is called (by the biologists) the charge capacity of
the logistic BD-process with parameters $\lambda$, $\mu$ and $c$.  In
the next section, we will consider large population processes, which
means logistic BD-processes with large charge capacity.

\begin{example}\upshape
  We develop now a numerical illustration of the logistic BD-process
  case. Across the whole example, the value of the charge capacity
  $\frac{\lambda-\mu}{c}$ is fixed, arbitrarily chosen equal to~$9$.

  \me In order to illustrate the concept of charge capacity, we
  represent in Figure \ref{figExample3_path} a random path of a logistic birth and death
  process with initial size $Z_0=1$ and with parameters $\lambda=10$,
  $\mu=1$ and $c=1$.
  \begin{figure}
    \includegraphics[width=13cm]{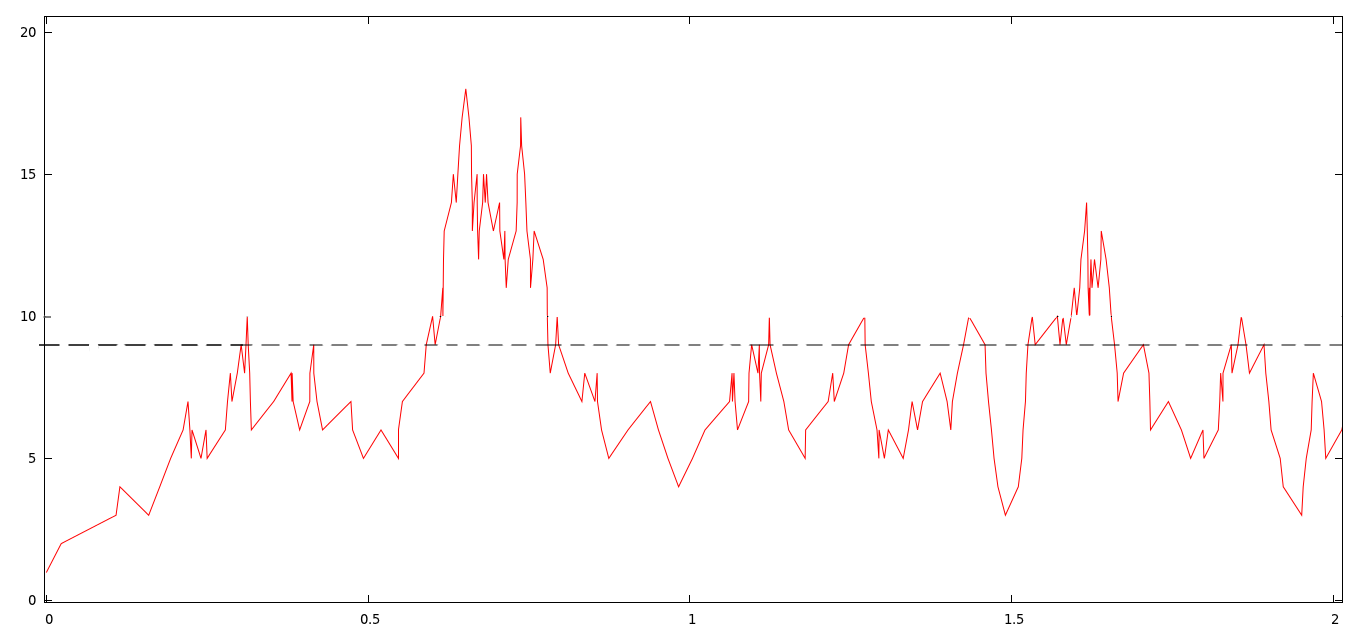}
    \caption{\label{figExample3_path} Example 3. A random path of a logistic birth and death
  process with initial size $Z_0=1$ and with parameters $\lambda=10$,
  $\mu=1$ and $c=1$}
  \end{figure}
  We observe that the process remains for long times in a region
  around the charge capacity. Moreover, we remark that the process
  remains mainly below the charge capacity; this is because the jumps
  rate are higher in the upper region, so that it is less stable than
  the region below the charge capacity.

  \me Let us now compare the Yaglom limits (numerically
  computed using the approximation method presented in
  Section~\ref{section:simulation}) of two different logistic BD
  processes whose charge capacities are equal to $9$ (see Figure
  \ref{example3_yaglom}):
  \begin{itemize}
  \item[(a)] $Z^{(a)}$, whose parameters are $\lambda=10$, $\mu=1$ and $c=1$,
  \item[(b)] $Z^{(b)}$, whose parameters are $\lambda=10$, $\mu=7$ and $c=1/3$.
  \end{itemize}
\begin{figure}
\centering
\subfloat{\includegraphics[width=13cm]{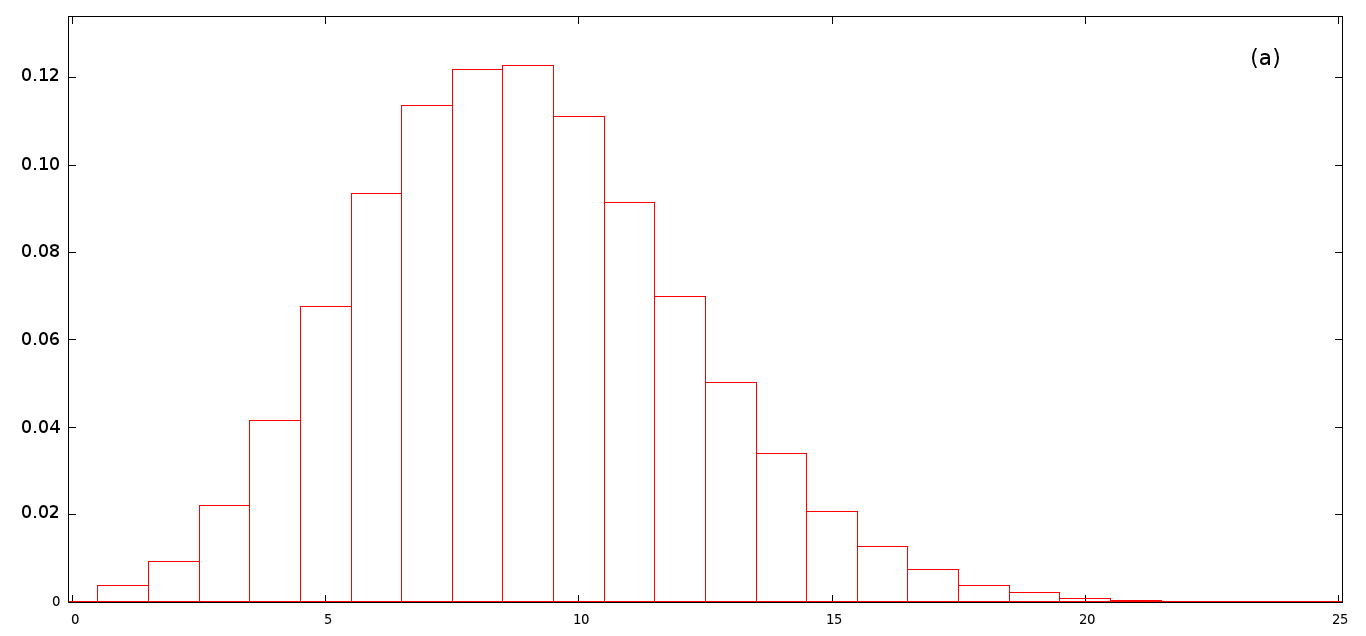}}\\
\subfloat{\includegraphics[width=13cm]{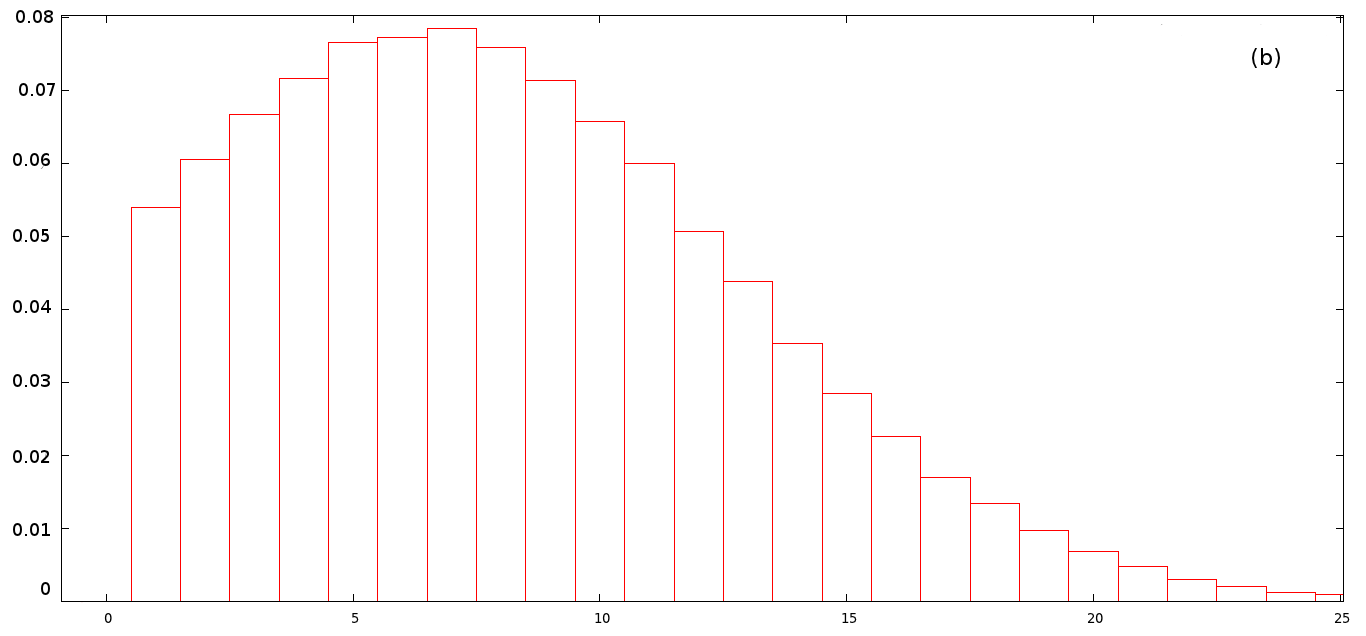}}
\caption{Example 3. The Yaglom limits of two logistic birth and death
  processes with the same charge capacity $\frac{\lambda-\mu}{c}=9$:
  (a) $\lambda=10$, $\mu=1$ and $c=1$;
  (b) $\lambda=10$, $\mu=7$ and $c=1/3$.
\label{example3_yaglom}}
\end{figure}

\me We observe that the Yaglom limits of $Z^{(a)}$ and $Z^{(b)}$ are
supported by a region which is around the charge capacity. We also
remark that the Yaglom limit of the process $Z^{(b)}$ has a more flat
shape than the Yaglom limit of $Z^{(a)}$. This is because the
competition parameter of $Z^{(b)}$ is small in comparison with the
birth and death parameters, so that the drift toward the charge
capacity is small too, both above and below the charge capacity.

\me We compute now the distance between the conditioned distribution
and the Yaglom limit for the two processes $Z^{(a)}$ and $Z^{(b)}$ for
different values of the initial state, namely $Z_0=1$, $Z_0=10$ and
$Z_0=100$. The numerical results are represented in Figure~\ref{example3_distance}.
\begin{figure}
\centering
\subfloat{\includegraphics[width=13cm]{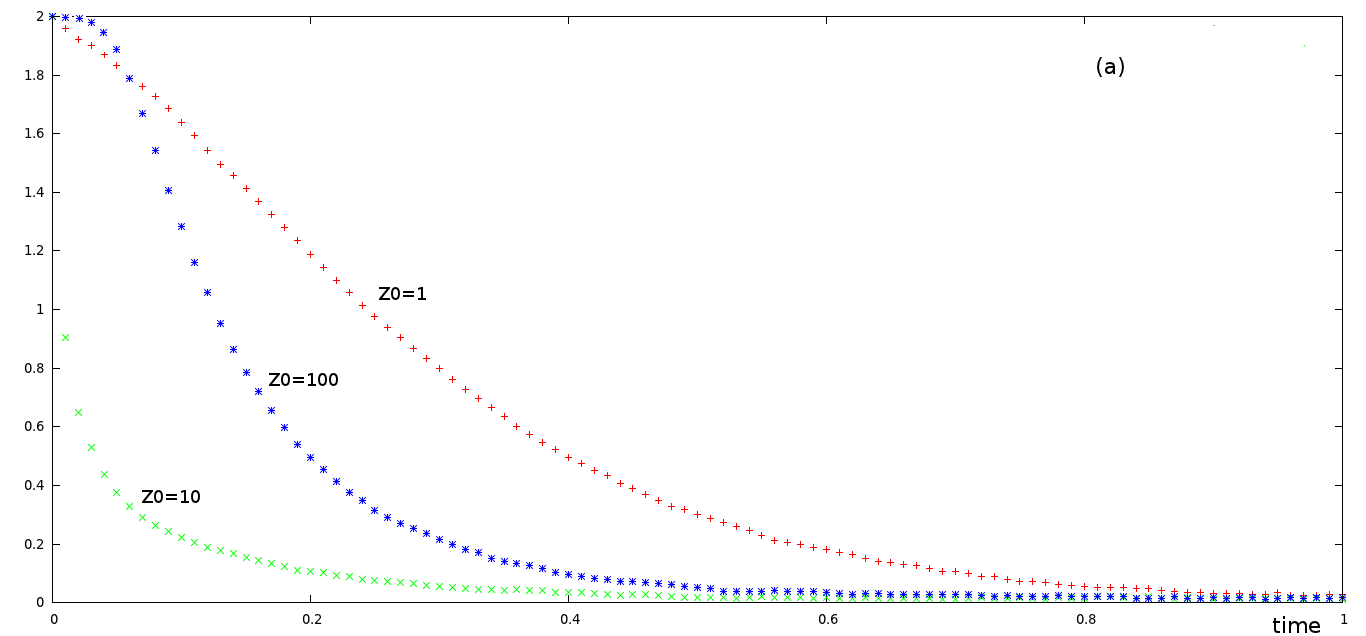}}\\
\subfloat{\includegraphics[width=13cm]{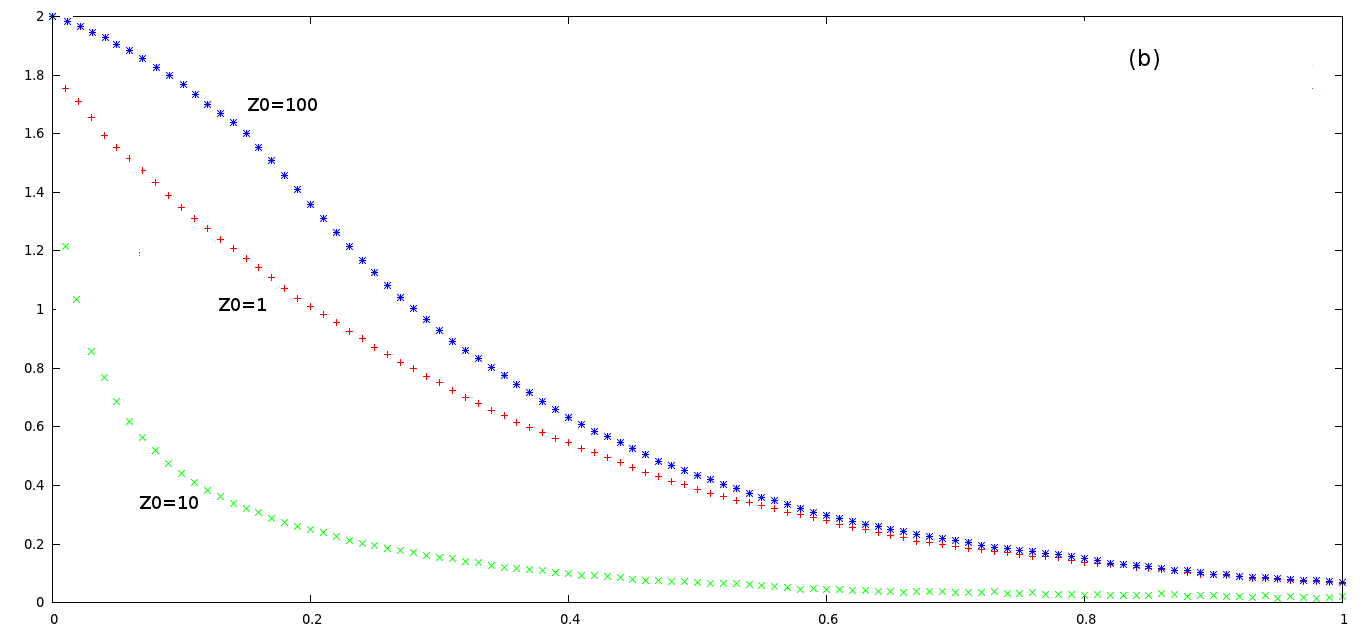}}
\caption{Example 3. Evolution of the distance between the conditioned
  distribution and the Yaglom limits of two logistic birth and death
  processes with the same charge capacity $\frac{\lambda-\mu}{c}=9$:
  (a) $\lambda=10$, $\mu=1$ and $c=1$; (b) $\lambda=10$, $\mu=7$ and
  $c=1/3$.
\label{example3_distance}}
\end{figure}
We observe a strong dependence between the speed of
convergence and the initial position of the processes. In the case of
$Z^{(a)}$, it only takes a very short time to the process starting
from $100$ to reach the charge capacity, because the competition
parameter is relatively high and so is the drift downward the charge
capacity. On the contrary, in the case of $Z^{(b)}$, it takes a longer
time for the process to come back from $100$ to the charge capacity,
so that the speed of convergence to the Yaglom limit is slow. In both
cases, the convergence to the Yaglom limit happens very fast when
starting from the value $10$, because it is near the charge capacity.

\end{example}

\section{The logistic Feller diffusion process}
\label{section:the-logistic-feller-diffusion-process}
\subsection{A large population model}
\label{subsection:a-large-population-model}

We are now rescaling logistic birth and death processes with a parameter $K\in \mathbb{N}^*$ modeling a
large population with small individuals, \textit{i.e.} a population
with a large initial size of order $K$  and a large charge capacity assumption. The individual's weights (or biomasses) are  assumed to be equal to ${1\over K}$ and we 
study the limiting behavior of the total biomass process $({Z_t^K\over K}, t\geq 0)$ when $K$
tends to infinity, $Z^K_{t}$ being the population's size  at time $t$.  In what follows, $\lambda$, $\mu$ and $c$ are
fixed positive constants.
 
\noindent In Subsection~\ref{subsubsectionLogisticEquation}, individual birth
and death rates are assumed to be constant and the competition rate
depends linearly on the individual biomass ${1\over K}$. 
In Subsection~\ref{subsubsectionLogisticFellerDiffusion}, we
investigate the qualitative differences of evolutionary dynamics
across populations with allometric demographics: life-lengths and
reproduction times are assumed to be proportional to the individual's
weights. 

\me
In both cases, the charge capacity of
$(Z^K)$ will be $(\lambda-\mu)K/c$.

\subsubsection{Convergence to the logistic equation}
\label{subsubsectionLogisticEquation}
\label{subsubsection:the-logistic-equation}
Given a parameter $K$  scaling the population's size, we consider
the logistic BD-process $Z^K$ with birth, death and competition
parameters $\lambda$, $\mu$ and $c/K$ respectively. We assume that
the initial value of $Z^K$ is of order $K$, in the sense that there
exists a non-negative real random variable $X_0$ such that
\begin{equation*}
   {Z_0^K \over K} \xrightarrow[K \to \infty]{} X_0\quad \hbox{ with }
   \ \mathbb{E}(X_0^3) < +\infty.
\end{equation*}
We consider the total biomass process  defined by $X^K=Z^K/K$ for all $K\geq 1$ and are interested in the limit of $X^K$ when $K\to \infty$. 
 The transitions
of the process $(X^K_t, t\geq 0)$ are the following ones: 
\begin{eqnarray*}
  \frac{i}{K} &\rightarrow& \frac{i+1}{K} \quad \text{ with rate }
  \    \lambda i = \lambda K \frac{i}{K};\\
  \frac{i}{K} &\rightarrow& \frac{i-1}{K} \quad \text{ with rate }
 \     \mu i + \frac{c}{K}\ i(i-1)
      = K \frac{i}{K} \left(\mu+c( \frac{i}{K} - \frac{1}{K})\right).
\end{eqnarray*}


\begin{theorem}
  \label{thm:cvLogistic}
  \label{theorem:convergence-logistic-equation}
  Assume that $X_0$ is a positive number $x_0$. Then, the process $(X^K_t, t\geq 0)$ converges in
  law in $\mathbb{D}([0,T], \mathbb{R}_+)$ to the unique continuous
  (in time) deterministic function solution of
  \begin{equation*}
    x(t) =x_0 + \int_0^t(\lambda - \mu -c x(s))x(s) ds.
  \end{equation*}
\end{theorem}

\begin{remark}\upshape
 The function $x$ is thus solution of the ordinary differential
 equation
 \begin{equation}
   \label{logistic} \dot{x} = (\lambda - \mu) x - c x^2\ ;\ x(0)=x_0,
 \end{equation}
 called logistic equation. This equation has been historically
 introduced as the first macroscopic model describing populations
 regulated by competition (cf. \cite{Verhulst1838}). In
 Theorem~\ref{theorem:convergence-logistic-equation} above, it is obtained
 as the limit of properly scaled stochastic jump models.
 
 \noindent The function $x$ solution of \eqref{logistic} hits $0$ in
 finite time if $\lambda<\mu$, while it remains positive forever if
 $\lambda>\mu$, converging in the long term to its charge capacity
 $\frac{\lambda-\mu}{c}$. Thus at this scale extinction does not
 happen.
\end{remark}

\me
\begin{proof}[Proof of Theorem~\ref{theorem:convergence-logistic-equation}]
The Markov process $(X^K_t, t\geq 0)$ is well defined and its
infinitesimal generator is given, for any measurable and bounded
function $\phi$, by 
\begin{equation}
  \label{gen-NMlogistic} L_K \phi(x) =  \lambda K x\, \left(\phi(x+{1\over
    K})-\phi(x)\right) +  K (\mu x + c x (x-{1\over K}))\, \left(\phi(x - {1\over
    K})-\phi(x)\right).
\end{equation}
 Hence, by  Dynkin's theorem (\cite{Ethier1986} Prop. IV-1.7), the process
\begin{equation}
  \label{dynkin}
  \phi(X^K_t) - \phi(X^K_0) - \int_0^t L_K\phi(X^k_s) ds
\end{equation}
defines a local martingale, and a martingale, as soon as each
term in \eqref{dynkin} is integrable.  In particular, taking
$\phi(x)=x$ gives that $(X^K_t, t\geq 0)$ is a semimartingale
and there exists a local martingale $M^K$ such that 
\begin{equation}
  \label{semimart} X^K_t = X^K_0 + M^K_t + \int_0^t
  X^K_s\left(\lambda - \mu - c\left(X^K_s -{1\over K}\right)\right)
  ds.
\end{equation}
Since $x_0$ is deterministic and using a localization argument, we
deduce that $ \mathbb{E}(\sup_{t\leq T} (X^K_t)^2)<\infty$. Moreover,
taking $\phi(x)=x^2$ applied to \eqref{dynkin}, and comparing with
It\^o's formula applied to $(X^K)^2$ prove that $(M^K)$ is a
square-integrable martingale with quadratic variation process
\begin{equation}
  \label{variationquadratique} 
  \langle M^K\rangle_t = {1\over K} \int_0^t \left(\lambda + \mu + c\left(X^K_s -
  {1\over K}\right) \right)X^K_s ds.
\end{equation}

\bi
Let us  now study the
convergence in law of the sequence $(X^K)$, when $K$ tends to
infinity. For any $K$, the law of $X^K$ is a probability measure
on the trajectory space $\mathbb{D}_T=\mathbb{D}([0,T], \mathbb{R}_+)$, namely the  Skorohod space of
left-limited and right-continuous functions from $[0,T]$ into
$\mathbb{R}_+$, endowed with the Skorohod topology.  This topology   makes $\mathbb{D}_T$ a Polish
state, that is a metrizable complete and separable space, which is
not true if $\mathbb{D}_T$ is endowed with the uniform topology.
See Billingsley \cite{Billingsley1968} for details.

\me
The proof of Theorem \ref{thm:cvLogistic} is obtained by a compactness-uniqueness
argument. The uniqueness of the solution of \eqref{logistic} is immediate.

\me By a natural coupling, one may bound the birth and death process
$X^K$ stochastically from above by the Yule process $Y^K$ started from
$x_0$, which jumps from $x$ to $x+{1\over K}$, at the same birth times
than $X^K$. One easily shows that $\sup_K\mathbb{E}(\sup_{t\leq T}
(Y^K_t)^3) <\infty$ and thus 
\begin{equation*} \label{moments}
\sup_K\mathbb{E}(\sup_{t\leq T} (X^K_t)^3) <\infty.
\end{equation*}
From this uniform estimate, we deduce the uniform
tightness of the laws of $X^K$ (as probability measures on
$\mathbb{D}_T$), using the Aldous criterion (cf. Aldous
\cite{Aldous1978}, Joffe-M\'etivier \cite{Joffe1986}). Then, by Prokhorov's
Theorem, the compactness of the laws of $(X^K)$ follows. Since $\, \sup_{t\leq
  T} |X^K_t - X^K_{t^-}| \leq {1\over K}$ and the function $x
\mapsto \sup_{t\leq T} |x_t - x_{t^-}|$ is continuous on
$\mathbb{D}_T$,  each limiting value (in law) of
the sequence $(X^K)$ will be a pathwise continuous process. In addition
using \eqref{variationquadratique} and \eqref{moments}, it can be shown
that $ \lim_{K\to \infty} \mathbb{E}(\langle M^K\rangle_t) = 0$.
Then, the random fluctuations disappear when $K$ tends to infinity and the
limiting values are deterministic functions. Now it remains to show
that these limiting values are solutions of \eqref{logistic}, which
can be done similarly to the proof of Theorem
\ref{thm:cvFellerLogistic} (4) stated below.
\end{proof}

\subsubsection{The logistic Feller diffusion process}
\label{subsubsectionLogisticFellerDiffusion}
\label{subsubsection:the-logistic-Feller-diffusion-process}

In this section, we study the logistic BD-processes $Z^K$ with birth
and death rates given by $\gamma\, K+\lambda$ and $\gamma\, K+\mu$
respectively. Here $\gamma$, $\lambda$ and $\mu$ are still positive
constants. We assume that the competition parameter is given by
$c/K$, so that the charge capacity of $Z^K$ is still
$(\lambda-\mu)K/c$.

\begin{remark}
This BD-process $(Z^K_{t})_{t}$ can also be interpreted as a
time-rescaled BD-process $Y^K_{Kt}$, whose birth, death and
competition parameters are given by $\gamma+\lambda/K$,
$\gamma+\mu/K$ and $c /K^2$ respectively, that is a
critical
BD-process with small pertubations.
\end{remark}

 \bi We are  considering as in Section
 \ref{subsubsection:the-logistic-equation} the sequence of processes
 $X^K$ defined for all $t\geq 0$ by
$
  \, X^K_t=\frac{Z^K_{t}}{K}$.

\me The transitions of the process $(X^K)$ are given by
\begin{eqnarray}
  \label{transitionsFellerlogistic}
        \frac{i}{K} &\rightarrow& \frac{i+1}{K} \ \text{ with rate }
                                \  \gamma K i + \lambda i \\ 
        \frac{i}{K} &\rightarrow& \frac{i-1}{K}\  \text{ with rate }
                                 \  \gamma K i + \mu i + \frac{c}{K} i (i-1).
                                  \nonumber
\end{eqnarray}
Formula \eqref{semimart} giving the semi-martingale decomposition of
$X^K$ will stay true in this case with another square integrable martingale
part $N^K$ such that

 \begin{equation*}
   \langle N^K\rangle_t
   = {1\over K} \int_0^t (2 \gamma K + \lambda + \mu + c\left(X^K_s -
   {1\over K}\right) X^K_s ds.
 \end{equation*} 

\noindent One immediately observes that the expectation of this
quantity does not tend to zero as $K$ tends to infinity. Hence the
fluctuations will not disappear at infinity and the limit will stay
random. Let us now state the convergence  theorem.

\begin{theorem}
  \label{thm:cvFellerLogistic}
  \label{theorem:convergence-FellerLogistic}
Assume $\gamma, c, \lambda, \mu>0$ and $\lambda>\mu$.

\noindent  i)  Assume that the sequence $(X^K_{0})_{K}$ converges in law to $X_{0}$  with $\mathbb{E}(X_0^3)<\infty$. Then the sequence of processes $(X^K)_{K}$ with transitions
  \eqref{transitionsFellerlogistic}  converges in law in ${\cal
    P}(\mathbb{D}_T)$ to the continuous process $X$, defined as the unique
  solution of the stochastic differential equation
  \begin{equation}
    \label{eq:Fellerlogistic}
    d X_t = \sqrt{2 \gamma X_t} dB_t +
    \left((\lambda-\mu)X_t - c X_t^2\right) dt; \  X_0\in ]0,+\infty[,
  \end{equation}
  where $(B_t)_{t\in[0,+\infty[}$ is a standard Brownian motion.
  
 \me ii) Let  us introduce for each $y\geq
  0$ the stopping time
\begin{equation}
\label{tempsd'entree}
  T_y=\inf\{t \in \mathbb{R}_+, X_t=y\}.
\end{equation}
For any $x\geq
  0$, we get
 $$\mathbb{P}_x(T_0<\infty)=1.$$
\end{theorem}

\noindent When $c=0$, Equation \eqref{eq:Fellerlogistic} defines the
super-critical Feller diffusion process and will explode with positive
probability. In the general case where $c\neq 0$, it defines the
so-called logistic Feller diffusion process following the terminology
introduced by Etheridge \cite{Etheridge2004} and Lambert
\cite{Lambert2005}. Let us remark that the quadratic term driven by
$c$ regulates the population size, which fluctuates until it attains the absorbing point
$0$. %
  Theorem \ref{thm:cvFellerLogistic} shows that an accumulation of a
  large amount of birth and death events may create stochasticity, often called by biologists ecological drift or
  demographic stochasticity. Contrarily to the  previous case (Theorem \ref{thm:cvLogistic}), the limiting process suffers extinction almost surely.

\begin{proof}
As for Theorem \ref{thm:cvLogistic}, the proof is
based on a uniqueness-compactness argument.

\me (1)  The uniqueness of the solution of \eqref{eq:Fellerlogistic} follows from a general existence and pathwise uniqueness result  in 
Ikeda-Watanabe \cite{Ikeda1989} Section IV-3 or  Karatzas-Shreve
\cite{Karatzas1988}. 
For
a stochastic differential equation
\begin{equation*}
  d X_t =\sigma( X_t) dB_t + b(X_t) dt,
\end{equation*}
with $\sigma$ and $b$ smooth enough, the existence and pathwise
uniqueness are determined thanks to the following scale functions: for
$x>0$,
\begin{eqnarray} 
  \label{lambda}Q(x) &=& -
  \int_1^x {2b(y)\over \sigma^2(y)}dy\ ;
  \  \Lambda(x)=\int_1^x e^{Q(z)} dz\ ;\nonumber\\
  \kappa(x) &=& \int_1^x e^{Q(y)} \left(\int_1^y e^{-Q(z)}
  dz\right)dy. \label{kappa}
\end{eqnarray}
  More precisely, it is proved that
  (i) $\forall x>0$,
  $\mathbb{P}_x(T_0<T_\infty)
  = 1$
  and
  (ii) $\Lambda(+\infty) = + \infty\ ;\ \kappa(0^+)<+\infty$
  are equivalent. 
  \me
  In that case, pathwise uniqueness follows,
  and then uniqueness in law.

  \noindent In our situation, the coefficients are given by
  \begin{equation*}
    \sigma(x) = \sqrt{2\gamma x} \ ;\ b(x)=(\lambda-\mu)x- cx^2,
  \end{equation*}
  so that the functions $\Lambda$ and $\kappa$ satisfy (ii). Thus the
  SDE \eqref{eq:Fellerlogistic} has a pathwise unique solution which
  reaches $0$ in finite time almost surely.

  \me (2) Let us assume that $\mathbb{E}(X_0^3)<\infty$ and 
  prove that
    $\sup_K \mathbb{E}(\sup_{t\leq T}
    (X^K_t)^2)<\infty.$
  The infinitesimal generator of $X^K$ is given by 
  \begin{eqnarray}
  \label{gen:saut} \tilde{L}_K\phi(x)&=& (\gamma K x + \lambda x)K\, \left(\phi(x+{1\over
    K})-\phi(x)\right)
  \nonumber\\
  &&+ (\gamma K x + \mu x+
  cx(x-{1\over K}))K\,  \left(\phi(x-{1\over K})-\phi(x)\right).
  \end{eqnarray}
  With $\phi(x)=x^3$, we obtain that 
  \begin{eqnarray*}
  (X^K_t)^3 &=& X_0^3 + M^K_t + \int_0^t \gamma K^2
  X^K_{s}\left[\left(X^K_{s}+{1\over K}\right)^3 -
    \left(X^K_{s}-{1\over K}\right)^3 -
    \left(X^K_{s}\right)^3\right] ds\\
  +&& \int_0^t  \lambda K X^K_{s} \left[\left(X^K_{s}+{1\over
      K}\right)^3  - \left(X^K_{s}\right)^3\right]ds\\
  &&+ \int_0^t  (\mu+ c(X^K_s -1)) K X^K_{s}
  \left[\left(X^K_{s}-{1\over K}\right)^3 -
    \left(X^K_{s}\right)^3\right]ds,
  \end{eqnarray*}
 where $M^K$ is a local martingale.
  Using a standard localization argument and  
  $$\left(X^K_{s}+{1\over K}\right)^3 -
  \left(X^K_{s}-{1\over K}\right)^3 - \left(X^K_{s}\right)^3 = 6\
       {X^K_{s}\over K^2},$$ we get
       \ben\mathbb{E}((X^K_t)^3) \leq \mathbb{E}(X_0^3)+ C\int_0^t
       \mathbb{E}((X^K_s)^3) ds,\een where $C$ is independent of $K$. 
       Gronwall's lemma yields  \be \label{unif3} \sup_{t\leq T}
       \sup_K \mathbb{E}(|X^K_t|^3)<\infty.\ee
       Now, thanks to \eqref{unif3} and  Doob's inequality, we deduce from the
       semi-martingale decomposition of $(X^K_t)^2$ (obtained using
       $\phi(x)=x^2$), that \be\label{uniformite}  \sup_K
       \mathbb{E}(\sup_{t\leq T}|X^K_t|^2)<\infty.\ee

\me (3) As previously, the uniform tightness of the laws of $(X^K)$ is obtained from  \eqref{uniformite} and the Aldous criterion
\cite{Aldous1978}. Therefore, the sequence of laws is relatively compact and
it remains to characterize its limit values.

\me (4) As in the proof of Theorem \ref{thm:cvLogistic}, we remark
that the limiting values only charge the set of continuous
trajectories, since $\sup_{t\leq T}|\Delta X^K_t| \leq {1\over
  K}$. Let $\mathbb{Q}\in {\cal P}(C([0,T], \mathbb{R}_+))$ be a
limiting value of the sequence of laws of the processes $X^K$. We will
identify $\mathbb{Q}$ as the (unique) law of the solution of
\eqref{eq:Fellerlogistic} and the convergence will be proved. Let us
denote $C_T=C([0,T], \mathbb{R}_+)$ and define, for $\phi\in C_b^2$
and $t>0$, the function
\begin{eqnarray*}
 \psi_t: C_T
\ &\rightarrow & \mathbb{R}\\ X &\mapsto& \phi(X_t)-\phi(X_0) -
\int_0^t \left(\gamma X_s \phi''(X_s) + ((\lambda-\mu)X_s -
cX_s^2)\phi'(X_s)\right) ds,
\end{eqnarray*}
which is continuous $\mathbb{Q}$-a.s.. Let us show first that
the process $(\psi_t(X))_t$ is a $\mathbb{Q}$-martingale.

\me For $x\in \mathbb{R}_+$, we define
\begin{equation*} 
L\phi(x) = \gamma x
\phi''(x) + ((\lambda-\mu)x - c x^2)\phi'(x).
\end{equation*}
Using Taylor's expansion, we immediately get (with $\tilde{L}_K$
defined in \eqref{gen:saut})
\begin{eqnarray}
  \label{comp-gen} |\tilde{L}_K\phi(x)
  - L\phi(x)| &=& \gamma K^2\ x\ \left|\phi(x+{1\over K})+\phi(x-{1\over
    K})-2\phi(x)- {1\over K^2}\phi''(x)\right|\nonumber\\ && +
  \lambda\ K\ x\left|\phi(x+{1\over K})-\phi(x)-{1\over
    K}\phi'(x)\right|\nonumber\\ &&+ K\ (\mu x+
  cx(x-{1}))\ \left|\phi(x-{1\over K})-\phi(x)+{1\over
    K}\phi'(x)\right|\nonumber\\ &\leq& {C\over K}\ (x^2+1),
\end{eqnarray}
where $C$ doesn't depend on $x$ and $K$. By \eqref{uniformite}, we
deduce that $\ \mathbb{E}\left(|\tilde{L}_K\phi(X^K_t) -
L\phi(X^K_t)|\right)$ tends to $0$ as $K$ tends to infinity, uniformly
for $t\in[0,T]$.

\me For $s_1<\cdots < s_k<s<t$, for $g_1, \cdots, g_k\in C_b$, we
introduce the function $H$ defined on the path space by 
\begin{equation*}
 H(X)=
g_1(X_{s_1})\cdots g_k(X_{s_k}) \left(\psi_t(X)-\psi_s(X)\right).
\end{equation*}
Let us show now that \be \label{prop-mart}
\mathbb{E}_\mathbb{Q}(H(X))=0,\ee which will imply that
$(\psi_{t}(X))_{t}$ is a $\mathbb{Q}$-martingale.

\me By construction, $\psi^K_t(X^K)= \phi(X^K_t)-\phi(X_0) - \int_0^t
\tilde{L}_K\phi(X^K_s) ds\ $ defines a martingale, then 
\begin{equation*}
\mathbb{E}\left[g_1(X^K_{s_1})\cdots g_k(X^K_{s_k})
  \left(\psi^K_t(X^K)-\psi^K_s(X^K)\right)\right]=0.
\end{equation*}
In another
way, this quantity is equal to
\begin{eqnarray*} &&
  \mathbb{E}\left[g_1(X^K_{s_1})\cdots g_k(X^K_{s_k})
    \left(\psi^K_t(X^K)-\psi^K_s(X^K)-\psi_t(X^K)+\psi_s(X^K)\right)\right]\\ &&
+\mathbb{E}\left[g_1(X^K_{s_1})\cdots g_k(X^K_{s_k})
  \left(\psi_t(X^K)-\psi_s(X^K)\right) - g_1(X_{s_1})\cdots
  g_k(X_{s_k})\left(\psi_t(X)-\psi_s(X)\right)\right]\\ &&+
\mathbb{E}\left[g_1(X_{s_1})\cdots
  g_k(X_{s_k})\left(\psi_t(X)-\psi_s(X)\right)\right] .
\end{eqnarray*}

\noindent The first term is equal to $
\mathbb{E}\left[g_1(X^K_{s_1})\cdots g_k(X^K_{s_k})\int_0^t
\left(\tilde{L}_K \phi(X^K_s)-L\phi(X^K_s)\right) ds\right]$  and tends to $0$ by \eqref{uniformite}  and 
\eqref{comp-gen}.

\noindent The second term is equal to $\mathbb{E}(H(X^K)-H(X))$.
The function $X\mapsto H(X)$ is continuous and since $H(X)\leq
C\left(1+\int_s^t (1+X_u^2) du\right)$, it is also
uniformly integrable by \eqref{unif3}.  This leads the second term to  tend to $0$ as
$K$ tends to infinity.

\noindent Therefore, it turns out that \eqref{prop-mart} is fulfilled and  the process
$\psi_t(X)=\phi(X_t) - \phi(X_0) - \int_0^t L\phi(X_s) ds$ is a $\mathbb{Q}$-martingale.

\me By \eqref{uniformite} and taking $\phi(x)=x$ leads to
$X_t=X_0+M_t+\int_0^t ((\lambda-\mu)X_s -cX_s^2)ds$, where $M$ is a
martingale. Taking $\phi(x)=x^2$ on the one hand and applying It\^o's
formula for $X_t^2$ on the other hand allow us to identify
\begin{equation*}
 \langle M\rangle_t=\int_0^t 2 \, \gamma\ X_s\ ds.
\end{equation*}
By the
 representation
theorem proved in \cite{Karatzas1988} Theorem III-4.2 or in \cite{Ikeda1989},
there exists a Brownian
motion $B$ such that 
\begin{equation*}
  M_t=\int_0^t \sqrt{2\gamma X_s}\ dB_s.
\end{equation*}
That concludes the proof. 

\end{proof}

\subsection{QSD for logistic Feller diffusion processes}
\label{subsection:qsd-for-logistic-Feller-diffusions}
\subsubsection{Statement of the  results}
We are now interested in studying the quasi-stationarity for  the logistic Feller
diffusion process  solution of the equation
\begin{equation*}
\label{fellz}
  dZ_t=\sqrt{Z_t} d B_t +(rZ_t-cZ_t^2)dt, \quad Z_0>0,
\end{equation*}
where the Brownian motion $B$ and the initial state $Z_0$ are given,
and $r$ and $c$ are assumed to be positive. (We have assumed that
$\gamma=1/2$). The results and proofs that are presented in  Section \ref{subsection:qsd-for-logistic-Feller-diffusions}
have been obtained by Cattiaux, Collet, Lambert, Mart\'inez, M\'el\'eard
and San Mart\'in \cite{Cattiaux2009}.

\bi
 Let us firstly state the main theorem of this part.
\begin{theorem}
  \label{thm:feller-qsd}
  \label{theorem:qsd-logistic-Feller}
  Assume that $Z_0$, $r$ and $c$ are positive. Then the Yaglom limit
  of the process $Z$ exists and is a QLD for $Z$ starting from any
  initial distribution. As a consequence, it is the unique QSD of $Z$.

  
\end{theorem}


\begin{remark}\upshape
\bi 1)  The theory studying the quasi-stationary distributions for
one-dimen\-sional diffusion processes started with Mandl
\cite{Mandl1961} and has been developed by many authors. See in
particular \cite{Collet1995}, \cite{Martinez2004},
\cite{Steinsaltz2007}, \cite{Knobloch2010}. Nevertheless in most of
the papers, the diffusion and drift coefficients are regular and the
"Mandl's condition" $\kappa(+\infty)=\infty$ (see \eqref{kappa}) is
assumed. This condition is not satisfied in our case because of  the degeneracy of the diffusion and the unboundedness of the drift coefficient. 

\me
2) Theorem \ref{thm:feller-qsd}
 differs from the results obtained in case of 
drifts  going slower to infinity. For example, Lambert
\cite{Lambert2007} proves that if $c=0$ and $r\leq 0$, then either
$r=0$ and there is no QSD, or $r<0$ and there is an infinite number of
QSD. Lladser and San Mart\'in \cite{Lladser2000} show that in the case
of the Ornstein-Uhlenbeck process
$
dY_t = dB_t - Y_t dt$, killed at $0$, there is also a continuum of QSD.  In the logistic Feller diffusion situation as in the logistic
BD-process, the uniqueness comes from the quadratic term $c\,X_t^2$
induced by the ecological constraints. 

\me 3) We have seen that the rescaled charge capacity of the logistic
birth and death process converges to the charge capacity of the
logistic Feller diffusion. However, whether the rescaled Yaglom limit
of the logistic birth and death process converges to the Yaglom limit of the logistic
Feller diffusion process remains an open problem.
\end{remark}

\me In order to prove Theorem~\ref{thm:feller-qsd}, we firstly make a
change of variable and introduce the process $(X_t, t\geq 0)$ defined
by $X_t=2\sqrt{Z_t}$. Of course, $X$ is still absorbed at $0$ and
 QSDs for $Z$ will be easily deduced from QSDs for $X$. From now on, we focus on the process $(X_t)$.

\me An elementary computation using It\^o's formula shows that $(X_t)$
is the Kolmogorov diffusion process defined by
\begin{equation}
  \label{kolmogorov}
  d X_t= dB_t - q(X_t) dt,
\end{equation}
with 
\begin{equation*} 
  q(x)= {1\over 2 x} -{r x\over 2} + {c x^3\over 8}.
\end{equation*}
 Mention that the function $q$ is continuous on $\mathbb{R}_+^*$
 but explodes at $0$ as ${1\over 2x}$ and at infinity as ${c\over 8}
 x^3$. The strong (cubic) downward drift at infinity will force the
 process to live essentially in compact sets. That will provide the
 uniqueness of the QSD, as seen below.

\me  We introduce the measure $\mu$,
defined by
\begin{equation*}
 \label{eqmu}
 \mu(dy) = e^{-Q(y)} dy,
\end{equation*}
where $Q$ is given  by
\begin{equation}
\label{q}
 Q(y) = \int_1^y 2 q(z) dz = \ln y + {r\over 2}(1-y^2) +{c\over
   16}(y^4-1).
\end{equation}
In particular $-Q/2$ is a potential of the drift $-q$. The following
result clearly implies Theorem~\ref{theorem:qsd-logistic-Feller}.

\begin{theorem}\cite{Cattiaux2009}
  \label{theorem:feller-logistic-qsd-for-X}
  Assume that $X_0$, $r$ and $c$ are positive. Then the Yaglom limit
  $\alpha$ of the process $X$ exists.

 Moreover, there exists a positive function $\eta_1\in\mathbb{L}^2(d\mu)$ such that
   \begin{enumerate}
   \item 
     \begin{equation}
       \label{equation:logistic-feller-alpha}
       \alpha(dx)=\frac{\eta_1(x) e^{-Q(x)}}{\int_{\R_+^*} \eta_1(y) e^{-Q(y)} dy} dx,
     \end{equation}
   \item $\forall x\in\R_+^*$, $\lim_{t\rightarrow\infty}\, e^{\theta(\alpha) t} \P_x(T_0>t)=\eta_1(x)$, 
   \item there exists $\chi>0$ such that, $\forall x\in\R_+^*$,
     \begin{equation*}
       \lim_{t\rightarrow+\infty} e^{-(\chi-\theta(\alpha))
         t}\left|\P_x\left(X_t\in A | T_0>t\right)-\alpha(A)\right|<+\infty.
     \end{equation*}
     \item  the QSD $\alpha$ attracts all initial distribution, which means
  that $\alpha$ is a QLD for $X$ starting from any initial
  distribution.
  \end{enumerate}
\end{theorem}

\me
The proof of Theorem \ref{theorem:feller-logistic-qsd-for-X} will be
decomposed in the next subsections.

 \subsubsection{Spectral theory for the killed semi-group}
\bi
 As previously
we are interested in the semi-group of the killed process, that
is, for any $x>0$,  $t>0$ and  any $f\in
C_b(\mathbb{R}_+^*)$,
\be \label{sg-diffkol} P_t f(x) =
\mathbb{E}_x( f(X_t) {\bf 1}_{t< T_0} ), \ee 
with the associated  infinitesimal generator  given for $\phi
\in C^2_c((0,+\infty))$ by
\begin{equation*}
 \label{gen-fiffkol} 
 L \phi = {1\over 2} \phi'' - q \phi'. 
\end{equation*}
We are led to develop a spectral theory for this generator in
$\mathbb{L}^2(\mu)$.  Though the unity function $1$ does not belong to
$\mathbb{L}^2(\mu)$, this space is the good functional space in which
to work.  The key point we firstly show is that, starting from $x>0$,
the law of the killed process at time $t$ is absolutely continuous
with respect to $\mu$ with a density belonging to $\mathbb{L}^2(\mu)$.
The first step of the proof is a Girsanov Theorem.

\begin{proposition}
  \label{propgirsanov} 
  \label{proposition:girsanov-logistic-feller}
  For any bounded Borel function $F$ defined on $\Omega =
  C([0,t],\mathbb{R}_+^*)$ it holds
  \begin{equation*}
    \E_x \left[F(\omega) {\bf 1}_{t<T_0(\omega)}\right]
    = \E^{\W_x} \left[ F(\omega) {\bf 1}_{t<T_0(\omega)}
      \exp \left(\frac{1}{2} Q(x) - \frac{1}{2} Q(\omega_t) -
      \frac{1}{2} \int_0^t (q^2 - q')(\omega_s) ds\right)\right]
  \end{equation*}
  where $\E^{\W_x}$ denotes the expectation with respect to the Wiener
  measure starting from $x$ and $\omega$ the current point in
  $\Omega$.
\end{proposition}
\begin{proof}
It is enough to show the result for  non-negative and bounded functions $F$.
Let $\varepsilon \in (0,1)$ and $\tau_\varepsilon=T_\varepsilon
\wedge T_{1/\varepsilon}$. Let us choose some $\psi_\varepsilon$
which is a non-negative $C^\infty$ function with compact support
included in $]\varepsilon/2, 2/\varepsilon[$ such that
$\psi_\varepsilon(u)=1$ if $\varepsilon \leq u \leq
1/\varepsilon$. For all $x$ such that $\varepsilon \leq x \leq
1/\varepsilon$ the law of the diffusion \eqref{kolmogorov}
coincides up to $\tau_\varepsilon$ with the law of a similar
diffusion process $X^\varepsilon$ obtained by replacing $q$ with
the cutoff function $ q_{\varepsilon}=q\psi_\varepsilon$. For the
latter we may apply Novikov criterion (cf. \cite{Revuz1999} p.332),
ensuring that the law of $X^\varepsilon$ is given via Girsanov's
formula. Hence
{\setlength\arraycolsep{0.1em}
\begin{eqnarray*}
  \mathbb{E}_x \left[F(\omega){\bf 1}_{t<\tau_\varepsilon(\omega)}\right]
     &=& \mathbb{E}^{\W_x}\left[ F(\omega) \, {\bf 1}_{t<\tau_\varepsilon(\omega)}
           \, \exp\left(\int_0^t \, - q_\varepsilon(\omega_s) d\omega_s
           - \, \frac{1}{2}\, \int_0^t \, (q_\varepsilon)^2(\omega_s) ds\right)\right] \\
     &=& \mathbb{E}^{\W_x} \left[ F(\omega) \, {\bf 1}_{t<\tau_\varepsilon(\omega)}
           \, \exp \left(\int_0^t \, - q(\omega_s) d\omega_s
           - \, \frac 12 \, \int_0^t \, q^2(\omega_s) ds\right)\right]\\
     &=& \mathbb{E}^{\W_x} \left[ F(\omega) {\bf 1}_{t<\tau_\varepsilon(\omega)}
            \exp \left(\frac{1}{2} Q(x) - \frac{1}{2} Q(\omega_t)
           -  \frac{1}{2} \int_0^t  (q^2 - q')(\omega_s) ds\right)\right]
\end{eqnarray*}
}
integrating by parts the stochastic integral. But ${\bf
1}_{t<\tau_\varepsilon}$ is non-decreasing in $\varepsilon$ and
converges  almost surely to ${\1}_{t<T_0}$ both for
$\W_x$ and for  $\P_x$ (since
$\mathbb{P}_x(T_0<\infty)=1$)). Indeed, almost surely,
\begin{equation*}
  \lim_{\varepsilon \to 0} X_{\tau_\varepsilon}
  = \lim_{\varepsilon \to 0} X_{\tau_\varepsilon}
  = \lim_{\varepsilon \to 0} \varepsilon = 0
\end{equation*}
so that $\lim_{\varepsilon \to 0} \tau_\varepsilon \geq T_0$. But
$\tau_\varepsilon \leq T_0$ yielding the equality. It remains to use
Lebesgue monotone convergence theorem to finish the proof.
\end{proof}

\begin{theorem}
\label{thmloit}
\label{theorem:density-FellerLogistic}
  For all $x>0$ and all $t>0$ there exists a
density function $r(t,x,.)$ that satisfies
$$\mathbb{E}_x[f(X_t) \, {\bf 1}_{t<T_0}] = \int_0^{+\infty} \, f(y) \,
r(t,x,y) \, \mu(dy)$$ for all bounded Borel function $f$.
In addition, for all $t>0$ and all $x>0$,
\begin{equation*}
  \int_0^{+\infty} \, r^2(t,x,y) \, \mu(dy) \, \leq (1/2\pi t)^{\frac
    12} \, e^{Ct} \, e^{Q(x)} \, ,
\end{equation*}
where 
\begin{equation*}
  C=-\inf_{y> 0} (q^2(y)-q'(y)) <+\infty.
\end{equation*}
\end{theorem}

\begin{proof}
Define
$$G(\omega) = {\bf 1}_{t<T_0(\omega)} \,
\exp \left(\frac 12 \, Q(\omega_0) - \frac 12 \, Q(\omega_t) - \,
\frac 12 \, \int_0^t \, (q^2 - q')(\omega_s) ds\right) \, .$$
Denote by
$e^{-v(t,x,y)}= (2\pi t)^{- \frac 12} \, \exp \left( -
\frac{(x-y)^2}{2t}\right)$ the density at time $t$ of the
Brownian motion starting from $x$. According to Proposition
\ref{propgirsanov}, we have
\begin{eqnarray*}
  \mathbb{E}_x\left(f(X_t) \, {\bf 1}_{t<T_0}\right) & = &
      \mathbb{E}^{\mathbb{W}_x}\left(f(\omega_t)\
      \mathbb{E}^{\mathbb{W}_x}(G|\omega_t)\right) \\
  & = & \int_0^{+\infty} \, f(y)\, \mathbb{E}^{\mathbb{W}_x}(G|\omega_t=y) \,
      e^{-v(t,x,y)} \, dy\\
  & = & \int_0^{+\infty} \, f(y) \, \mathbb{E}^{\mathbb{W}_x}(G|\omega_t=y) \,
      e^{-v(t,x,y)+Q(y)} \, \mu(dy) \, ,
\end{eqnarray*}
because $\mathbb{E}^{\mathbb{W}_x}(G|\omega_t=y)=0$ if $y\leq 0$.
In other words, the law of $X_t$ restricted to non extinction has
a density with respect to $\mu$ given by
\begin{equation*}
r(t,x,y)=\mathbb{E}^{\mathbb{W}_x}(G|\omega_t=y) \, e^{-v(t,x,y)+Q(y)}
\, .
\end{equation*}
Hence
\begin{eqnarray*}
  \int_0^{+\infty} \, r^2(t,x,y) \, \mu(dy) & = & \int_0^{+\infty} \,
      \left(\mathbb{E}^{\mathbb{W}_x}(G|\omega_t=y) \,
      e^{-v(t,x,y)+Q(y)}\right)^2\\
      & & \quad\quad\quad \times\, e^{-Q(y)+v(t,x,y)} \, e^{-v(t,x,y)}
      \, dy \\
  & = & \mathbb{E}^{\mathbb{W}_x}
      \left(e^{-v(t,x,\omega_t)+Q(\omega_t)} \,
      \left(\mathbb{E}^{\mathbb{W}_x}(G|\omega_t)\right)^2 \right) \\
  &\leq & \mathbb{E}^{\mathbb{W}_x}
      \left(e^{-v(t,x,\omega_t)+Q(\omega_t)} \,
      \mathbb{E}^{\mathbb{W}_x}(G^2|\omega_t) \right) \\
  & \leq & e^{Q(x)} \, \mathbb{E}^{\mathbb{W}_x} \left({\bf
        1}_{t<T_0(\omega)} \, e^{-v(t,x,\omega_t)} \, e^{- \int_0^t \,
        (q^2 - q')(\omega_s) ds}\right) \, ,
\end{eqnarray*}
where we have used Cauchy-Schwarz's inequality. Since
$e^{-v(t,x,.)} \leq (1/2\pi t)^{\frac 12}$, the proof is
completed.
\end{proof}

\bigskip
\noindent Thanks to Theorem \ref{thmloit}, we can show, using the
theory of Dirichlet forms (cf. Fukushima's book
\cite{Fukushima1980}) that the infinitesimal generator $L$ of $X$,
defined by \eqref{gen-fiffkol}, can be
extended to the generator of a continuous symmetric semi-group of
contractions of $\mathbb{L}^2(\mu)$ denoted by $(P_t)_{t\geq 0}$.
In all what follows, and for $f, g \in
\mathbb{L}^2(\mu)$, we will denote 
$ \langle f, g\rangle_\mu =
\int_{\mathbb{R}_+} f(x) g(x) \mu(dx).$ The symmetry of $P_t$ means
that $\langle P_t f, g\rangle_\mu = \langle f, P_t
g\rangle_\mu.$

\noindent In Cattiaux et al. \cite{Cattiaux2009}, the following
spectral theorem in $\mathbb{L}^2(\mu)$ is proved.
\begin{theorem}
  \label{spectral}
  \label{theorem:spectral-LogisticFeller}
  \me The  operator $-L$ has a purely discrete spectrum $0<
  \lambda_1 < \lambda_2 < \cdots$. Furthermore each $\lambda_i$
  ($i\in\mathbb{N}^*$)
  is
  associated with a unique (up to a multiplicative constant)
  eigenfunction $\eta_i$ of class $C^2((0,\infty))$, which satisfies
  the ODE
  \begin{equation}
    \label{ODE} \frac12\eta_i''-q \eta_i'=-\lambda_i \eta_i.
  \end{equation}
  The sequence $(\eta_i)_{i\geq 1}$ is an orthonormal basis of
  $\mathbb{L}^2(\mu)$ and  $\eta_1(x)>0$ for all $x>0$.

  \me
  In addition, for each $i$, $\eta_i\in \mathbb{L}^1(\mu)$.
\end{theorem}

\me The proof of this theorem is based on a relation between the
Fokker-Planck operator $L$ and a Schr\"odinger operator.
 Indeed, let us set for $g\in \mathbb{L}^2(dx)$,
 \ben \tilde{P}_t g = e^{-Q/2}\ P_t(g\ e^{Q/2}).\een
 $\tilde{P}_t$ is a strongly semi-group on $\mathbb{L}^2(dx)$ with generator defined
  for $g\in C^\infty_c((0,+\infty))$ by
  \ben \tilde{L} g = {1\over 2} \triangle g - {1\over 2} (q^2- q')\ g.\een
 The spectral theory for such Schr\"odinger operator with
potential ${(q^2-q')\over 2}$ on the line (or the half-line)
  is well known    (see for example the book of
Berezin-Shubin  \cite{Berezin1991}), but the potential ${(q^2-q')\over
2}$ does not belong to $\mathbb{L}^\infty_{loc}$ as generally assumed.
Nevertheless, in our case $\inf (q^2-q')>-\infty$, which ensures the
compactness of the operators $\tilde{L}$ and $\tilde{P}_t$.

\bi The following corollary of
Theorem~\ref{theorem:spectral-LogisticFeller} is a generalization of the
Perron-Frobenius Theorem in this infinite-dimensional framework. 
 \begin{corollary}
   \label{cor-spectral} 
   \label{corollary:spectral-decomposition-logistic-feller}
   For any bounded and measurable function $f$, we have
   \begin{equation}
     \label{limit-Pt} P_t f=_{ \mathbb{L}^2(\mu)} \ \sum_{i\in\mathbb{N}^*}
     e^{-\lambda_it}\ \langle \eta_i,f\rangle_\mu \ \eta_i.
   \end{equation}
 \end{corollary}

\begin{proof}
Fix $t>0$ and let $f$ be a bounded measurable function on $\mathbb{R}_{+}^*$. Let us first prove
that $P_t f$ belongs to $\mathbb{L}^2(\mu)$. On the one hand, we have
\begin{equation*}
  \int_1^{+\infty} (P_t f(x))^2 d\mu(x)\leq \|f\|_{\infty}^2 \int_1^{\infty} e^{-Q(x)}dx <\infty.
\end{equation*}
On the other hand, by
Proposition~\ref{proposition:girsanov-logistic-feller}, we have, for
all $x\in \mathbb{R}_{+}^*$,
\begin{eqnarray*}
  P_t f(x)&\leq& \|f\|_{\infty}\, e^{\frac{1}{2}Q(x)+\frac{1}{2}Ct}\,
                 \E^{\mathbb{W}_x}\left[\mathbf{1}_{t<T_0(\omega)} e^{-\frac{1}{2} Q(\omega_t)} \right]\\
          &\leq& \|f\|_{\infty}\,e^{\frac{1}{2}Q(x)+\frac{1}{2}Ct}\,
                 \int_0^{\infty} e^{-\frac{1}{2} Q(y)}\frac{e^{-\frac{1}{2t}(y-x)^2}}{\sqrt{2\pi t}} dy.
\end{eqnarray*}
But the function
\begin{equation*}
y\mapsto e^{-\frac{1}{2} Q(y)}=\frac{1}{\sqrt{y}} e^{-\frac{r}{4}(1-y^2) -\frac{c}{32}(y^4-1) },
\end{equation*}
is integrable on $]0,+\infty[$. Since $e^{-\frac{1}{2t}(y-x)^2}\leq
1$, we deduce that there exists a constant $K_t>0$ independent of $x$ and
$f$ such that
\begin{equation*}
  P_t f(x)\leq K_t \|f\|_{\infty} e^{\frac{1}{2}Q(x)},
\end{equation*}
and thus
\begin{equation*}
  \int_0^1 (P_t f(x))^2 d\mu(x) \leq K_t^2 \|f\|_{\infty}^2.
\end{equation*}
Finally $(P_t f)^2$ is integrable with respect to $\mu$, so that $P_t f\in \mathbb{L}^2(\mu)$. 

\me Now we deduce from Theorem~\ref{theorem:spectral-LogisticFeller}
that
\begin{equation}
  \label{equationFirstStep1}
 P_t f=_{\mathbb{L}^2(\mu)} \sum_{i\in\mathbb{N}^*} \langle P_t f, \eta_i \rangle_\mu\ \eta_i
\end{equation}
If $f$ belongs to $\mathbb{L}^2(\mu)$, then the symmetry of $P_t$ implies that
\begin{eqnarray*}
  \langle P_t f, \eta_i \rangle_\mu  = \langle f, P_t\eta_i \rangle_\mu
           = e^{-\lambda_i t}\ \langle f, \eta_i\rangle_\mu.
\end{eqnarray*}
Since $\eta_i \in \mathbb{L}^1(\mu)$, we deduce from the Dominated
Convergence Theorem that the equality $\langle P_t f, \eta_i
\rangle_\mu=e^{-\lambda_i t}\ \langle f, \eta_i\rangle_\mu$ extends to
all measurable bounded functions. This and the equality
\eqref{equationFirstStep1} allow us to conclude the proof of
Corollary~\ref{corollary:spectral-decomposition-logistic-feller}.
\end{proof}

\subsubsection{Existence of the Yaglom limit}
\label{subsubsection:existence-Yaglom-limit}

By Corollary \ref{corollary:spectral-decomposition-logistic-feller},
we have for any bounded and measurable function $f$,
\begin{eqnarray*}
  \|e^{\lambda_1 t} P_t f - \langle \eta_1,
      f\rangle\eta_1\|^2_{\mathbb{L}^2(\mu)}
  &&=\sum_{i\geq 2}
      e^{-2t(\lambda_i-\lambda_1)}|\langle \eta_i,f\rangle|^2\\
  &&\leq e^{-2(t-1)(\lambda_2-\lambda_1)}\sum_{i\in\mathbb{N}^*} 
      e^{-2 (\lambda_i-\lambda_1)} |\langle \eta_i,f\rangle|^2\\
  &&\leq  e^{-2(t-1)(\lambda_2-\lambda_1)}\,e^{2\lambda_1}\,\|P_1 f\|^2_{\mathbb{L}^2(\mu)}
\end{eqnarray*}
Using Cauchy-Schwartz inequality, we deduce that, for any 
function $h\in\mathbb{L}^2(\mu)$,
\begin{equation}
  \left| e^{\lambda_1 t} \langle P_t f,h\rangle_{\mu} - \langle \eta_1,
      f\rangle \langle \eta_1, h\rangle_{\mu}\right|
  \leq e^{-2(t-1)(\lambda_2-\lambda_1)}\|P_1 f\|_{\mathbb{L}^2(\mu)} \|h\|_{\mathbb{L}^2( \mu)}.
\label{equation:attract-L2}
\end{equation}
  By Theorem
\ref{theorem:density-FellerLogistic}, $\delta_x P_1$ has the density
$r(1,x,.)\in\mathbb{L}^2(\mu)$ with respect to $\mu$, so that, taking $h=r(1,x,\cdot)$,
\begin{equation*}
  \left| e^{\lambda_1 t} P_{t+1} f(x) - \langle \eta_1,
      f\rangle \langle \eta_1, r(1,x,\cdot)\rangle_{\mu}\right|
  \leq e^{-2(t-1)(\lambda_2-\lambda_1)}\|P_1 f\|_{\mathbb{L}^2(\mu)} \|r(1,x,\cdot)\|_{\mathbb{L}^2(\mu)}.
\end{equation*}
 By definition of $\eta_1$, $\langle \eta_1,
 r(1,x,\cdot)\rangle_{\mu}=e^{-\lambda_1}\eta_1(x)$. Thus we have
 \begin{equation*}
   e^{\lambda_1 t} P_{t+1}f(x)\xrightarrow[t\rightarrow+\infty]{}
   \langle\eta_1, f\rangle_{\mu}\, e^{-\lambda_1}\eta_1(x)
 \end{equation*}
 and
 \begin{equation*}
   e^{\lambda_1 t} P_{t+1}\mathbf{1}_{\R_+^*}(x)\xrightarrow[t\rightarrow\infty]{}
   \langle\eta_1, \1_{\R_+^*}\rangle_{\mu}\, e^{-\lambda_1}\eta_1(x)
 \end{equation*}
 Finally, $\eta_1(x)$ being positive, for any $x\in\R_+^*$,
 \begin{equation*}
   \frac{P_{t}f(x)}{P_{t}\mathbf{1}_{\R_+^*}(x)}\xrightarrow[t\rightarrow+\infty]{} \frac{\langle\eta_1, f\rangle_{\mu}}{\langle\eta_1, \1_{\R_+^*}\rangle_{\mu}} = \alpha(f),
 \end{equation*}
 where $\alpha$ is defined in \eqref{equation:logistic-feller-alpha}.
 We conclude that $\alpha$ is the Yaglom limit for $Z$. We also deduce
 parts (2) and (3) of Theorem \ref{theorem:feller-logistic-qsd-for-X}.

\subsubsection{Attractiveness of any initial distribution}
\label{subsubsection:attractiveness}

Let us first consider a compactly supported probability measure $\nu$
on $(0,+\infty)$. By Theorem \ref{theorem:density-FellerLogistic},
$y\mapsto \int_{E^*} r(1,x,y) \nu(dx)$ is the density of $\nu P_1$
with respect to $\mu$.  By \cite[Lemma 5.3]{Cattiaux2009}, there
exists a locally bounded function $\Theta$ such that
\begin{equation*}
  r(1,x,y)\leq \Theta(x)\eta_1(y),\ \forall x,y\in (0,+\infty).
\end{equation*}
In particular, $h:y\mapsto \int_{E^*} r(1,x,y) \nu(dx)$ belongs to
$\mathbb{L}^2$.  Then we deduce from \eqref{equation:attract-L2} that
\begin{equation*}
\E_{\nu}\left(\left. f(X_{t+1}) \right| T_0>t+1 \right)=
\frac{\nu P_{t+1} (f)}{\nu P_{t+1}
  (\mathbf{1}_{E^*})}\xrightarrow[t\rightarrow\infty]{} \alpha(f).
\end{equation*}
We conclude that $\alpha$ attracts any compactly supported probability
measure.

\bi 

\bi Let us now prove that $\alpha$ attracts all initial distributions
$\nu$ supported in $(0,\infty)$. We want to show that, for any probability
measure $\nu$ on $\mathbb{R}_{+}^*$, for any Borel set $A$, we get
\begin{equation}
  \label{attract}
  \lim_{t\to \infty}\mathbb{P}_{\nu}(X_t\in A|T_0>t)
  = \alpha(A).
\end{equation}
This is part (4) of Theorem~\ref{theorem:feller-logistic-qsd-for-X} and it clearly
implies the uniqueness of the QSD for $X$ (and hence for $Z$).

\begin{proposition}
\label{controle-moment}
\label{proposition:controle-moment-logistic-feller}
For any $a>0$, there exists $y_a>0$ such
that $\sup_{x>y_a} \mathbb{E}_x(e^{aT_{y_a}}) <\infty$.
\end{proposition}

\begin{proof} Let us remark that
$\int_{1}^\infty e^{Q(y)} \int_y^\infty \,e^{-Q(z)} \, dz\, dy <\infty.$
Let $a>0$, and pick $x_a$ large enough so that
$
\int_{x_a}^\infty e^{Q(x)} \int_x^\infty \,e^{-Q(z)} \, dz\, dx
\le \frac{1}{2a}\, .
$
Let $J$ be the nonnegative increasing function defined on
$[x_a,\infty)$ by
$$
J(x)=\int_{x_a}^x e^{Q(y)}\int_y^\infty \, e^{-Q(z)} \, dz\, dy.
$$
Then we check that $J''= 2q J' -1$, so that $LJ = -1/2$. Set now
$y_a=1+x_a$, and consider a large $M>x$. It\^o's formula gives
\begin{eqnarray*}
\mathbb{E}_x(e^{a(t\wedge T_M\wedge T_{y_a})} \, J(X_{t\wedge
T_M\wedge T_{y_a}})) = J(x) + \mathbb{E}_x \left(\int_0^{t\wedge
T_M\wedge T_{y_a}} \, e^{as}\,(a J(X_s)+LJ(X_s)) \, ds\right).
\end{eqnarray*}
But $LJ=-1/2$, and $J(X_s)< J(\infty) \le 1/(2a)$ for any $s\leq
T_{y_a}$, so that
\begin{eqnarray*}
\mathbb{E}_x(e^{a(t\wedge T_M\wedge T_{y_a})} \, J(X_{t\wedge
T_M\wedge T_{y_a}})) \leq J(x).
\end{eqnarray*}
For $x\geq y_a$, one gets $1/(2a)>J(x)\geq J(y_a)> 0$. It follows that
$\mathbb{E}_x(e^{a(t\wedge T_M\wedge T_{y_a})}) \leq 1/(2a J(y_a))$.
Letting $M\to \infty$ then $t\to \infty$
, we deduce $\mathbb{E}_x(e^{a T_{y_a}}) \leq 1/(2a
J(y_a))$, by the monotone convergence theorem. So Proposition
\ref{controle-moment} is proved.
\end{proof}

\bi Proving that $\alpha$ attracts all initial distributions
requires the following estimates near $0$ and $\infty$.

\begin{lemme}
  \label{tension} 
  For $h\in \mathbb{L}^1(\mu)$ strictly positive on $(0,\infty)$ we
  have
  \begin{eqnarray}
    \label{tight1}
    & &\lim\limits_{\varepsilon\downarrow 0}\limsup\limits_{t\to \infty}
    \frac{\int_0^\varepsilon h(x) \mathbb{P}_x(T_0>t) \mu(dx)} {\int_0^{+\infty}
      h(x) \mathbb{P}_x(T_0>t) \mu(dx)}=0,\\
          \label{tight2} & &\lim\limits_{M\uparrow \infty}\limsup\limits_{t\to
            \infty} \frac{\int_M^\infty h(x) \mathbb{P}_x(T_0>t) \mu(dx)}
                {\int_0^{+\infty} h(x) \mathbb{P}_x(T_0>t) \mu(dx)}=0.
  \end{eqnarray}
\end{lemme}

\begin{proof}
We start with \eqref{tight1}. Using Harnack's inequality (see
\cite[Theorem~1.1]{Trudinger1968}), we have for $\varepsilon<1$ and
large $t$

\begin{equation*}
\frac{\int_0^\varepsilon h(x) \mathbb{P}_x(T_0>t) \mu(dx)} {\int_0^{+\infty}
h(x) \mathbb{P}_x(T_0>t) \mu(dx)}\le
\frac{\mathbb{P}_1(T_0>t)\int_0^\varepsilon h(z) \mu(dz)} {C
 \int_1^{3/2} h(x) \mu(dx) \mathbb{P}_1(T_0>t-1)  },
\end{equation*}
then
\begin{eqnarray*}
  \limsup\limits_{t\to \infty} \frac{\int_0^\varepsilon h(x)
    \mathbb{P}_x(T_0>t) \mu(dx)}{\int h(x) \mathbb{P}_x(T_0>t)
    \mu(dx)}
  &\le& \limsup\limits_{t\to \infty}
    \frac{\mathbb{P}_1(T_0>t)\int_0^\varepsilon h(z) \mu(dz)} {C
    \int_1^{3/2} h(x) \mu(dx) \mathbb{P}_1(T_0>t-1) }\\
  &=& \frac{e^{-\lambda_1} \int_0^\varepsilon h(z) \mu(dz)}
{C\;\int_1^{3/2} h(x) \mu(dx) },
\end{eqnarray*}
and the first assertion of the lemma is proved.

\me For the second limit, we set $A_{0}:=\sup\limits_{x\ge
  y_{\lambda_1}} \mathbb{E}_x(e^{\lambda_1
  T_{y_{\lambda_1}}})<\infty$, where $y_{\lambda_1}$ is taken from
Proposition~\ref{proposition:controle-moment-logistic-feller}. Then
for large $M>y_{\lambda_1}$, we have
\begin{equation*}
  \mathbb{P}_x(T_0>t)=\int_0^t \mathbb{P}_{x_0} (T_0>u)
  \mathbb{P}_x(T_{x_0} \in d(t-u)) + \mathbb{P}_x(T_{x_0}>t).
\end{equation*}
Using $\lim\limits_{u\to \infty} e^{\lambda_1 u}
\mathbb{P}_{x_0}(
T_0>u)=\eta_1(x_0)\langle\eta_{1},1\rangle_{\mu}$, we obtain
$  B_{0}:=\sup\limits_{u\ge 0}e^{\lambda_1 u} \mathbb{P}_{x_0}(
  T_0>u)<\infty$.
Then
\begin{eqnarray*}
    \mathbb{P}_x(T_0>t) &\le& B_{0}\int_0^t e^{-\lambda_1 u}
      \mathbb{P}_x(T_{x_0} \in d(t-u)) + \mathbb{P}_x(T_{x_0}>t)\\
    &\le& B_{0}\ e^{-\lambda_1 t}\ \mathbb{E}_x(e^{\lambda_1 T_{x_0}})+
      e^{-\lambda_1 t}\ \mathbb{E}_x(e^{\lambda_1 T_{x_0}}) \le
      e^{-\lambda_1 t}A_{0}(B_{0}+1),
\end{eqnarray*}
and \eqref{tight2} follows immediately (since $x\geq x_0\geq
y_{\lambda_1}\Rightarrow\,T_{x_0}\leq T_{y_{\lambda_1}}$).
\end{proof}

\bi Let $\nu$ be any fixed probability distribution whose support is
contained in $(0,\infty)$. We must prove \eqref{attract}. We begin by
claiming that $\nu$ can be assumed to have a strictly positive density
$h$, with respect to $\mu$. Indeed, let
$$
\ell(y)=\int_0^{+\infty} r(1,x,y) \nu(dx).
$$
Using Tonelli's theorem we have
\begin{eqnarray*}
\int_0^{+\infty}\int_0^{+\infty} r(1,x,y) \nu(dx)\, \mu(dy)&=&\int_0^{+\infty}\int_0^{+\infty} r(1,x,y) \, \mu(dy)\,
\nu(dx)\\&=&\int_0^{+\infty} \mathbb{P}_x(T_0>1) \nu(dx)\le 1,
\end{eqnarray*}
which implies that $\int r(1,x,y) \nu(dx)$ is finite $dy-$a.s..
Finally, define $h=\ell/\int \ell d\mu$. Notice that for $d\rho=h
d\mu$
$$
\mathbb{P}_\nu(X_{t+1} \in \cdot \mid T_0>t+1)=\mathbb{P}_\rho(X_t
\in \cdot \mid T_0>t),
$$
showing the claim.

\me Consider $M>\varepsilon>0$ and  any Borel set $A$ included in
$(0,\infty)$. Then
$$
\left|\frac{\int \mathbb{P}_x(X_t \in A,\, T_0>t)
h(x)\,\mu(dx)}{\int
 \mathbb{P}_x( T_0>t)  h(x)\,\mu(dx)}-
\frac{\int_\varepsilon^M  \mathbb{P}_x(X_t \in A,\, T_0>t) h(x)\,
  \mu(dx)}{\int_\varepsilon^M  \mathbb{P}_x( T_0>t)  h(x)\,\mu(dx)}\right|
$$
is bounded by the sum of the following two terms
\begin{eqnarray*}
  I1&=&\left|\frac{\int \mathbb{P}_x(X_t \in A,\, T_0>t) h(x)\,
    \mu(dx)}{\int
    \mathbb{P}_x( T_0>t) h(x)\, \mu(dx)}-
  \frac{\int_\varepsilon^M \mathbb{P}_x(X_t \in A,\, T_0>t) h(x)\,
    \mu(dx)}{\int
 \mathbb{P}_x( T_0>t)  h(x)\,\mu(dx)}\right|\\
  & &\\
  I2&=&\left|\frac{\int_\varepsilon^M \mathbb{P}_x(X_t \in A,\,
    T_0>t)
  h(x)\,\mu(dx)}{\int \mathbb{P}_x( T_0>t)  h(x)\,\mu(dx)}-
  \frac{\int_\varepsilon^M  \mathbb{P}_x(X_t \in A,\, T_0>t)  h(x)\,
    \mu(dx)}{\int_\varepsilon^M  \mathbb{P}_x( T_0>t)
h(x)\,\mu(dx)}\right|.
\end{eqnarray*}
We have the bound
$$
I1\vee I2\le \frac{\int_0^\varepsilon  \mathbb{P}_x( T_0>t) h(x)\,
 \mu(dx)+\int_M^\infty  \mathbb{P}_x( T_0>t)  h(x)\,\mu(dx)}
{\int \mathbb{P}_x( T_0>t)  h(x)\,\mu(dx)}.
$$
Thus, from Lemma \ref{tension} we get
\begin{multline*}
\lim\limits_{\varepsilon\downarrow 0, \, M\uparrow \infty}
\limsup_{t\to \infty} \left|\frac{\int  \mathbb{P}_x(X_t \in A,\,
T_0>t)
  h(x)
\mu(dx)}{\int  \mathbb{P}_x( T_0>t)  h(x)\,\mu(dx)}\right.\\
-
\left. \frac{\int_\varepsilon^M  \mathbb{P}_x(X_t \in A, T_0>t) h(x) \mu(dx)}
{\int_\varepsilon^M  \mathbb{P}_x( T_0>t)  h(x)\mu(dx)}\right|=0.
\end{multline*}
On the other hand we have
\begin{equation*}
\lim\limits_{t\to \infty} \frac{\int_\varepsilon^M
\mathbb{P}_x(X_t \in A,\,
  T_0>t)  h(x)\, \mu(dx)}
{\int_\varepsilon^M \mathbb{P}_x( T_0>t)  h(x)\,
\mu(dx)}=\frac{\int_A \eta_1(z) \mu(dz)}{\int_{\mathbb{R}^+}
\eta_1(z)
  \mu(dz)}=\alpha(A),
\end{equation*}
since $\alpha$ attracts any compactly supported probability measures
, and the result follows.

\bi

\bi

\begin{example}\upshape
We develop now a numerical illustration of this logistic Feller
diffusion case. As for the logistic birth and death process
(see Example 3, Section \ref{section:QSD-for-BD-process}), the value
of the charge capacity $\frac{r}{c}$ will remain  equal to the fixed value $9$ across the
whole example.

\me We begin by showing in Figure \ref{example4_path} a random path of
a logistic Feller diffusion process with initial size $Z_0=1$ and with
parameters $r=9$ and $c=1$ (an Euler method is used for the numerical
simulation of the random path). We observe that the process quickly attains the value of
the charge capacity and remains around it for a long time.
\begin{figure}
    \includegraphics[width=13cm]{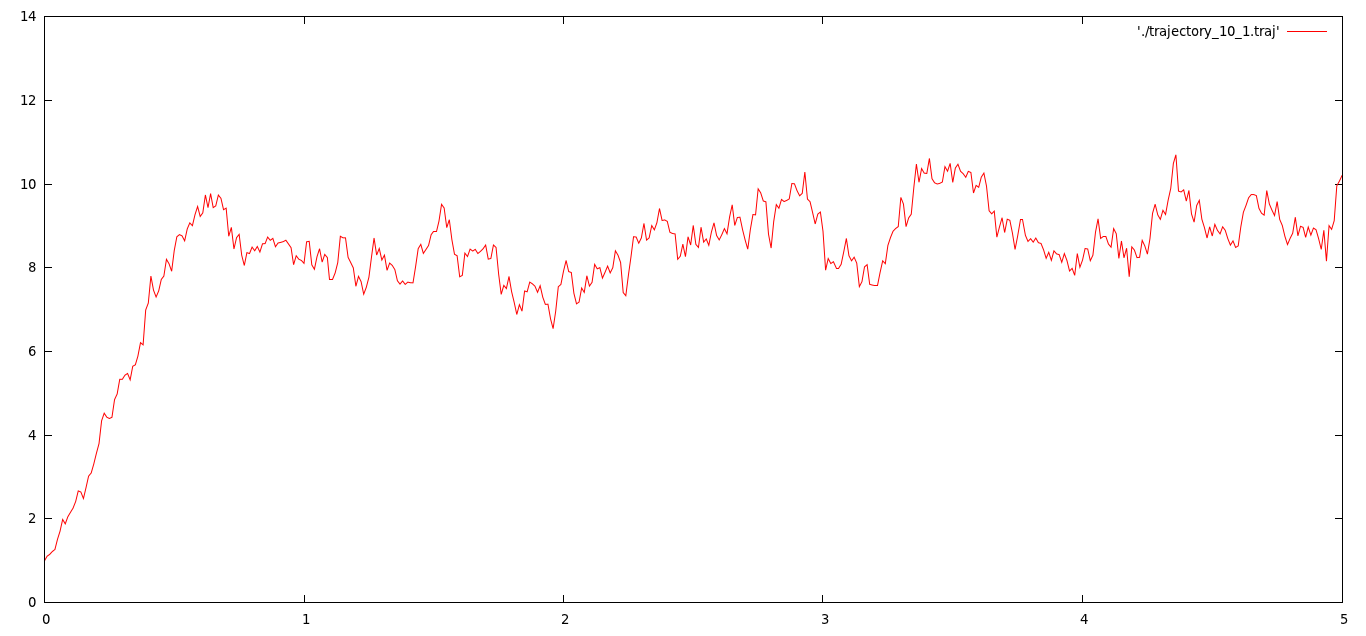}
    \caption{\label{example4_path} Example 4. A random path for
      logistic Feller diffusion process with initial size $Z_0=1$ and
      parameters $r=9$ and $c=1$}
  \end{figure}

\me We compare now the Yaglom limits of two
different logistic Feller diffusion processes whose charge capacity is
equal to $9$ (see Figure \ref{example4_yaglom}):
  \begin{itemize}
  \item[(a)] $Z^{(a)}$, whose parameters are $r=9$ and $c=1$,
  \item[(b)] $Z^{(b)}$, whose parameters are $r=3$ and $c=1/3$.
  \end{itemize}
\begin{figure}
\centering
\subfloat{\includegraphics[width=13cm]{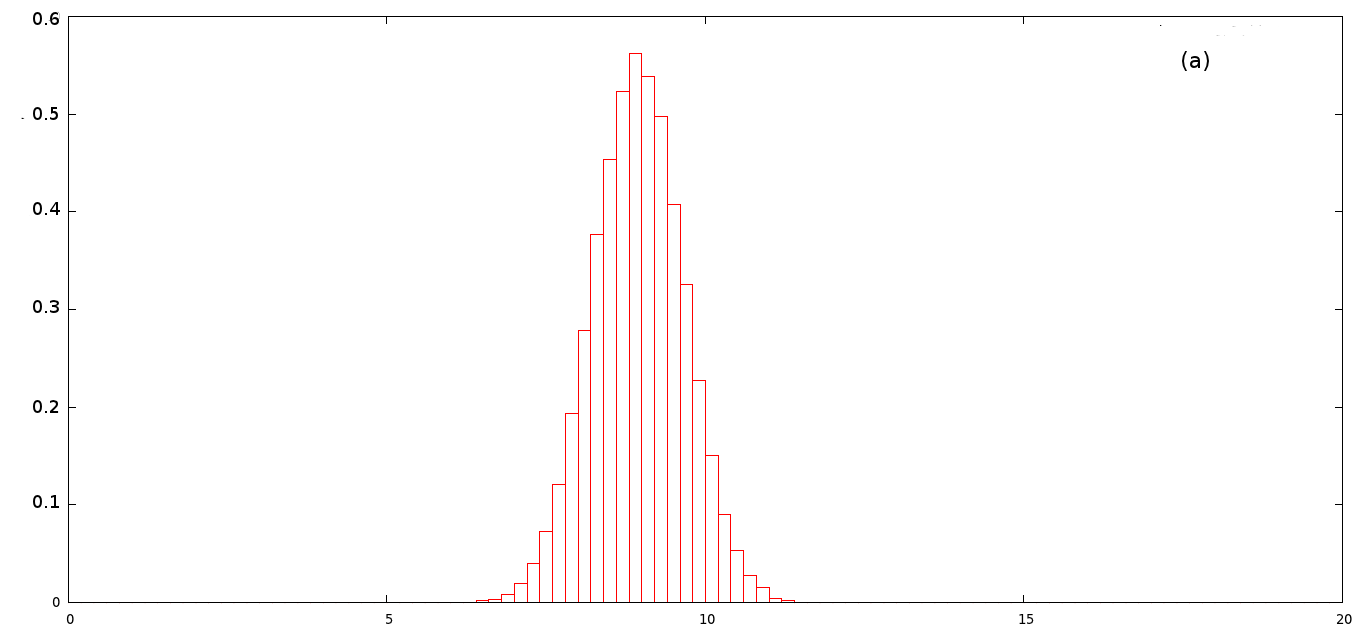}}\\
\subfloat{\includegraphics[width=13cm]{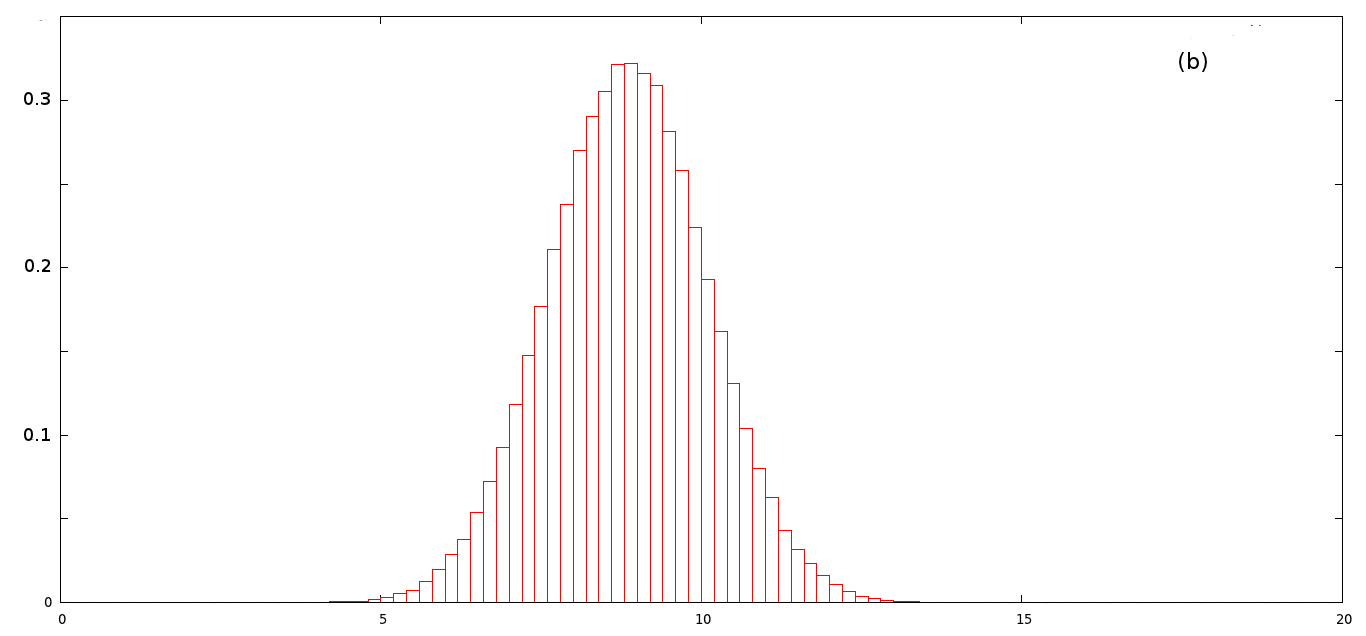}}
\caption{Example 4. The Yaglom limits of two logistic Feller
  diffusion processes with the same charge capacity: (a) $r=9$ and
  $c=1$; (b) $r=3$ and $c=1/3$.
\label{example4_yaglom}}
\end{figure}

\noindent As for the logistic BD-processes, we observe that the two
Yaglom limits are \textit{centered} around the charge capacity. But as
a consequence of the relatively \textit{weak} noise around the charge
capacity, the Yaglom limit has clearly a smaller variation around
this value in the logistic Feller diffusion case than in the logistic
BD process case. We also observe that the smaller are the parameters,
the flatter is the Yaglom limit and with a similar explanation as in
the logistic BD-process case.


\me We observe now the distance between the conditional distributions
of $Z^{(a)}$ and $Z^{(b)}$ and their respective Yaglom limits, for
different initial states, namely $Z_0=1$, $Z_0=10$ and $Z_0=100$. The
results, computed with the help of the approximation method studied
in Section~\ref{section:simulation}, are represented on
figure~\ref{figure:example4_distance}.
\begin{figure}
\centering
\subfloat{\includegraphics[width=13cm]{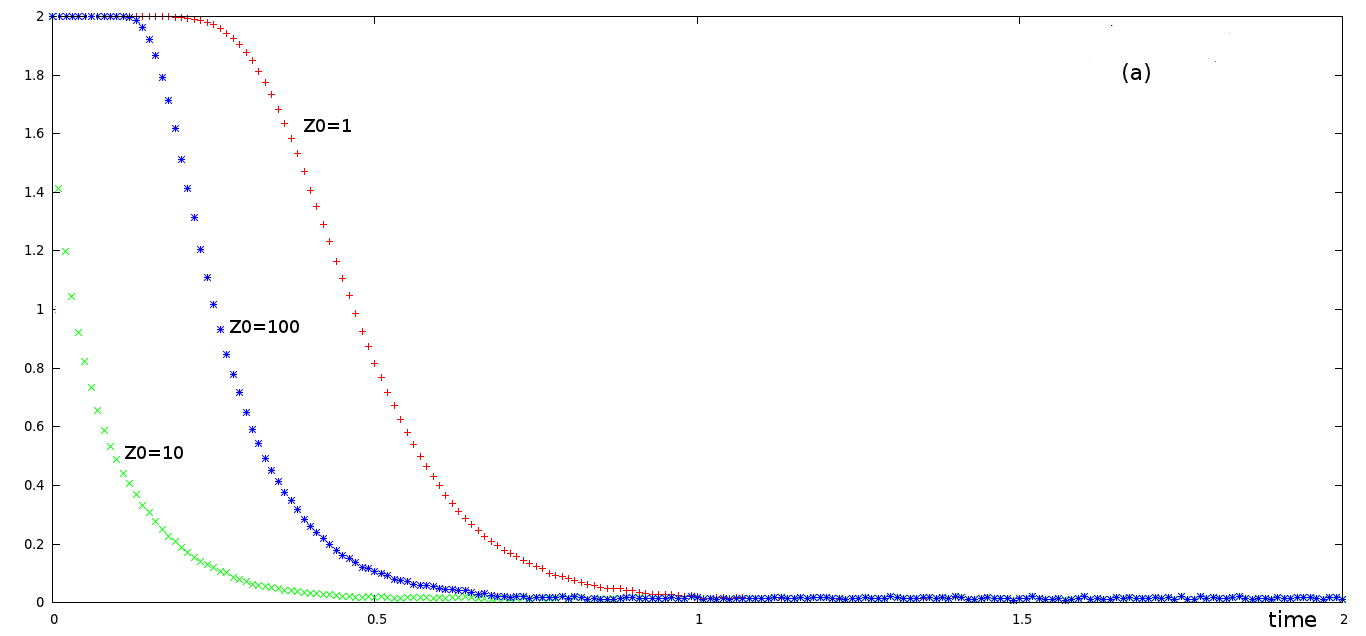}}\\
\subfloat{\includegraphics[width=13cm]{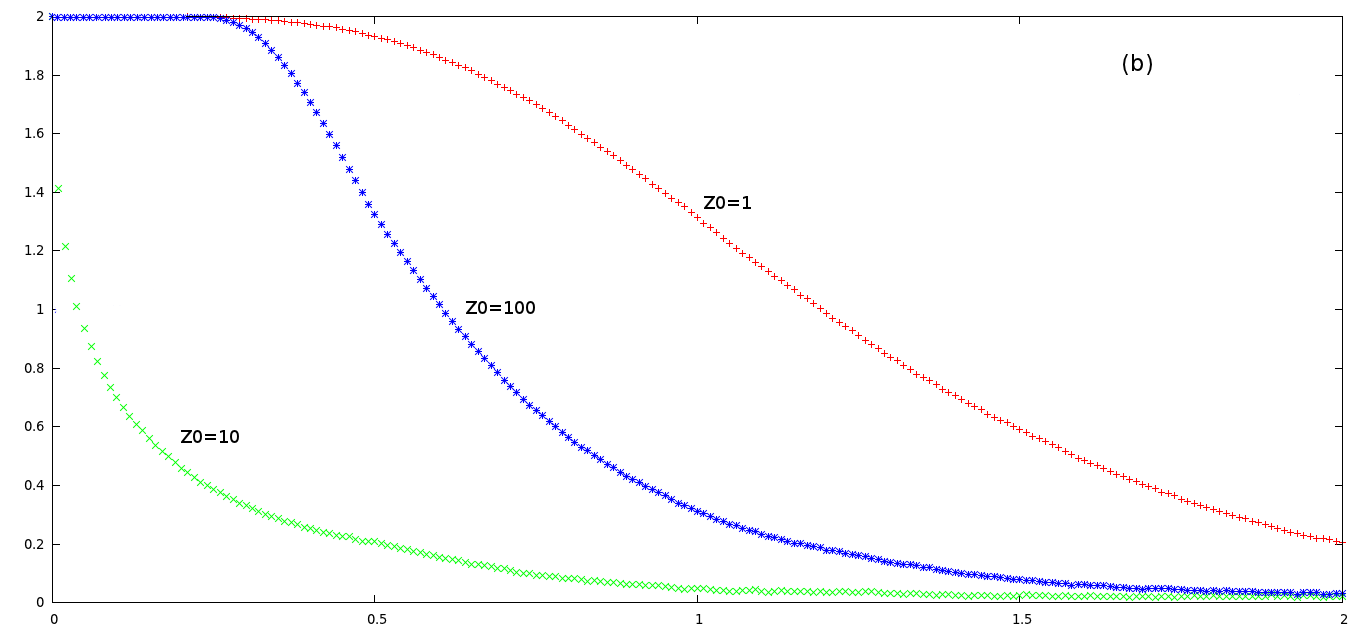}}
\caption{Example 4. Evolution of the distance between the conditioned
  distribution and the Yaglom limits of two logistic Feller diffusion
  processes with the same charge capacity $\frac{r}{c}=9$: (a) $r=9$
  and $c=1$; (b) $r=3$ and $c=1/3$.
\label{figure:example4_distance}}
\end{figure}
 For both $Z^{(a)}$ and $Z^{(b)}$, the speed of convergence to the
 Yaglom limit is the highest for $Z_0=10$, which is quite intuitive
 since the value of the charge capacity is $9$. We also observe that
 it is higher for the processes starting from $100$ than for the
 processes starting from $1$. In particular, this behavior is
 different than in the logistic birth and death process case.

\end{example}

\subsection{The $Q$-process}
\label{subsection:Q-process-logistic-Feller}

Let us now describe the law of the trajectories conditioned to
never attain $0$.

\begin{theorem}\cite{Cattiaux2009}
\label{Q-process}
\label{theorem:Q-process-FellerLogistic}
Let us fix a time $s$ and consider $B$ a
measurable subset of $C([0,s],\mathbb{R}_+)$. Then for any $x\in
\mathbb{R}_+^*$,
 \begin{equation*}
   \lim_{t\to \infty} \mathbb{P}_x(X\in B| t<T_0)=Q_x(X\in B),
 \end{equation*}
 where $Q_x$ is the law of a continuous process with transition
 probabilities given by $q(s,x,y)dy$, with
 \begin{equation*}
 q(s,x,y)=e^{\lambda_1 s}\ {\eta_1(y)\over \eta_1(x)}\ r(s,x,y)\
 e^{-Q(y)}.
 \end{equation*}
\end{theorem}

\begin{proof}
Since  $\ r(s,x,y)\
 e^{-Q(y)} dy\ $ is the law of $X_s$ started from $x$ before extinction, we have to prove that
 \ben
 Q_x(X\in B) = e^{\lambda_1 s}\mathbb{E}_x\left({\bf 1}_B(X)\ {\eta_1(X_s)\over
 \eta_1(x)}\ {\bf 1}_{T_0>s}\right).\een
 For $t>s$,
 \ben
{\mathbb{P}_x(X\in B ; T_0>t)\over \mathbb{P}_x(T_0>t)} =
{\mathbb{P}_x(X\in B ; T_0>s ; \mathbb{E}_{X_s}(T_0>t-s) )\over
\mathbb{P}_x(T_0>t)},\een
and we have proved that
 \ben \lim_{t\to \infty} {\mathbb{P}_y(T_0> t-s)\over
 \mathbb{P}_x(T_0>t)} = e^{\lambda_1 s}\ {\eta_1(y)\over
 \eta_1(x)}.\een
 Then,
\ben \lim_{t\to \infty} {\mathbb{P}_x(X\in B ; T_0>t)\over
\mathbb{P}_x(T_0>t)} = {e^{\lambda_1 s}\over
 \eta_1(x)}\ \mathbb{P}_x\left({\bf 1}_B(X)\ {\eta_1(X_s)}\ {\bf 1}_{T_0>s}\right).\een
\end{proof}

\begin{corollary}
\label{corjavaitort} For any Borel  set $A\subset (0,\infty)$ and
any $x$, \ben \lim_{s\to \infty} Q_x(X_s\in A)= \int_A \eta_1^2(y)
\mu(dy) = <\eta_1, 1>_\mu  \int_A \eta_1(y) \alpha(dy).\een
\end{corollary}

\begin{proof}

 Since ${\bf 1}_A\, 
\eta_1 \in \mathbb{L}^2(\mu)$, thus
$$
\eta_1(x)\ Q_x(X_s\in A)=\int {\bf 1}_A(y)\ \eta_1(y) \,
e^{\lambda_1 s}\ r(s,x,y)\ \mu(dy)$$ converges to $\eta_1(x)
\, \int_B \eta^2_1(y) \mu(dy)$ as $s \to +\infty$, since
$e^{\lambda_1 s} r(s,x,.)$ converges to $\eta_1(x) \, \eta_1(.)$
in $\mathbb{L}^2(d\mu)$.
\end{proof}

\begin{remark}\upshape
The stationary measure of the $Q$-process is absolutely continuous
with respect to $\alpha$, with Radon-Nikodym derivative $<\eta_1,
1>_\mu \, \eta_1$.
\end{remark}

\subsection{The case of a multi-type population}
\label{subsection:multi-type-population}

Until now, we have considered a population where all individuals have
the same ecological parameters. This biological assumption corresponds
to the case where individuals have the same type.  In this section, we
generalize the previous study to a population composed of $k$
different types. The population size process describing the dynamics
of each subpopulation is given by a $k$-dimensional stochastic
Lotka-Volterra process $Z=(Z^1_t,\cdots, Z^k_t)_{t\geq 0}$ (SLVP),
which describes the size of a $k$-types density dependent
population. This model generalizes to $k$ types the $2$-types density
dependent model introduced by Cattiaux and M\'el\'eard \cite{Cattiaux2008}.

\me More precisely, we consider for $i,j\in\{1,\cdots,k\}$ the
coefficients
\begin{equation*}
  \label{coeffts} 
  \gamma_i > 0\ ,\ r_i>0\ ;\ c_{ij} >0,\ \forall i,j\in\{1,\cdots,k\}.
\end{equation*}
The process $Z$ takes its values in $(\mathbb{R}_+)^k$ and is
solution of the stochastic differential system 
\begin{equation}
  \label{SLVP} 
  dZ^i_t=\sqrt{\gamma_i Z^i_t}dB^i_t + (r_i Z^i_t - \sum_{{j=1}}^k c_{ij} Z^i_t Z^j_t)\ dt,
\end{equation}
where $(B^i)_{i=1,\cdots, k}$ are independent standard Brownian motions
independent of the initial data $Z_0$. The system \eqref{SLVP} can be obtained as \eqref{eq:Fellerlogistic} as 
 approximation of  renormalized $k$-types birth and death processes in
 case of large population and small life lengths and reproduction
 times. The coefficients $r_i$ are the asymptotic growth rates of
 $i$-type's populations. The positive coefficients $\gamma_i$ can be
 interpreted as demographic parameters describing the ecological
 timescale. The coefficient $c_{ij}$, for $i,j =1,\cdots,k$,
 represents the pressure felt by an individual holding type $i$ from
 an individual with type $j$.  Intra-specific competition is modeled by the
 rates $c_{ii}$, while  inter-specific competition is described
 by the coefficients $c_{ij}>0, i\neq j$.
 If $c_{ij}=0$ for all $i\neq j$, the stochastic $k$-dimensional process reduces to $k$ independent Feller logistic diffusion processes.  Extinction of the
population is modeled by the absorbing state $(0,\cdots, 0)$ and the
extinction of the subpopulation of type $i$ is modeled by the
absorbing set $$H_{i}= (\mathbb{R}_+^*)^{i-1}\times \{0\}\times
(\mathbb{R}_+^*)^{k-i}.$$

 \bi 

\me We denote by $D$ the open subset of $\mathbb{R}^k $ defined by
$D=(\mathbb{R}_{+}^* )^k$ and by $\partial {D}$ its boundary.  We
denote by $T_{0}$ the first hitting time of $(0,\cdots,0)$, by $T_{A}$
the first hitting time of some subset $A$ and thus by $T_{\partial D}$
the exit time of $D$. Of course, some of these stopping times are
comparable. For example if the initial condition belongs to $D$, \be
\label{hittingtime} T_{\partial D}\leq T_{H_{i}}\leq  T_0, \quad \forall i=1,\cdots, k.
\ee On the other hand, $T_{H_{i}}$ and $T_{H_{j}}$ are not directly
comparable for $i\neq j$.

\bi Let us prove the existence of the SLVP .
\begin{proposition} \label{exi-compet} The process $(Z_t)_t$ is well defined
on $\mathbb{R}_+$. In addition, for all $x\in (\mathbb{R}_+)^k$,
$$\mathbb{P}_x(T_0<+\infty)=1$$
and there exists $\lambda>0$ such that
$$\sup_{x\in (\mathbb{R}_+)^k}\mathbb{E}_x(e^{\lambda T_0})
<+\infty.$$
\end{proposition}

\begin{proof}
The existence of the SLVP is  shown by a comparison
argument (cf. Ikeda-Watanabe \cite{Ikeda1989} Chapter 6 Thm 1.1). Indeed, the 
coordinates $(Z^i_t)_t$ can be upper-bounded by the independent solutions of 
logistic Feller equations 
\begin{equation} 
  dY^i_t=\sqrt{\gamma_iY^i_t}dB^i_t + (r_i
  Y^i_t - c_{ii} (Y^i_t)^2 )\ dt,\label{dom}
\end{equation}
for which we have obtained in the previous section that 
  extinction occurs a.s. in
finite time and that the extinction time has some finite exponential moments. 
  The almost sure  finiteness of each $T_{H_{i}}$, hence of
$T_{\partial D}$ and $T_{0}$, thus follows.  
\end{proof}

\me As in the previous sections we are interested in the quasi-stationary
distributions for the process \eqref{SLVP}.  We firstly reduce the
problem by a change of variable. Let us define $X_{t}=(X^1_t,\cdots,
X^k_t) $ with $X^i_{t} = 2\sqrt{{Z^i_{t}\over \gamma_{i}}}$. We obtain
via It\^o's formula and for any $i\in\{ 1,\cdots,k\}$,

\begin{eqnarray}
  \label{eqfell3}
  dX^i_t & = & dB_t^i \, + \, \left(\frac {r_i X_t^i}{2} \, - \,
  \sum_{j=1}^k \frac{c_{ij} \gamma_j \,X_t^i (X_t^j)^2}{8}\, -
  \frac{1}{2 X_t^i}\right) .
\end{eqnarray}

\noindent In the following, we will focus on the symmetric case where $X$ is a
Kolmogorov diffusion, that is a Brownian motion with a drift in
gradient form as 
\begin{equation} 
  \label{kolmo}
  dX_t = dB_t \, - \, \nabla V
  (X_t) dt.
\end{equation}
Let us state  a necessary and sufficient condition
to write the drift of $(X)$ as in \eqref{kolmo}.  The proof is obtained by computation and requires the
equality of the second order cross-derivatives of $V$.

\begin{proposition}
If the following balance conditions on the ecological parameters are
satisfied,
\begin{equation}
  \label{balance}
  c_{ij} \gamma_{j} = c_{ji}\gamma_{i}, \quad \forall i, j,
\end{equation}
then the process $X$ is a Kolmogorov process with potential $V$ given by
$$V(x^1,\cdots,x^k)= {1\over 2} \, \sum_{i=1}^k \left(\ln(x^i) +
{c_{ii}\gamma_{i}(x^i)^4\over 16} - {r_{i}(x^i)^2\over 2}\right) +
\sum_{i\neq j}c_{ij} \gamma_{j}\, (x^i)^2\, (x^j)^2.$$
\end{proposition}

\bi We will establish an existence and uniqueness result for the QSD
of the process $(X_t)$. The re-statement of the results for the
initial stochastic Lotka-Volterra process follows immediately, since
the hitting time of $(0,\cdots,0)$ and $H_{i}$ and the exit time of
$D$ are the same for both processes $(X)$ and $(Z)$.

\me By generalizing to $k$-types populations the results proved in
\cite{Cattiaux2008} for two-types populations (an easy consequence of
Girsanov's theorem), we get
\begin{proposition}
\label{sortie}
For all $x\in D$, for all $i\neq j$,
$$\mathbb{P}_{x}(T_{\partial D} = T_{H_{i}\cap H_{j}})=0.$$
\end{proposition}

\noindent
Let us now state the first theorem, which is concerned by conditioning
on the co-existence of the $k$ types.

\begin{theorem}
  \label{QSDD}
  Under the balance conditions \eqref{balance}, there exists a unique
  quasi-stationary distribution $\nu$ for the process $(X)$ and the
  absorbing set $\partial D$, which is the quasi-limiting distribution
  starting from any initial distribution: for any $\mu$ on $D$ and any $A\subset D$,
  $$\lim_{t\to +\infty} \,\mathbb{P}_{\mu}(X_{t} \in A |
  T_{\partial D} >t)= \nu(A).
  $$
Furthermore, there exist $\lambda>0$ and a positive function $\eta$ such that   $$\lim_{t\to +\infty} e^{\lambda t}\,\mathbb{P}_{x}(X_{t} \in A ;
  T_{\partial D} >t)= \eta(x)\,\nu(A).
  $$

\end{theorem}

\noindent
\begin{proof}
The proof of the existence of a quasi-stationary distribution results
from the spectral theory for the semi-group of the killed process
$(P_{t})$ (related to $X$) established in Cattiaux-M\'el\'eard
\cite{Cattiaux2008}, Appendices A, B, C. Define the
reference measure on $(\mathbb{R}_{+})^k$ by
$$\mu(dx_{1},\cdots,dx_{k}) = e^{- 2 V(x)} \, dx_{1}\cdots dx_{k}.$$
As in Subsection 5.2.2, one builds a self-adjoint
operator on $\mathbb{L}^2(\mu)$ which coincides with $P_{t}$ for
bounded functions belonging to $\mathbb{L}^2(\mu)$.  Its generator $L$
is self-adjoint on $\mathbb{L}^2(\mu)$ and $$L g = {1\over 2} \Delta g
- V\cdot \nabla g, \quad \forall g\in C_{0}^\infty(D).$$ We  check
that the assumptions required in \cite{Cattiaux2008} Theorem A.4
are satisfied and therefore, the operator $-L$ is proved to have a
purely discrete spectrum of non-negative eigenvalues and the smallest
one $\lambda$ is positive. The corresponding eigenfunction $\eta$ is
proved to be in $\mathbb{L}^1(\mu)$ and the probability measure $\nu=
\frac{\eta \, d\mu}{\int_{D}\eta \, d\mu} $ is the Yaglom limit.

\noindent let us emphasize that the uniqueness of the quasi-stationary
distribution results by \cite{Cattiaux2008} Proposition B.12 from the
ultracontractivity of the semi-group $P_{t}$ (ultracontractivity means
that $P_{t}$ maps continuously $\mathbb{L}^2(\mu)$ in
$\mathbb{L}^\infty(\mu)$ for any $t>0$). The proof of the latter is
easily generalized from the two-types case (\cite{Cattiaux2008}
Proposition B.14) to the $k$-types case.

\end{proof}

\noindent Theorem \ref{QSDD} shows that in some cases, a stabilization
of the process with co-existence of the $k$ types will occur before
one of these types disappears.  Let us now come back to our initial
question: the long-time behavior of the process conditioned on
non-extinction. For each $ i= 1,\ldots,k$, we denote by $\lambda_{i}$
the smallest eigenvalue related to the purely discrete spectrum of the
generator for the $i$-axis diffusion defined by the stochastic
differential equation \eqref{dom}.

\begin{theorem}
\label{theorem:QSD-multi-type}
Under the balance conditions \eqref{balance}, there exists a Yaglom limit $m$ for the
process $(X)$ conditioned on non extinction: for any $x\neq 0$, 
for any $A\subset D$,
  $$\lim_{t\to +\infty}  \,\mathbb{P}_{x}(X_{t} \in A |
  T_{0} >t)= m(A).
  $$
The support of this
measure is included in the $k$ axes.

\noindent Furthermore, if there exist $i_{1},..., i_{l} \in \{1,\cdots, k\}$
such that $\lambda^{i_{1}}=\cdots=\lambda^{i_{l}} < \min_{i\neq
  i_{1},...,i_{l}} \lambda^{i} $, then this QSD is concentrated on the axes of
coordinates $i_1,..., i_l$.
\end{theorem}


\begin{proof}

\me Recall that the existence of a Yaglom limit has been proved in the case $k=1$ (Section~\ref{subsection:qsd-for-logistic-Feller-diffusions}).
 In what follows, we prove by induction the existence of a
Yaglom limit for any $k$-type system \eqref{eqfell3}.

\me The induction assumption $(A_{k-1})$ is as follows: we assume
that, for any $(k-1)$-type Kolmogorov process $X^{(k-1)}$ satisfying \eqref{eqfell3} with \eqref{balance}, there exist a constant $\lambda>0$, a uniformly bounded function
$\eta>0$ and a probability measure $\nu$ on $(\R_+)^{k-1}$ such
that, for any $x\in(\R_+)^{k-1}\setminus\{0\}$ and any bounded measurable function
$f$ on $(\R^+)^{k-1}$ such that $f(0)=0$, we have
\begin{equation}
\label{equation:induction-assumption}
\begin{split}
&\lim_{t\to \infty}e^{\lambda t}\,\mathbb{E}_{x}(f(X^{(k-1)}_t)) = \eta(x) \nu(f);\\
&\sup_{t\geq 0,\ x\in (\R_{+}^*)^{k-1}} |e^{\lambda t}\,\E_x(f(X^{(k-1)}_t))|<+\infty.
\end{split}
\end{equation}

\me As mentioned above,  
Assumption $(A_{1})$ is already proved.  Let us  assume that $(A_{k-1})$ is
true and  show that $(A_{k})$ follows.

\me Let  $X^{(k)}$ be a $k$-type Kolmogorov process  satisfying \eqref{eqfell3} with \eqref{balance}.
Once hitting the boundary $\partial D = \cup_{i=1}^k H_{i}$, the
process will no more leave it. Hence, for $t\geq T_{\partial D}$, the
process will stay on the union of hyperplanes $H_{i}$. Moreover,
$T_{\partial D} = \inf_{i=1,\cdots, k} T_{H_{i}}$.  Fix
$i\in\{1,\cdots, k\}$ and assume that the process leaves $D$ through
$H_{i}$. The dynamics on $H_{i}$ is given by the process
$(U^{(i),j}_{t})_{j\neq i}$ defined in $(\mathbb{R}_{+})^{k-1}$ by:
\begin{eqnarray}
\label{hyperplan}
dU^{(i),j}_{t}  & = & dB_t^j \, + \, \left(\frac {r_j U_t^{(i),j}}{2} \, - \, \sum_{\ell=1, \ell\neq i}^k
\frac{c_{j\ell} \gamma_j \,U_t^{(i),j} (U_t^{(i),\ell})^2}{8}\, -
\frac{1}{2 U_t^{(i),j}}\right) . \nonumber
\end{eqnarray}
Remark that by Proposition \ref{sortie}, the process really leaves
$\partial D$ by the interior of $H_{i}$.  
Each system $(U^{(i),j})_{j\neq i}$ is a $(k-1)$-type kolmogorov process  \eqref{eqfell3} with balance conditions.  Hence, by our
induction assumption ($A_{k-1}$), there exist for each
$i\in\{1,\cdots, k\}$ a positive constant $v_i$, a positive
function $\eta_i$ and a probability measure $\nu_i$ on $H_i$ such
that~\eqref{equation:induction-assumption} holds for
$(U^{(i),j})_{j\neq i}$, $i\in\{1,\cdots,k\}$.

\me Let us define
\begin{equation*}
  v_{min}=\inf_{i\in\{1,\cdots,k\}} v_i.
\end{equation*}
For any bounded measurable function $f$ on $(\R_+)^k$ such that $f(0)=0$ and for all $t\geq 0$, we have
\begin{equation}
\label{equation:decomposition-1}
  e^{v_{min}t}\E_x\left(f(X^{(k)}_t)\right)
  =e^{v_{min} t}\E_x\left(f(X^{(k)}_t)\1_{T_{\partial D}>t}\right)
  +\sum_{i=1}^k \E_x\left(e^{v_{min} t} f(X^{(k)}_t)\1_{T_{\partial D}=T_{H_i}\leq t}\right),
\end{equation}
where we used the fact that $X^{(k)}$ reaches $\partial D$ by
hitting the interior of one and only one $H_i$. By Theorem~\ref{QSDD},
there exist a positive constant $\lambda'$, a positive function $\eta'$
and a probability measure $\nu'$ on $(\R_+^*)^k$ such that
\begin{equation*}
\lim_{t\rightarrow+\infty} e^{\lambda' t}
\E_x\left(f(X^{(k)}_t)\1_{T_{\partial D}>t}\right)=\eta'(x)\nu'(f).
\end{equation*}
Moreover, a similar coupling argument as in~\cite{Cattiaux2008} yields
 $\lambda'>v_{min}$. We deduce that
\begin{equation*}
\lim_{t\rightarrow\infty}e^{v_{min} t}\E_x\left(f(X^{(k)}_t)\1_{T_{\partial D}>t}\right)=0.
\end{equation*}
For each $i\in\{1,\cdots, k\}$, we have by the Markov property
\begin{align}
\E_x\left(e^{v_{min} t} f(X^{(k)}_t)\right.&\left.\1_{T_{\partial
    D}=T_{H_i}\leq t}\right) = \E_x\left( e^{v_{min}
  t}\,\1_{T_{\partial D}=T_{H_i}\leq t}\, \E_{X_{T_{\partial D}}}\left(f(U^{(i)}_{t-T_{\partial D}})\right)\right)\nonumber\\
     &= \E_x\left( e^{v_{min} T_{\partial D}}\,\1_{T_{\partial D}=T_{H_i}\leq
  t}\, \E_{X_{T_{\partial D}}}\left(e^{v_{min}(t-T_{\partial D})}f(U^{(i)}_{t-T_{\partial
    D}})\right)\right).\label{equation:Markov-1}
\end{align}
By the induction assumption $(A_{k-1})$,
$e^{v_{min}(t-T_{\partial D})}f(U^{(i)}_{t-T_{\partial D}})$ is
uniformly bounded. Moreover the inequality $0<v_i<\lambda'$ and
Proposition~\ref{proposition:moment-for-extinction-time} ensure that
$\E_x\left( e^{v_{min} T_{\partial D}}\right)<+\infty$. Using
the convergence property of the induction assumption $(A_{k-1})$ and
the dominated convergence theorem, we deduce that
\begin{equation*}
\lim_{t\rightarrow \infty} \E_x\left(e^{v_{min} t}
f(X^{(k)}_t)\1_{T_{\partial D}=T_{H_i}\leq t}\right) =
\left\lbrace
\begin{array}{l}
\E_x\left(
e^{v_{min} T_{\partial D}}\1_{T_{\partial D}=T_{H_i}}
\eta_i(X_{T_{\partial D}})\right)\nu_i(f),\ \text{if}\ v_i=v_{min}\\
0,\ \text{otherwise}.
\end{array}
\right.
\end{equation*}
We have then
\begin{equation*}
\lim_{t\rightarrow \infty} e^{v_{min} t} \E_x(f(X^{(k)}_t))=\sum_{i=1}^k \1_{v_i=v_{min}}
\E_x\left(e^{v_{min} T_{\partial D}}\1_{T_{\partial D}=T_{H_i}}\eta_i(X_{T_{\partial D}})\right)\nu_i(f),
\end{equation*}
which gives us the first part of the induction assumption $(A_k)$.

%
%
\bi In order to prove the second part of $(A_k)$, let us introduce the SLVP $Y^{(k)}$ with coefficients $(c'_{ij})$ defined by
\begin{equation*}
  c'_{kk}=c_{kk},\ c'_{ij}=c_{ij}\ \text{ and } c'_{ki}=c'_{ik}=0,\ \forall i,j=1,\cdots,k-1.
\end{equation*}
By the same coupling argument as above, the return time to $\partial
D$ for $X^{(k)}$ is stochastically dominated by the return time to
$\partial D$ for $Y^{(k)}$, \textit{i.e.} $\P_x(X^{(k)}_t\in D)\leq
\P_x(Y^{(k)}_t\in D)$ for all $t\geq 0$.

\me Since the $k-1$ first components of $Y^{(k)}$ are independent of the
last one and since 
$$
\{Y^{(k)}_t\in D\}=\{(Y^{(k),1}_t,\cdots,Y^{(k),k-1}_t)\in (\R_+^*)^{k-1}\}\cap\{Y^{(k),k}_t\in \R_+^*\}, 
$$
we have
$$
\P_x(Y^{(k)}_t\in D)\leq \P_x((Y^{(k),1}_t,\cdots,Y^{(k),k-1}_t)\in (\R_+^*)^{k-1})\times \P_x(Y^{(k),k}_t\in \R_+^*).
$$
On the one hand, the dynamic of $(Y^{(k),1},\cdots,Y^{(k),k-1})$ is the same as
$U^{(k)}$, so that, by the second part of the induction assumption
$(A_{k-1})$ and by the definition of $v_{min}$,
$$
\sup_{t\geq 0, x\in D} e^{v_{min} t} \P_x((Y^{(k),1}_t,\cdots,Y^{(k),k-1}_t)\in (\R_+^*)^{k-1})<+\infty.
$$
On the other hand, $Y^{(k),k}$ is a one dimensional SLVP, thus we
deduce from $(A_1)$ that there exists a positive constant $\lambda_1$
such that
$$
\sup_{t\geq 0, x\in D} e^{\lambda_1 t} \P_x(Y^{(k),k}_t\in \R_+^*) <+\infty.
$$
As a consequence, we have
$$
\sup_{t\geq 0, x\in D} e^{(v_{min}+\lambda_1) t} \P_x(X^{(k)}_t\in D)
\leq \sup_{t\geq 0, x\in D} e^{(v_{min}+\lambda_1) t}
\P_x(Y^{(k)}_t\in D)<+\infty
$$
and we deduce that
$$
\sup_{x\in D} \E_x(e^{v_{min}T_{\partial D}})<+\infty.
$$
For any bounded measurable function $f$, this immediately leads us
to $$\sup_{t\geq 0, x\in E} \E_x(e^{v_{min}t}\1_{t<T_{\partial
    D}}f(X^{(k)}_t))<+\infty.$$  Moreover, by
Equality~\eqref{equation:Markov-1} and the second part of $(A_{k-1})$,
we deduce that, for each $i\in\{1,\cdots,k\}$,
\begin{equation*}
 \sup_{t\geq 0, x\in E} \E_x\left(e^{v_{min} t} f(X^{(k)}_t)\1_{T_{\partial
    D}=T_{H_i}\leq t}\right)<+\infty.
\end{equation*}
By Equality~\eqref{equation:decomposition-1}, the second part of the
induction assumption $(A_k)$ is thus proved.

\me By induction on $k\geq 1$, we conclude that Assumption $(A_k)$ is
true for any $k\geq 1$, thus Theorem~\ref{theorem:QSD-multi-type} follows.
\end{proof}

\begin{example}\upshape
  Let us numerically  study  a $3$-type system and
  observe its long-time behavior. The  $3$-tuple process $(Z^1, Z^2, Z^3)$ evolves as  
    \begin{eqnarray*}
      dZ^1_t=\sqrt{\gamma_1 Z_t^1}dB^1_t+\left( r_1 Z^1_t - c_{11}(Z^1_t)^2-c_{12}Z^1_t Z^2_t - c_{13}Z^1_t Z^3_t \right) dt,\\
      dZ^2_t=\sqrt{\gamma_2 Z_t^2}dB^2_t+\left( r_2 Z^2_t -c_{21}Z^1_t Z^2_t - c_{22}(Z^2_t)^2 - c_{23}Z^2_t Z^3_t\right) dt,\\
      dZ^3_t=\sqrt{\gamma_3 Z_t^3}dB^3_t+\left( r_3 Z^3_t -c_{31}Z^1_t Z^3_t - c_{32}Z^2_t Z^3_t - c_{33}(Z^3_t)^2 \right) dt,
  \end{eqnarray*}
with
  \begin{eqnarray*}
    \gamma_{i}=1,\ c_{ii}=10,\quad  \forall i\in\{1, 2, 3\}\quad\text{ and }\quad
     c_{ij}=0.5, \ \forall \ i\neq j\in\{1, 2, 3\},
  \end{eqnarray*}
  and
  \begin{equation*}
    r_1=1.5,\ r_2=1,\ r_3=0.5\ ;\ Z^1_0=Z^2_0=Z^3_0=1.
  \end{equation*}


  \me We describe the dynamics of  $\P_{(1,1,1)}((Z^1_t,Z^2_t,Z^3_t)\in\cdot\ |\ T_0>t)$. As explained above,
  the process conditioned on non-extinction initially behaves as a $3$-type
  population. Then a type goes extinct, then a second one and finally it
  only remains one type in the population. In order to represent
  graphically these transitions, we compute numerically the dynamics
  of the probabilities of coexistence and existence of the different
  types as  functions of  time. In
  Figure~\ref{figure:example6_abc}, we represent
\begin{itemize}
\item[(a)] the probability of coexistence of the three types
  $\P_{(1,1,1)}(Z^1_t>0,\,Z^2_t>0,\,Z^3_t>0\ |\ T_0>t)$;
\item[(b)] the probability
  $\P_{(1,1,1)}(Z^i_t>0,\,Z^j_t>0,\,Z^k_t=0\ |\ T_0>t)$ of coexistence
  of exactly two types $i\neq j$, for each combination of types
  $(i,j,k)=(1,2,3)$, $(i,j,k)=(2,3,1)$ and $(i,j,k)=(1,3,2)$;
\item[(c)] the probability
  $\P_{(1,1,1)}(Z^i_t>0,\,Z^j_t=0,\,Z^k_t=0\ |\ T_0>t)$ of existence
  of one and only one type $i$, for each type $i=1,\,2$ and $3$.
\end{itemize}
\begin{figure}
\begin{center}
\includegraphics[width=15cm]{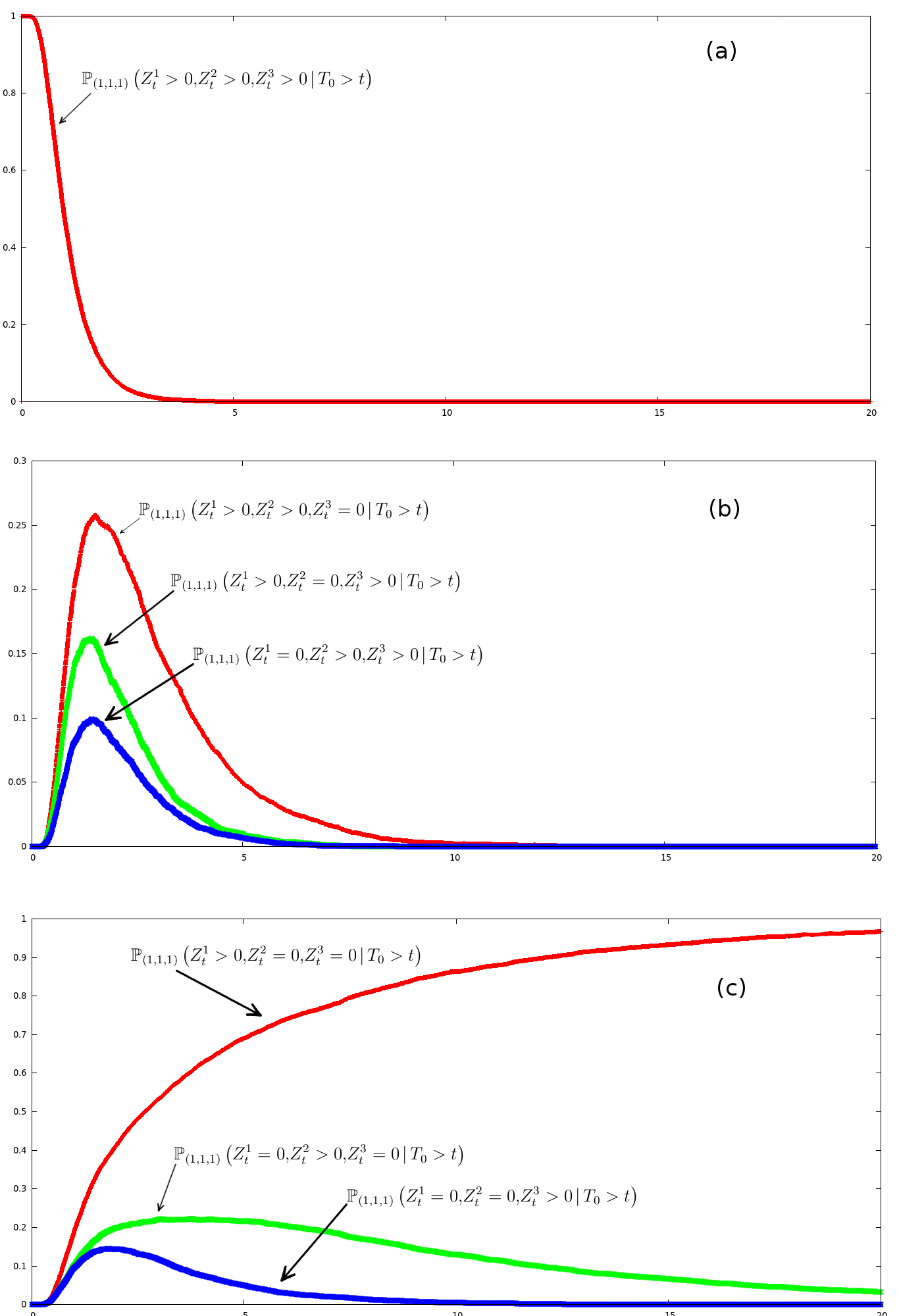}
\caption{\label{figure:example6_abc} Dynamics of the probabilities of
  co-existence and existence of the different types for a $3$-type
  stochastic Lotka-Volterra system. The horizontal axis is the time axis}
\end{center}
\end{figure}

\me As expected, the $3$-type mode disappears quickly and the $2$-type
modes are transient. We also observe that the probability
$\P_{(1,1,1)}(Z^1_t>0,\,Z^2_t=0,\,Z^3_t=0\ |\ T_0>t)$ 
converges to $1$ when $t$ increases, meaning that the last state of the population before extinction is  monotype with type $1$.
 It turns out that the support of the conditional
law $\P_{(1,1,1)}((Z^1_t,Z^2_t,Z^3_t)\in\cdot\ |\ T_0>t)$ becomes more
and more concentrated on $\mathbb{R}_+^*\times\{0\}\times\{0\}$ in the
long time. The Yaglom limit is thus equal to
$\nu_1\otimes\delta_{(0,0)}$, where $\nu_1$ is the Yaglom limit of the
process 
  \begin{equation*}
    dZ'^1_t=\sqrt{Z'^1_t}dB^1_t+\left( r_1 Z'^1_t - c_{11}(Z'^1_t)^2\right) dt,
  \end{equation*}
 absorbed at $0$ and  is
  represented in Figure~\ref{figure:example6_QSD}.
\begin{figure}
\begin{center}
\includegraphics[width=16cm]{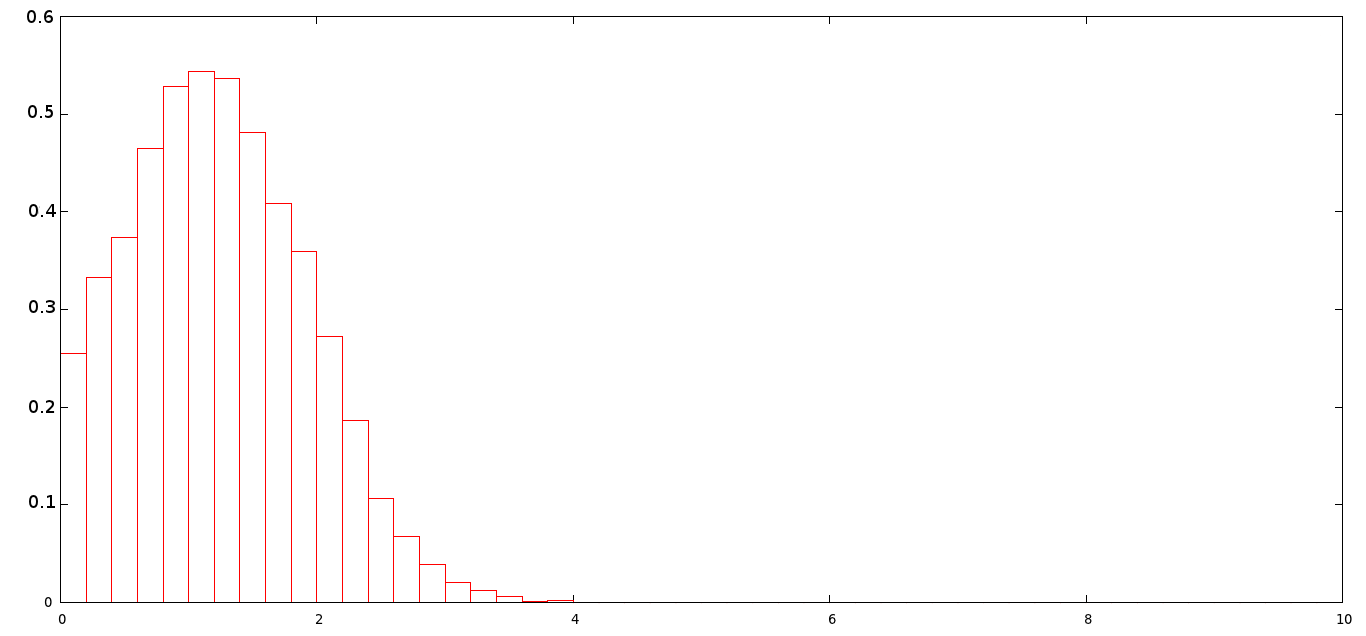}
\caption{\label{figure:example6_QSD} First marginal of the Yaglom
  limit of a $3$-type stochastic Lotka-Volterra system. The two other
  marginals are equal to the null measure.}
\end{center}
\end{figure}
\end{example}

\section{Simulation: the Fleming-Viot system}
\label{section:simulation}

 As seen in the previous sections, the spectral theory is a
 powerful tool to prove  existence and eventually  uniqueness of
 a QSD for a given process $Z$. It is based on the equivalence
 property of Proposition \ref{proposition:qsd-spectral-point-of-view},
 stating that a probability measure $\alpha$ on $E^*$ is a QSD
 for the killed process $Z$ if and only if
 \begin{equation}
   \label{equation:simulation-spectral-point-of-view}
   \alpha L=-\theta(\alpha) \alpha,
 \end{equation}
 where $L$ denotes the infinitesimal generator of $Z$ and
 $\theta(\alpha)$ a positive constant.  In some cases, such as in
 the finite state space case, one can easily compute numerically the whole
 set of eigenvalues and eigenvectors of $L$ as seen in Example 1 and
 Example 2. For these numerical illustrations, we used the software
 {\small SCILAB} and its function {\verb spec }. We also refer to
 \cite{vanDoorn2011} for a detailed description of some algorithms
 available in {\small MATLAB} for the computation of eigenfunctions
 and eigenvalues in large (but finite) state space cases.

 \me In other cases, such as the logistic birth and death process of
 Section \ref{section:QSD-for-BD-process} and the logistic Feller
 diffusion of Section
 \ref{section:the-logistic-feller-diffusion-process}, solving
 numerically Equation
 \eqref{equation:simulation-spectral-point-of-view} is too hard and
 we use a different approach. This approach consists in approximating
 the QSD and the conditioned distribution $\P_z(Z_t\in .|t<T_0)$ by
 the empirical distribution of a simulable interacting particle
 system. This Fleming-Viot type system, built for any number of
 particles $N\geq 2$, has been introduced by Burdzy, Holyst and March
 \cite{Burdzy1996} and explored in \cite{Burdzy2000} and in
 Grigorescu-Kang \cite{Grigorescu2004} for $d$-dimensional killed
 Brownian motions. It has also been studied in Villemonais \cite{Villemonais2010}
 for multi-dimensional diffusion processes with unbounded drifts and a
 general result is available in \cite{Villemonais2011}. Similar
 systems have also been considered by Ferrari-Mari\`c \cite{Ferrari2007}
 for continuous Markov chains in a countable state space. In this
 section, we explain the approximation method based on the
 Fleming-Viot type interacting particle systems.

 \bi Let $Z$ be a killed Markov process which evolves in the state
 space $E$. Fix $N\geq 2$ and let $Z_0\in E$ be its initial value. The interacting particle system with $N$ particles
 $(Z^1,\cdots,Z^N)$ starts from $(Z_0,\cdots,Z_0)$ and belongs to $\left(E^*\right)^N$.  The particles evolve independently from
 this initial position according to the law of the
 killed Markov process $Z$, until one of them hits the state $0$. At
 that time $\tau_1$, the killed particle jumps to the position at
 $\tau_1$ of one of the $N-1$ remaining particles, chosen uniformly
 among them. Then the particles evolve independently according to the law of $Z$ until one of them
 attains $0$ (time $\tau_2$), and so on. The sequence of jumps is
 denoted by $(\tau_n)_n$ and we set
 \begin{equation*}
   \tau_{\infty}=\lim_{n\rightarrow \infty} \tau_n.
 \end{equation*}
 This procedure defines the $(E^*)^N$-valued process $(Z^1, \cdots,
 Z^N)$ for all time ${t\in [0,\tau_{\infty}[}$. 
     Figure~\ref{figure:FV-2-particules} shows an illustration of
     such a system with two particles evolving between their jumps as
     Markov processes absorbed in $0$ and $1$.
     \begin{figure}
       \begin{center}
         \includegraphics[width=16cm]{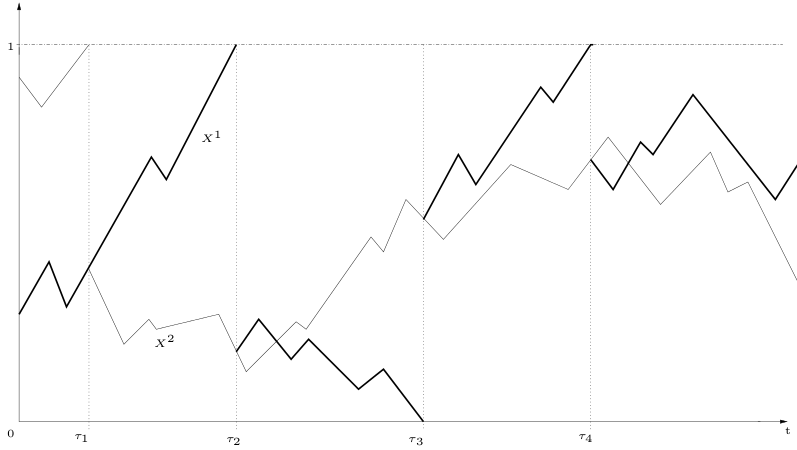}
         \caption{\label{figure:FV-2-particules} Fleming-Viot type
           system with two particles absorbed in $0$ and $1$.}
       \end{center}
     \end{figure}

 \me If $\tau_{\infty}=+\infty$ almost surely, then the
 Fleming-Viot particle system will be well defined at all time $t>0$.
 The condition $\tau_{\infty}=+\infty$ is clearly fulfilled for
 continuous time Markov chains with bounded jump rates. In the
 diffusion process case, criteria have been provided in
 \cite{Bieniek2009}, \cite{Grigorescu2011}, \cite{Villemonais2010}
 and~\cite{Villemonais2011}.

 \me   In
 that case,  denote by $\mu^N_t$ the empirical
 distribution of $(Z^1,\cdots,Z^N)$ at time $t$:
 \begin{equation*}
   \mu^N_t=\frac{1}{N}\sum_{i=1}^N \delta_{Z^i_t},\quad\forall t\geq 0.
 \end{equation*}
 The following result is obtained in \cite{Villemonais2011} by martingale method.
 \begin{theorem}
   \label{theorem:convergence-FV-to-conditioned-distribution}
   Assume that for all $N\geq 2$, $(Z^1,\cdots, Z^N)$ is well defined
   at any time $t\geq 0$. Then, for any time $t>0$, the sequence of
   empirical distributions $(\mu^N_t)$ converges in law to the
   conditioned distribution $\P_{Z_0}\left(Z_t\in \cdot | t<
   T_0\right)$, when $N$ goes to infinity.
 \end{theorem}

 \noindent If moreover $(Z^1,\cdots, Z^N)$ is ergodic, we denote by
 $M^N$ its stationary distribution and by ${\cal X}^N$ its empirical
 stationary distribution, which is defined by $\, {\cal
   X}^N=\frac{1}{N}\sum_{i=1}^N\delta_{z_i}$, where
 $(z_1,\cdots,z_N)\in E^*$ is a random vector distributed with respect
 to $M^N$. In particular, $\mu^N_t$ converges in law to ${\cal X}^N$
 when $t\rightarrow\infty$. We refer to \cite{Villemonais2010} for the
 proof of the following theorem.
 \begin{theorem}
   \label{theorem:convergence-FV-stationary-to-Yaglom-limit}
   Assume that $Z$ has a QLD $\alpha$ which attracts all initial
   distributions: for any probability measure $\mu$
   on $E^*$,
   \begin{equation*}
     \lim_{t\rightarrow+\infty}\P_{\mu}\left(Z_t\in \cdot|t< T_0\right)=\alpha.
   \end{equation*}
   Assume moreover that $(Z^1,\cdots,Z^N)$ is ergodic and that the
   family of laws of $({\cal X}^N)_{N\geq 2}$ is uniformly tight. Then
   the sequence of random probability measures $({\cal X}^N)$
   converges weakly to $\alpha$.
 \end{theorem}

 \noindent 
 If $E$ is a bounded subset of $\mathbb{R}^d$, $d\geq 1$, and if $Z$
 is a drifted Brownian motion with bounded drift which is killed at
 the boundaries of $E$, then the assumptions of Theorems
 \ref{theorem:convergence-FV-to-conditioned-distribution} and
 \ref{theorem:convergence-FV-stationary-to-Yaglom-limit} are fulfilled
 (see \cite{Villemonais2010}). The proofs of  Theorems
 \ref{theorem:convergence-FV-to-conditioned-distribution} and
 \ref{theorem:convergence-FV-stationary-to-Yaglom-limit} are based on a
 coupling argument. More general (but longer) proofs can also be found
 in \cite{Grigorescu2011} or \cite{Villemonais2011}. In particular,
 these results provide us a numerical approximation method of the
 Yaglom limit for such processes.

 \bi Let us now consider  the  Kolmogorov diffusion process $X$ defined in \eqref{kolmogorov}. In that case,
 the existence of the Fleming-Viot particle system remains an open
 problem because of the unboundedness of the drift coefficient. In order to avoid this difficulty, we introduce
 the law $\mathbb{P}^{\varepsilon}$ of the diffusion process with
 bounded coefficients defined by
 \begin{equation}
   \label{killedkol}
   dX^{\varepsilon}_t= dB_t - q(X^{\varepsilon}_t) dt\ ;\ X_0\in (\varepsilon, {1/
     \varepsilon}),
 \end{equation}
 killed when it hits $\varepsilon$ or ${1\over \varepsilon}$.  One can
 easily show that at any time $t\geq 0$, the conditioned distribution
 of $X^{\varepsilon}$ converges to the one of $X$:
 \begin{equation*}
   \mathbb{P}^{\varepsilon}\left(X^{\varepsilon}_t\in\cdot|t<T_{\varepsilon}\wedge
   T_{1/\varepsilon}\right)
   \xrightarrow[\varepsilon\rightarrow 0]{} \mathbb{P}\left(X_t\in\cdot|t<T_{0}\right).
 \end{equation*}
 The existence of the Yaglom limit denoted by $\alpha_{\epsilon}$ and the uniqueness of the QSD for
 $\P^{\varepsilon}$ are obtained from Pinsky \cite{Pinsky1985}.  The following approximation result
 is proved in \cite{Villemonais2010}  using a compactness-uniqueness argument.
 \begin{proposition}
   The sequence $(\alpha^\varepsilon)_\varepsilon$ weakly converges to
   the Yaglom limit $\alpha$ of $X$ as $\varepsilon$ tends to $0$.
 \end{proposition}

 \me For all $N\geq 2$, we denote by
 $(X^{\varepsilon,1},\cdots, X^{\varepsilon,N})$ the interacting
 particle system built as above, with the law $\P^{\epsilon}$. Since
 the diffusion process $X^{\epsilon}$ is a drifted Brownian motion
 with bounded drift evolving in the bounded interval
 $]\epsilon,1/\epsilon[$, the interacting particle system
     $(X^{\varepsilon,1},\cdots, X^{\varepsilon,N})$ fulfills the
     assumptions of Theorems
     \ref{theorem:convergence-FV-to-conditioned-distribution} and
     \ref{theorem:convergence-FV-stationary-to-Yaglom-limit}.
     Denoting by $\mu^{\epsilon, N}$ the empirical
 distribution of the simulable particle system
 $(X^{\varepsilon,1},\cdots , X^{\varepsilon,N})$, we get
 \begin{equation*}
   \lim_{\epsilon\rightarrow 0} \lim_{N\rightarrow\infty} \mu^{\varepsilon,N}_t=\P_{X_0}\left(X_t\in\cdot|t<T_0\right),\ \forall t\geq 0,
 \end{equation*}
 and
 \begin{equation*}
   \lim_{\epsilon\rightarrow 0} \lim_{N\rightarrow\infty}
   \lim_{t\rightarrow\infty}
   \mu^{\varepsilon,N}_t=\lim_{\epsilon\rightarrow 0} \alpha_{\varepsilon}=\alpha.
 \end{equation*}
 Then, choosing $\epsilon$ small enough and $N$ big enough, we get a
 numerical approximation method for the conditioned distribution and
 the Yaglom limit of $X$.

 \begin{example}\upshape
   Let us now  develop this simulation method in the case of the
   Wright-Fisher diffusion conditioned to be absorbed at $0$, which
   evolves
   in $[0,1[$ and is defined by
   \begin{equation*}
     dZ_t=\sqrt{Z_t(1-Z_t)}dB_t-Z_t dt,\ Z_0=z\in]0,1[.
   \end{equation*}
    This is a
   model for a bi-type population in which the second type cannot
   disappear. In that model, $Z_t$ is the proportion of the first type
   in the population at time $t\geq0$ and $1-Z_t$ the proportion of
   the other one. 
   The existence of a Yaglom limit for this process has been
   proved by Huillet in \cite{Huillet2007}, which also proved that it
   has the density $\,2-2x\,$ with respect to the Lebesgue measure. 

   \me Using the approximation method described above with
   $\epsilon=0.001$ and $N=10000$, we obtain numerically the density
   of the Yaglom limit for $Z$ represented in
   Figure~\ref{figure:wright-fisher-yaglom}, which is very close to
   the function $x\mapsto 2-2x$ and shows the efficiency of the method.
   \begin{figure}[htbp]
     \begin{center}
       \includegraphics[width=13cm]{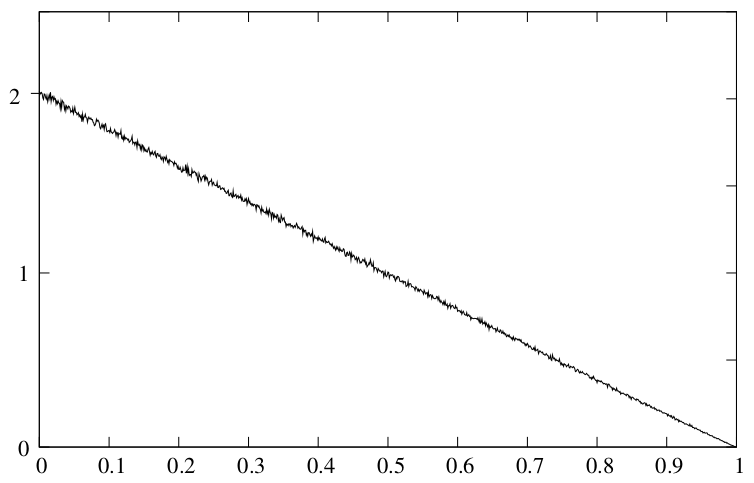}
       \caption{The Yaglom limit of the Wright-Fisher diffusion
         conditioned to be absorbed at $0$ obtained by numerical
         simulation.}
       \label{figure:wright-fisher-yaglom}
     \end{center}
   \end{figure}

 \end{example}


\end{document}